\documentclass[11pt,letterpaper]{amsart}
\usepackage{color}
\usepackage{amsmath,amssymb,amsfonts,amsthm,bm}
\usepackage{graphicx}
\usepackage{float}
\usepackage{enumerate}
\usepackage{appendix}
  \usepackage[shortlabels]{enumitem}
\usepackage[numbers, sort&compress]{natbib}
 \usepackage[left=1.35in, right=1.35in]{geometry}
\def\R{\mathbb R}
\def\Z{\mathbb Z}
\def\T{\mathbb T}
\def\p{\mathbb P}
\def\q{\mathbb Q}

\def\Q{\mathbb Q}
\def\N{\mathbb N}

\def\cC{\mathcal C}
\def\cE{\mathcal E}
\def\d{\partial}

\def\t{\dot}

\def\B{\dot{B}}
\def\H{\dot{H}}
\def\W{\dot{W}}

\renewcommand\div{{\rm div}\,}
\renewcommand\lim{{\rm lim}\,}
\renewcommand\exp{{\rm exp}\,}
\renewcommand\sup{{\rm sup}\,}
\renewcommand\inf{{\rm inf}\,}
\renewcommand\log{{\rm log}\,}

\newcommand{\with}{\quad\hbox{with}\quad}
\newcommand{\andf}{\quad\hbox{and}\quad}
\newcommand{\Sum}{\displaystyle \sum}

\newcommand{\Int}{\displaystyle \int}

\def\dr{\delta\!\rho}
\def\du{\delta\!u}
\def\dP{\delta\!P}

\newtheorem{theorem}{Theorem}[section]
 \newtheorem{corollary}[theorem]{Corollary}
 
 \newtheorem{proposition}[theorem]{Proposition}
 \theoremstyle{definition}
 \newtheorem{definition}[theorem]{Definition}
 \theoremstyle{remark}
 \newtheorem{remark}[theorem]{Remark}
 
 \numberwithin{equation}{section}

\newcommand{\wh}{\widehat}
\newcommand{\wt}{\widetilde}
\newcommand{\eps}{\varepsilon}

\newcommand{\abs}[1]{\left\vert#1\right\vert}

\newcommand{\norm}[1]{\Vert#1\Vert}

\begin{document}
\title[Incompressible  Navier-Stokes equation with variable density]
{Global unique solutions for  the inhomogeneous Navier-Stokes equation with only bounded
density,  in critical regularity spaces}
\author{ Rapha\"el Danchin}
\author{Shan Wang}
\begin{abstract}
We here aim at proving the global existence and uniqueness of solutions to the inhomogeneous incompressible Navier-Stokes system in the case where 
\emph{the initial density $\rho_0$  is discontinuous and the initial velocity $u_0$ 
has critical  regularity.}

Assuming  that  $\rho_0$  is close to a positive constant, 
we obtain global existence and uniqueness  in the two-dimensional case whenever the initial velocity $u_0$ belongs
to  the critical homogeneous Besov space $\B^{-1+2/p}_{p,1}(\R^{2})$ $(1<p<2)$ and, in the three-dimensional case, if  $u_0$  is small 
in   $\B^{-1+3/p}_{p,1}(\R^{3})$ $(1<p<3).$ 

Next, still in a critical functional framework, 
we establish  a uniqueness  statement that is valid 
in the case of large variations of density with, possibly, 
vacuum.  Interestingly, our result implies that
the Fujita-Kato type solutions constructed by  P. Zhang in \cite{Zhang19}  
are unique.

Our work relies on interpolation results, time weighted estimates 
and  maximal  regularity estimates  in Lorentz spaces (with respect to the time variable) for the 
evolutionary Stokes system.
\end{abstract}

\date{}

\maketitle
\section*{Introduction}
We are concerned with the 
initial value problem for the following inhomogeneous incompressible Navier-Stokes system: 
 \begin{equation*}
\left\{\begin{aligned}
&\rho_{t}+u\cdot \nabla \rho=0, \\
 &\rho u_{t}+\rho u\cdot \nabla u-\mu\Delta u+\nabla P=0, \\
&\div u=0,\\
&(\rho,u)|_{t=0}=(\rho_{0},u_{0}),
\end{aligned}\right.\eqno(INS)
\end{equation*}
where $\rho=\rho(t,x)\geq0,$ $P=P(t,x)\in\R$ and $u=u(t,x)\in\R^d$ 
stand for the density, pressure  and velocity field of the fluid, respectively. 
We consider  the evolution for positive times $t$ in the case where the space variable 
$x$ belongs to the whole space $\R^d$ 
%or to the periodic box $\T^d$ 
with $d=2,3.$
%In what follows,  we will use the common notation $\A^d$ to  designate $\R^d$ or $\T^d.$
\medbreak
It has long been observed that smooth enough solutions obey the following energy balance: 
\begin{equation}\label{eq:energy}\frac 12\norm{\sqrt{\rho(t)}u(t)}^2_{L_2}+\int_0^t \norm{\nabla u(t)}^2_{L_2} \,d\tau = \frac 12 \norm{\sqrt{\rho_0}u_0}^2_{L_2},\end{equation}
and that, as a consequence of the divergence free property of the velocity field, 
 the Lebesgue measure of 
\begin{equation}\label{eq:density}\bigl\{ x\in \R^d : \alpha \leq \rho(t,x)\leq \beta\bigr\}\end{equation}
is independent of $t,$  for any $0\leq\alpha\leq\beta.$
\medbreak
In 1974, by combining these relations with Galerkin approximation and
compactness arguments,  Kazhikhov \cite{AVK} established that for any data
$(\rho_0,u_0)$ such that $\rho_0 \in L_\infty,$ $\div u_0=0$ and $\sqrt{\rho_0}u_0 \in L_2,$
and \emph{provided $\rho_{0}$ is  bounded away from vacuum} (that is $\inf\rho_0(x)>0$), 
(INS) has at least one global distributional solution satisfying \eqref{eq:energy} with an inequality. 
The no vacuum assumption was removed later by  J. Simon in \cite{JS}, 
then, by taking advantage of the theory developed in \cite{DL1989}, 
P.-L. Lions \cite{PL} extended the previous results to the case of a density dependent viscosity, 
proved that the mass equation of (INS) is satisfied in the renormalized meaning, 
that the velocity field admits a unique generalized flow and, finally, 
that \eqref{eq:density} is true. 
However, from that time whether these  weak solutions are unique is an open question, even in  dimension two.
\smallbreak
By using totally different approaches, a number of authors proved that
in the case of smooth enough data, (INS) admits a unique solution
at least locally in time.  In fact, as for the classical incompressible
Navier-Stokes equations (that is (INS) with constant positive density), 
the general picture is that provided the initial density is 
sufficiently smooth, bounded and bounded away from zero, 
there exists a global unique solution if the initial velocity 
is small in the sense of some `critical norm', and that it can be arbitrarily large
in dimension $d=2.$ This general fact has been 
first observed by 
O.~Ladyzhenskaya and V.~Solonnikov \cite{OV2} in the case where
the fluid domain  $\Omega$ is  a smooth bounded subset   of $\R^d$
($d=2,3$) and the velocity vanishes at the boundary. 
More precisely, assuming that $u_0$ is  in
the Sobolev–Slobodeckij space $W^{2-\frac 2p,p}(\Omega)$  with $p>d,$ is divergence free 
and has null trace  on $\d \Omega,$ and  that $\rho_0$ is  $C^1$ and  is bounded away from zero, they proved:  
\begin{itemize}
 \item[--] the global well-posedness in dimension $d=2$,\smallbreak
 \item[--] the local well-posedness in dimension $d=3$
 (and global well-posedness if $u_0$ is small in $W^{2-\frac 2p,p}(\Omega)$).
 \end{itemize}
 Results in the same spirit in other functional frameworks have been proved by a number of authors (see e.g. the survey paper \cite{RD}). 
  Still for smooth enough data,   
 the non vacuum assumption has been weakened  by Choe and Kim in \cite{CK}
 \medbreak
 A natural question  is the minimal regularity requirement for the data ensuring 
 (at least local) existence and uniqueness. It has been observed by Fujita and Kato \cite{FK}
 in the constant density case (and later for a number of evolutionary equations) that this 
 issue is closely linked to the scaling invariance of the system under consideration. 
Here it is obvious that 
  if $(\rho, u, P)$ is a solution of (INS) on $\R_{+}\times \R^{d}$ for data $(\rho_0,u_0)$
   then, for all $\lambda>0,$ the rescaled triplet $(\rho, u, P)\to (\rho_{\lambda}, u_{\lambda}, P_{\lambda})$ defined by
\begin{equation}\label{isi}
(\rho_{\lambda}, u_{\lambda}, P_{\lambda})\overset{\text{def}}{=}(\rho(\lambda^{2}t, \lambda x), \lambda u(\lambda^{2}t, \lambda x), \lambda^{2}P(\lambda^{2}t, \lambda x))
\end{equation}
is  a solution of (INS) on $\R_{+}\times \R^{d},$ with data $(\rho_0(\lambda\cdot),\lambda u_0(\lambda\cdot)).$ 
 \medbreak
 A number of works have been dedicated to the  well-posedness of (INS) in $\R^d,$
 in so-called critical framework, that is to say in functional spaces  with the above scaling invariance. 
Restricting our attention to the case where the density tends to some positive constant 
at infinity (say $1$ for notational simplicity) and setting 
$a\overset{\text{def}}{=}{1}/{\rho}-1,$  System (INS) rewrites in terms of $(a,u,P)$ as follows:
\begin{equation}\label{rins2}
\left\{\begin{aligned}
&a_t+u\cdot\nabla a=0, \\
&u_t+u\cdot\nabla u-(1+a)(\mu\Delta u-\nabla P)=0,\\
&\div u=0,\\
&(a,u)|_{t=0}=(a_{0},u_{0}).
\end{aligned}\right.
\end{equation}
In \cite{DR2003}, the first author established the existence and uniqueness of a solution to \eqref{rins2} in 
critical Besov spaces. More precisely,  in the case where
$a_{0}\in \B^{d/2}_{2,1}(\R^d)$ and $u_{0}\in \B^{d/2-1}_{2,1}(\R^d)$ 
with $\div u_{0}=0,$ he proved that
there exists a constant $c$  depending  only on $d$ such that, if  
\begin{equation*}
 \norm{a_{0}}_{\B^{1}_{2,1}}
\leq c,\end{equation*}
then \eqref{rins2} admits a unique  local solution $(a,u,\nabla P)$ with 
$$a\!\in\! C_{b}([0,T);\B^{d/2}_{2,1}),\quad\!\!
u\!\in\! C_{b}([0,T);\B^{d/2}_{2,1})\cap L_{1}(0,T;\B^{d/2+1}_{2,1})
\!\andf\! 
 \nabla P\!\in\! L_{1}(0,T;\B^{d/2-1}_{2,1})$$
and that there exists $c'>0$ such that 
this solution is global (i.e. one can take $T=\infty$)
if $\norm{u_{0}}_{\B^{d/2-1}_{2,1}}\leq c'\mu.$
\smallbreak
Shortly after, these results have been extended by H. Abidi in \cite{AH}, then H. Abidi and M. Paicu \cite{AP} 
to critical Besov spaces of type $\dot B^s_{p,1}$ with $p>1.$
\medbreak
Again in the critical functional framework,
 J.~Huang, M.~Paicu and P.~Zhang noticed 
 that, somehow, only $d-1$ components of $u_0$ need
 to be small for global existence of weak solutions: 
 in \cite{HPZ2013}, they just required that  $$(\mu \norm{a_{0}}_{L_{\infty}}+\norm{u_{0}^{h}}_{\B^{-1+\frac{d}{p}}_{p,r}})\exp(C_{r}\mu^{-2r}\norm{u^{d}_{0}}^{2r}_{\B^{-1+\frac{d}{p}}_{p,r}})\leq c_{0}\mu$$ for some positive constants $c_{0}$ and  $C_{r}$.
\medbreak
Achieving results in the critical functional framework  when the density has large variations requires techniques that are not just
based on perturbation arguments. 
In \cite{DR2004}, the  first author investigated the problem 
in Sobolev spaces but failed  to reach 
the critical exponent. 
Recently, H.~Abidi and G.~Gui \cite{AG} proved the global unique solvability of the 2-D incompressible inhomogeneous Navier-Stokes equations
 whenever $\rho^{-1}_0-1$ is in  $\B^{2/p}_{p,1}(\R^2)$  for some $2<p<\infty,$ and  $u_0$ is in $\B^0_{2,1}(\R^2)$.
 This is, to  our knowledge, the first global 
 well-posedness result at the critical level of regularity, that 
 does not require any smallness condition
 (see also the work by H. Xu in \cite{Xu}, based on different techniques). 
\medbreak
A number of recent works aimed at proving existence
and uniqueness results in the case where the density 
is only bounded (and not continuous). 
In this respect, significant progress has been done by  M. Paicu, P. Zhang and Z.~Zhang in \cite{PZZ} where the global existence and uniqueness 
of solution to (INS) is shown in $\R^d,$ $d=2,3$
assuming only that $\rho_0^{\pm1}$ is bounded and that  $u_{0}\in H^{s}(\R^{2})$ for some  $s>0$ (2D case) or $u_{0}\in H^{1}(\R^{3})$ with $\norm{u_{0}}_{L_{2}}\norm{\nabla u_{0}}_{L_{2}}$  sufficiently small (3D case). This result was extended to 
velocities in $H^s(\R^3)$ with $s>1/2$ by D. Chen, Z. Zhang and W. Zhao in \cite{DZW}.  
 Finally, the lower bound assumption was totally removed by the first author and P.B. Mucha in \cite{DM1}
in the case where the fluid domain $\Omega$ is either bounded or the torus. 
There, it is only needed that $u_0$ is 
 $H^1_0(\Omega)$ and that $\rho_0$ is bounded. 
 \smallbreak
Very recently, in the 3D case, P.~Zhang \cite{Zhang19} established the
global existence of weak solutions to the 3D inhomogeneous incompressible Navier-Stokes system   with initial density in $L_\infty(\R^3),$   bounded away from zero, and  initial velocity sufficiently small in the critical Besov space $\B^{1/2}_{2,1}(\R^{3}).$ 
This is the first example of a global existence result within 
a Besov critical framework for the velocity and no regularity for the density,  in the large variations case. Note 
however that the uniqueness of these solutions has not been proved
so far. 
\medbreak
The primary goal of our paper is to establish  the 
global existence of solutions of (INS) \emph{that are unique}
in a critical regularity framework, 
in the case where the initial density is close
to a positive constant in $L_\infty$ but has no regularity whatsoever. To our knowledge, no result of this type 
has been proved before. 
In accordance with the state-of-the art for 
the homogeneous  Navier-Stokes equations 
(that is, with constant density), smallness of the  velocity  
will be  required  if $d=3,$ but not if $d=2.$

The uniqueness part of our statements will come up has an easy consequence of a much more general 
 result within a critical regularity framework, that allows for density with large variations
 (that is even allowed to vanish on arbitrary sets if the dimension is $3$). 
As a by-product, we shall obtain that the global 
solutions constructed by P. Zhang 
in \cite{Zhang19} (that  are   allowed to have large
density variations), are actually unique. 
\medbreak
Our existence results are strongly based on 
a novel maximal regularity estimate for 
the Stokes system  equation originating from the recent paper \cite{DM2}
by P.B.~Mucha, P.~Tolksdorf  and the first author, where   the time regularity 
  is measured in \emph{Lorentz spaces}. 
  Time weighted estimates will also play an important role (see the end 
  of the next section for more explanation).

 %More precisely, we will only assume that 
%the initial velocity is in the homogeneous Besov space $\B^{-1+d/p}_{p,1}(\R^{d}),$ while the density 
% satisfies  $\norm{\rho_{0}-1}_{L_{\infty}(\R^{d})}\leq c$ for some suitably %small constant $c.$ So far, we have no idea how to handle the case with %vacuum in the 2D case, without assuming more regularity. 

\medbreak

%%%%%%%%%%%%%%%%%%%%%%%%%%%%%%%%%%%%%%%

\section{Tools, results and approach}\label{section1}

Before stating our main existence results for (INS), introducing a few notations and recalling some results is in order. 
\medbreak
First, throughout the text,
$A\lesssim B$ means that  $A\leq CB$, where $C$ designates 
 various positive real numbers the value of which does not matter. 
 \smallbreak
For any Banach space $X,$ index $q$ in $[1,\infty]$  and time $T\in[0,\infty],$ we use the notation 
$\|z\|_{L_q(0,T;X)}\overset{\text{def}}{=}
\bigl\| \|z\|_{X}\bigr\|_{L_q(0,T)}.$
If $T=\infty$, then we  just  write $\|z\|_{L_q(X)}.$
In the case where $z$ has $n$ components $z_k$ in $X,$ we 
 keep the notation  $\norm{z}_X$ to 
mean $\sum_{k\in\{1,\cdots,n\}} \norm{z_k}_X$. 
\smallbreak
We shall use the following notation for the \emph{convective derivative}: 
\begin{equation}\label{def:convective} 
\frac{D}{Dt}\overset{\text{def}}{=}\d_{t}+u\cdot \nabla \andf \t{u}\overset{\text{def}}{=}u_{t}+u\cdot \nabla u.\end{equation}
 Next, let us recall the definition of Besov spaces on $\R^d.$ 
Following \cite[Chap. 2]{BCD}, we fix two smooth functions $\chi$ and $\varphi$ 
 such that 
$$\displaylines{\text{Supp}\  \varphi \subset \{\xi\in\R^{d},\: 3/4\leq \abs{\xi}\leq 8/3\}\ \ \text{and  } \ \ \forall\xi\in\R^{d}\setminus\{0\}, \;\underset{j\in\Z}{\sum}\varphi(2^{-j}\xi)=1,\cr
\text{Supp}\ \chi\subset\{\xi\in \R^{d},\: \abs{\xi}\leq 4/3\} \ \text{and}\ \ \forall \xi\in \R^{d},\;  \chi(\xi)+\sum_{j\geq 0}\varphi(2^{-j}\xi)=1,}$$
and set for all $j\in\Z$ and  tempered distribution $u,$ 
 $$\dot{\Delta}_{j}u\overset{\text{def}}{=}\mathcal{F}^{-1}(\varphi(2^{-j}\cdot)\wh{u})\overset{\text{def}}{=}2^{jd} \wt{h}(2^{j}\cdot)\star u\with\wt{h}\overset{\text{def}}{=}\mathcal{F}^{-1}\phi,$$
\begin{equation}\label{eq:Sj}\dot{S}_{j}u\overset{\text{def}}{=}\mathcal{F}^{-1}(\chi(2^{-j}\cdot)\wh{u})\overset{\text{def}}{=}2^{jd}h(2^{j}\cdot)\star u\with h \overset{\text{def}}{=}\mathcal{F}^{-1}\chi,\end{equation}
%and $$\Delta_j u=0\quad \text{if $j\leq -2$}, \quad %\Delta_{-1}u=\mathcal{F}^{-1}(\chi\wh{u})=h\star u,$$ 
%$$\Delta_j u=\mathcal{F}^{-1}(\varphi(2^{-j}\cdot)\wh{u})=2^{jd} \wt{h}(2^{j}\cdot)\star %u=\dot{\Delta}_{j}u\quad \text{if $j\geq 0$} \andf S_j %u\overset{\text{def}}{=}\sum_{k\leq j-1}\Delta_j u,$$
where $\mathcal{F}u$ and $\wh{u}$ denote the Fourier transform of $u.$ 
%The above dyadic blocks operator has  nice properties of quasi-orthogonality
%$$ \dot\Delta_j\dot \Delta_k u\equiv 0 \quad \text{if 
%$\abs{j-k}\geq 2$}  \andf \dot \Delta_j(\dot S_{k-1}u\dot \Delta_k u)\equiv 0 %\quad\text{if 
%$\abs{j-k}\geq 5$}$$
%$$(\text{resp.}\Delta_j \Delta_k u\equiv 0 \quad \text{if 
%$\abs{j-k}\geq 2$}  \andf  \Delta_j( S_{k-1}u \Delta_k u)\equiv 0 \quad\text{if 
%$\abs{j-k}\geq 5$} )\cdotp$$

\begin{definition}[Homogeneous Besov spaces]
Let $(p,r)\in [1,\infty]^{2}$ and $s\in\R.$ 
We set 
$$\norm{u}_{\B^{s}_{p,r}(\R^{d})}\overset{\text{def}}{=}\norm{(2^{js}\norm{\dot{\Delta}_{j} u}_{L_{p}(\R^{d})})_{j\in \Z}}_{\mathit{\ell}_{r}(\Z)}.$$
We  denote by $\B^{s}_{p,r}(\R^{d})$ the set of tempered distributions $u$ 
 such that  $\norm{u}_{\B^{s}_{p,r}(\R^{d})}<\infty$ and
 \begin{equation}\label{eq:lf}
 \underset{j\to -\infty}{\lim}\norm{\dot S_{j}u}_{L_{\infty}(\R^{d})}=0.
 \end{equation}
\end{definition}
%We shall often use the fact that  a distribution $u$ of %$\mathcal{S}'_{h}(\R^{d})$ belongs to $\B^{s}_{p,r}(\R^{d})$ if and only if %there exists some nonnegative sequence $(c_{j})_{j\in\Z}$ such that 
%$$\forall j\in \Z, \norm{\dot{\Delta}_{j}u}_{L_{p}(\R^{d})}\leq %Cc_{j}2^{-js}\ \text{and}\ \norm{c_{j}}_{\ell_{r}(\Z)}=1.$$
It is classical that the scaling invariance
condition for $u_0$ pointed out in \eqref{isi} 
is satisfied for all elements of 
$\dot B^{-1+d/p}_{p,r}(\R^d)$ with $1\leq p,r\leq\infty.$

 \medbreak

%\begin{definition}[Chemin-Lerner type space]
%For $T>0,$ $s\in \R,$ and $1\leq r,\rho\leq \infty,$ we set 
%$$\norm{u}_{\tilde L_{\rho}(0,T;\B^{s}_{p,r}(\R^{d}))}\overset{\text{def}}{=}\norm{(2^{js}\norm{\dot{\Delta}_{j} u}_{L_{\rho}(0,T;L_{p}(\R^{d}))})_{j\in \Z}}_{\mathit{\ell}_{r}(\Z)}.$$ We  define the space $\tilde L_{\rho}(0,T;\B^{s}_{p,r}(\R^{d}))$ as the set of tempered distributions $u$ over $(0,T)\times\R^d$ such that $\underset{j\to -\infty}{\lim}\norm{\dot S_{j}u}_{L_{\rho}(0,T;L_{\infty}(\R^{d}))}=0$ and $\norm{u}_{\tilde L_{\rho}(0,T;\B^{s}_{p,r}(\R^{d}))}<\infty.$ 
%\end{definition}
%Comparing the space $\tilde L_{\rho}(0,T;\B^{s}_{p,r}(\R^{d}))$ with more classical spaces $ L_{\rho}(0,T;\B^{s}_{p,r}(\R^{d}))$ via the Minkowski inequality, we have
%$$\|u\|_{\tilde L_{\rho}(0,T;\B^{s}_{p,r})}\leq \|u\|_{L_{\rho}(0,T;\B^{s}_{p,r})} \quad \text{if $\ r\geq \rho$},\quad \|u\|_{\tilde L_{\rho}(0,T;\B^{s}_{p,r})}\geq \|u\|_{L_{\rho}(0,T;\B^{s}_{p,r})} \ \text{if $\ r\leq \rho$}.$$ 

Next, we define  Lorentz spaces,   and recall a useful characterization. 
\begin{definition}
Given $f$ a measurable function on a measure space $(X,\mu)$ and $1\leq p,r\leq \infty$, we define 
$$\wt\|{f}\|_{L_{p,r}(X,\mu)}:=
\begin{cases}
(\int_{0}^{\infty}(t^{\frac{1}{p}}f^{*}(t))^{r}\,\frac{dt}{t})^{\frac1r} & \text{if $r<\infty$},\\
\underset{t>0}{\sup} t^{\frac{1}{p}}f^{*}(t)& \text{if $r=\infty$},
\end{cases}$$
where $$f^{*}(t):=\inf\bigl\{s\geq 0:|\{\abs{f}>s\}|\leq t\bigr\}\cdotp$$
The set of all $f$ with $\wt\|{f}\|_{L_{p,r}(X,\mu)}<\infty$ is called the Lorentz space with indices $p$ and $r$.
\end{definition}
\begin{remark}\label{lorentzdef2}
It is well known that $L_{p,p}(X,\mu)$ coincides with 
the Lebesgue space $L_p(X,\mu).$
Furthermore, according to \cite[Prop.1.4.9] {LG}, the Lorentz spaces may be endowed with the following (equivalent) quasi-norm:  
\begin{equation*}
\norm{f}_{L_{p,r}(X,\mu)}:=
   \begin{cases}
   p^{\frac{1}{r}}\biggl(\Int_{0}^{\infty}\bigl(s|\{\abs{f}>s\}|^{\frac{1}{p}}\bigr)^{r}\,\frac{ds}{s}\biggr)^{\frac{1}{r}} & \text{if $r<\infty$}\\
  \underset{s>0}{\sup} s|\{\abs{f}>s\}|^{\frac{1}{p}}& \text{if $r=\infty$}.
   \end{cases}
\end{equation*}
\end{remark}

Our results will strongly rely on  a maximal regularity property  for the following evolutionary Stokes system:
  \begin{equation}\label{eq:stokes}
\left\{\begin{aligned}
 &u_{t}-\mu \Delta u+\nabla P=f &\ \ \text{in }\ \R_+\times \R^{d}, \\
 &\div u=0  &\ \ \text{in }\ \R_+\times \R^{d}, \\
&u|_{t=0}=u_{0} &\ \ \text{in }\  \R^{d}.
\end{aligned}\right.
\end{equation}
 It has been pointed out in  \cite[Prop. 2.1]{DM2} 
 that for the free heat equation
 supplemented with initial data $u_0$ in 
 $\dot B^{2-2/q}_{q,r}(\R^d),$
 the solution $u$ is such that $u_t$ and $\nabla^2u$ are in $L_{q,r}(\R_+;L_p(\R^d))$ and 
 that, conversely, the  Besov regularity 
  $\dot B^{2-2/q}_{q,r}(\R^d)$
 corresponds to the regularity of the trace at $t=0$ 
 of functions $u:\R_+\times\R^d\to\R$
 such that  $u_t,\nabla^2u\in L_{q,r}(\R_+;L_p(\R^d)).$
 
 This motivates us to introduce the following function space:
  \begin{equation}\label{eq:W}
  \W^{2,1}_{p,(q,r)}(\R_+\times \R^{d}):=\bigl\{u\in \mathcal{C}(\R_+;\B^{2-2/q}_{p,r}( \R^{d})):u_{t}, \nabla^{2}u\in L_{q,r}(\R_+;L_{p}( \R^{d})) \bigr\}\cdotp\end{equation}
 Back to (INS), in accordance with \eqref{isi}, we need $2-2/q=-1+d/p.$ Furthermore, 
for reasons that will be explained later on
 (in particular the fact $\dot B^{d/p}_{p,r}(\R^d)$
 embeds in $L_\infty(\R^d)$ if and only if $r=1$), 
 we shall only consider Besov spaces of type 
 $\dot B^{d/p-1}_{p,1}(\R^d).$
\medbreak
It is now time to state the main results of the paper. 
In the two-dimensional case, our  global existence result reads: 
\begin{theorem}\label{themd2} Let $p\in(1,2)$ and $q$
be defined by $1/q+1/p=3/2.$ Denote by 
$s$ and $m$ the conjugate Lebesgue exponents of $p$ and $q,$
respectively. 
Assume that the initial divergence-free velocity  $u_{0}$ is in $\B^{-1+2/p}_{p,1}(\R^{2}),$  
  and that $\rho_{0}$ belongs to $L_{\infty}(\R^{2})$. There exists a constant $c>0$ such that if 
\begin{equation}\label{inidr}
    \norm{\rho_{0}-1}_{L_{\infty}(\R^{2})}< c,
\end{equation}
 then (INS) has a unique global-in-time  solution 
 $(\rho,u,\nabla P)$  satisfying  the energy balance~\eqref{eq:energy},
 $u\in \dot W^{2,1}_{p,(q,1)}(\R_+\times\R^2),$ 
 $\nabla P\in L_{q,1}(\R_+;L_p(\R^2)),$
   \begin{equation}\label{eq:smallrho}\norm{\rho-1}_{L_{\infty}(\R_+\times\R^{2})}
 =    \norm{\rho_{0}-1}_{L_{\infty}(\R^{2})}< c,
 \end{equation}
 and  the following properties: 
 \begin{itemize}
 \item     $\nabla u\in L_1(\R_+;L_\infty(\R^2))$ and $u\in L_2(\R_+; L_\infty(\R^2))$;
 \item $tu\in L_\infty(\R_+;\dot B^{1+2/m}_{m,1}(\R^2))$ and $\bigl(u,(tu)_t, \nabla^2(tu),\nabla(tP)\bigr)
 \in L_{s,1}(\R_+;L_m(\R^2))$;
 \item $t\t u\in  \dot W^{2,1}_{p,(q,1)}(\R_+\times\R^2)\,$
and $\,t\t u\in L_2(\R_+;L_\infty(\R^2))$;
     \item $t^{\frac k2}\nabla^k u\in L_\infty(\R_+;L_2(\R^2))$
     and $t^{\frac k2}\nabla^{k+1}u\in L_2(\R_+\times\R^2)$
     for $k=0,1,2,$
    \item  $t^{\frac{k+2}2}\nabla^{k}\dot u\in L_\infty(\R_+;L_2(\R^2))$ for $k=0,1$ and
    $t^{\frac{k+1}2}\nabla^{k}\dot u\in L_2(\R_+\times\R^2)$
     if $k=0,1,2,$
     \item $t^{\frac12}\nabla P\in L_2(\R_+\times\R^2)$
     and $t\nabla P\in L_\infty(\R_+; L_2(\R^2)).$
 \end{itemize}
% Furthermore, the following inequalities are satisfied: 
%$$\displaylines{\quad
%\mu^{\frac 1p-\frac12}\norm{u}_{L_{\infty}(\R_{+};
%\B^{-1+2/p}_{p,1}(\R^{2}))}+\norm{u_{t},  
%\mu\nabla^{2} u,\nabla P}_{L_{q,1}(\R_{+};L_{p}( \R^{2}))}+\mu^{\frac12}\norm{u}_{L_{s,1}(\R_{+};L_{m}(\R^{2}))}
%\hfill\cr\hfill
%\leq C \mu^{\frac 1p-\frac12} %\norm{u_{0}}_{\B^{-1+2/p}_{p,1}(\R^{2})}e^{C\mu^{-2}\norm{u_{0}}^{2}_{L^2(\R^{2})}},}$$
%$$\displaylines{\mu^{\frac2{s}}\int_{0}^{\infty}\norm{\nabla u}_{L_{\infty}(\R^{2})}\,dt\hfill\cr\hfill
%\leq C\norm{u_{0}}^{2/s}_{\B^{-1+2/p}_{p,1}(\R^{2})}  
%e^{C\mu^{-2}\norm{u_{0}}^{2}_{L^2(\R^{2})}}\exp\Bigl(C\mu^{-s}\norm{u_{0}}^{s}_{\B^{-1+2/p}_{p,1}(\R^{2})}e^{C\m%u^{-2}\norm{u_{0}}^{2}_{L_2(\R^{2})}}\Bigr)
%% \exp\left(C(\mu^{-s}\norm{u_{0}}^{s}_{\B^{-1+2/p}_{p,1}(\R^{2})}+\mu^{-2}\norm{u_{0}}^{2}_{\B^{-1+2/p}_{p,1}(\%R^{2})})\right),}$$
%$$\displaylines{\quad
%\mu^{\frac2{p'}}\norm{tu}_{L_{\infty}(\R_{+};\B^{1+2/m}_{m,1}(\R^{2}))}+
%\norm{\mu^{-\frac1{s'}}(tu)_{t}, %\mu^{\frac1s}\nabla^{2}(tu),\mu^{-\frac1{s'}}\nabla(tP)}_{L_{s,1}(\R_{+};L_{m}(\R^{2}))}
%\hfill\cr\hfill
%\leq C  \exp\Bigl(C\mu^{-s}\norm{u_{0}}^{s}_{\B^{-1+2/p}_{p,1}(\R^{2})}e^{C\mu^{-2}\norm{u_{0}}^{2}_{L_2(\R^{2})%}}\Bigr)
% %\exp\left(C(\mu^{-s}\norm{u_{0}}^{s}_{\B^{-1+2/p}_{p,1}(\R^{2})}+\mu^{-2}\norm{u_{0}}^{2}_{\B^{-1+2/p}_{p,1}(\R^{2})})\right)\cdotp }$$
\end{theorem}
%\begin{remark}\label{r:d=2} For $p\leq2,$ we have %$\B^{-1+2/p}_{p,1}(\R^{2})\hookrightarrow L_{2}(\R^{2}).$ 
%Hence the solutions 
%constructed above have finite energy.\end{remark}
In dimension three, our global existence result reads~:
\begin{theorem}\label{them1d3}
Let $p\in(1,3)$ and $q\in(1,\infty)$ such that $3/p+2/q=3.$ There exist a 
 positive constant $c$ such that if the initial density 
is such that 
\begin{equation}\label{inidr3}
    \norm{\rho_{0}-1}_{L_{\infty}(\R^{3})}< c,
\end{equation}
and if the initial divergence-free velocity  satisfies
$$u_{0}\in \B^{-1+3/p}_{p,1}(\R^{3})\quad (1<p\leq 2)\ \hbox{ or }\  u_{0}\in \B^{-1+3/p}_{p,1}(\R^{3})\cap L_2(\R^3)\quad (2<p<3)$$
with 
\begin{equation}\label{ini1d3} \norm{u_{0}}_{\B^{-1+3/p}_{p,1}(\R^{3})}<c\mu,
\end{equation}
 then (INS) has a unique global-in-time solution $(\rho,u,\nabla P)$ 
 with  $\nabla P\in L_{q,1}(\R_+;L_p(\R^3))$ and 
  $u\in\dot W_{p,(q,1)}^{2,1}(\R_+\times\R^3),$
 satisfying the energy balance \eqref{eq:energy} if $p>2$,
   \begin{equation}\label{eq:smallrhod3}\norm{\rho-1}_{L_{\infty}(\R_+\times\R^{3})}
 =    \norm{\rho_{0}-1}_{L_{\infty}(\R^{3})}< c,
 \end{equation}
 
  and, furthermore, the following properties:
  \begin{itemize}
 \item  $\nabla u\in L_1(\R_+;L_\infty(\R^3))$ and $u\in L_2(\R_+; L_\infty(\R^3))$;
  \item $(tu)\in  W_{m,(s,1)}^{2,1}(\R_+\times\R^3)$ and $t\nabla P\in L_{s,1}(\R_+;L_m(\R^3))$
   for all  $3<m<\infty$  and  $q<s<\infty$  such that  $3/m+2/s=1$;
   \item  $t\t u\in\dot W_{p,(q,1)}^{2,1}(\R_+\times\R^3)$;
   \item $(u, t\t u)\in L_{s,1}(\R_+;L_m(\R^3)).$
   \end{itemize} 
  %  $$\displaylines{\mu^{\frac{3}{2p}-\frac 12} \norm{u}_{L_{\infty}(\R_{+};
%\B^{-1+3/p}_{p,1}(\R^{3}))}+\norm{u_{t}, \mu\nabla^{2} u,\nabla P}_{L_{q,1}(\R_{+};L_{p}( \R^{3}))}\leq %C\mu^{\frac{3}{2p}-\frac 12} \norm{u_{0}}_{\B^{-1+3/p}_{p,1}(\R^{3})},\cr
%\int_{0}^{\infty}\norm{\nabla u}_{L_{\infty}(\R^{3})}\,dt\leq C \norm{u_{0}}_{\B^{-1+3/p}_{p,1}(\R^{3})},}$$
%where $q$ is defined by  $2/q+3/p=3.$ 
%\medbreak
%Furthermore, if $(m,s)$   are such that    $3/m+2/s=1$ with $3<m<\infty,$ $q<s<\infty,$  
%then we have
%$$\displaylines{\quad\mu^{1-\frac 1s}\norm{tu}_{L_{\infty}(\R_{+};\B^{1+3/m}_{m,1}(\R^{3}))}+\norm{(tu)_{t}, %\mu\nabla^{2}(tu),\nabla(tP)}_{L_{s,1}(\R_{+};L_{m}(\R^{3}))}
%\hfill\cr\hfill
%\leq C \mu^{-\frac 1s} \norm{u_{0}}_{\B^{-1+3/p}_{p,1}(\R^{3})}\cdotp} $$
\end{theorem}

\begin{remark}
If $p>2,$ then 
the (subcritical) assumption $u_0\in L_2(\R^3)$  ensures the constructed solution to have finite energy. It is only required for proving uniqueness, 
and it is not needed if $p\leq2.$
 At the same time,   the priori estimates 
leading to global existence  are performed in  critical spaces, and do not require the energy to be finite.

Like in the two-dimensional case, higher order time weighted energy estimates may be proved. However, since they are not needed 
for getting uniqueness, we refrain from stating them. 
\end{remark}

The uniqueness part of the above two theorems is a consequence of the following much more general result. 
\begin{theorem}\label{thm:uniqueness}
 Let $(\rho_1,u_1,P_1)$ and  $(\rho_2,u_2,P_2)$ be two
solutions of $(INS)$ on $[0,T]\times\R^d$ corresponding to the same initial data. 
Assume in addition that: 
\begin{itemize}
\item $\sqrt{\rho_1}(u_2-u_1)\in L_\infty(0,T;L_2(\R^d))$;
\item $(\nabla u_2-\nabla u_1)\in L_2(0,T\times\R^d)$;
\item $\nabla u_2\in L_1(0,T;L_\infty(\R^d))$;
\item $t\dot u_2\in L_2(0,T;L_\infty(\R^d))$;
\item Case $d=2$:   $\rho_0$ is bounded away from zero and 
$t\nabla^2\dot u_2\in L_q(0,T;L_p(\R^2))$ for some $1<p,q<2$  such that $1/p+1/q=3/2,$
%or $\rho_0$ may vanish but  $t \dot u_2\in L_2(0,T; H^1(\T^2))$ and   $\rho_0$ may vanish with compact support but  $t \dot u_2\in L_2(0,T;\H^1(\R^2)\cap \H^{\frac 12}(\R^2)).$
\item Case $d=3$: $t\nabla\dot u_2\in L_2(0,T;L_3(\R^3)).$
\end{itemize}
Then, $(\rho_1,u_1,P_1)\equiv (\rho_2,u_2,P_2)$ on  $[0,T]\times\R^d.$ 
\end{theorem}
\begin{remark}
Although the density may have large variations
(and even vanish in the three-dimensional case), 
the regularity requirements  in the above
 uniqueness result are all at the critical level
 in the sense of \eqref{isi}.

 In the last section of the paper, we shall present
 another uniqueness statement in dimension two, 
 that allows for vacuum, but require
 a slightly supercritical regularity assumption. 
\end{remark}
% Gagliardo-Nirenberg inequality:
%$$\norm{z}_{L_{\infty}(\R^{2})}\lesssim %\norm{z}^{1/2}_{L_{2}(\R^{2})}\norm{\nabla^2 %z}^{1/2}_{L_{2}(\R^{2})}$$
%to ensure that the solutions built in Theorem \ref{themd2} %satisfy 
% $t \dot u\in L_2(0,T;L_\infty(\R^2)\cap \H^1(\R^2)\cap %\H^{\frac 12}(\R^2))$ and $t \dot u\in %L_2(0,T;L_\infty(\R^2)\cap H^1(\R^2)).$ Hence the above  %uniqueness result indeed applies to them. 
%\end{remark}

We shall also see that the Fujita-Kato type global 
solutions  constructed 
by P. Zhang in \cite{Zhang19} satisfy 
$\nabla u\in L_1(\R_+;L_\infty(\R^3)).$
As a consequence,  Theorem \ref{thm:uniqueness} will 
ensure uniqueness. 
This leads to the following global well-posedness statement. 
\begin{theorem}\label{them:PZuniqueness}
Let $(\rho_0,u_0)$ satisfy $$0<c_0\leq \rho_0\leq C_0<\infty \andf u_{0}\in \B^{{1}/{2}}_{2,1}(\R^3).$$
Then, there exists a constant $\eps_0>0$ depending only on $c_0,C_0$ such that if 
\begin{equation}\label{eq:Z1}
\norm{u_0}_{ \B^{{1}/{2}}_{2,1}(\R^3)}\leq \varepsilon_0\mu,
\end{equation}
then System (INS) has a unique global solution $(\rho,u, \nabla P)$ with  $u\in \cC(\R_+;\B^{1/2}_{2,1}(\R^3))\cap L_2(\R_+;\B^{3/2}_{2,1}(\R^3))$ which satisfies
\begin{equation}\label{eq:Z2}c_0\leq \rho \leq C_0 \ \hbox{ on }\  \R_+\times \R^3,\end{equation}
and, for some absolute constant $C,$
\begin{multline}\label{eq:PZsolutions}
 \norm{u}_{L_{\infty}(\R_+;\B^{1/2}_{2,1})}+
 \sqrt\mu\,\norm{(u,t\t{u})}_{ L_{2}(\R_+;\B^{3/2}_{2,1})}+\norm{\sqrt{\mu t}\,u}_{ L_{\infty}(\R_+;\B^{3/2}_{2,1})}\\+\norm{\sqrt{t}(\mu\nabla u,P )}_{ L_{2}(\R_+;\B^{1/2}_{6,1})}
 +\mu\norm{\sqrt{t}\,u_t}_{ L_{2}(\R_+;\B^{1/2}_{2,1})}\\
 +\norm{\sqrt{t}(\mu\nabla^2 u,\nabla P )}_{ L_{2}(\R_+;L_3)}+\norm{tu_t}_{L_{\infty}(\R_+;\B^{1/2}_{2,1})}
 \leq C\norm{u_0}_{ \B^{{1}/{2}}_{2,1}}. 
\end{multline}
Furthermore, we have $\nabla u$ is in $L_1(\R_+;L_\infty(\R^3))$ with
$$\mu\norm{\nabla u}_{L_1(\R_+;L_\infty(\R^3))}\leq C\norm{u_0}_{ \B^{{1}/{2}}_{2,1}(\R^3)}\cdotp $$
\end{theorem}

Let us shortly present the main ingredients leading 
to the above statements. 

The common starting point for proving the 
existence part in Theorems  \ref{themd2} and \ref{them1d3}
is the maximal regularity result in Lorentz spaces
stated in Proposition \ref{propregularity}. In fact, in 
parabolic spaces, the Besov regularity that is required
for the initial velocity exactly corresponds 
to the trace at $t=0$ of functions $u:\R_+\times\R^d\to\R$
such that $u_t$ and $\nabla^2u$ are in $L_{q,1}(\R_+;L_p(\R^d)).$
Then,  proving estimates for (INS) is based on a perturbation argument from the Stokes system (this is the only 
place where we need the density 
to be close to some positive constant). 
In dimension $d=2,$ the space $L_2(\R^2)$ turns out to be critical, and one can combine these estimates with  the energy 
balance \eqref{eq:energy} so as to discard any 
smallness assumption for the velocity.

Since the first part of (INS) is a transport equation, in order
to prove the uniqueness, it is essentially mandatory to 
have at least $\nabla u\in L_{1,loc}(\R_+;L_\infty(\R^d)).$
This property will be achieved by combining 
critical estimates for $u$ and $tu,$ with 
an interpolation argument involving, again, 
Lorentz norms for the time variable. 

In our setting, it is not clear whether knowing 
only $\nabla u\in L_{1,loc}(\R_+;L_\infty(\R^d))$ is 
enough to get uniqueness. 
Here, to conclude,  we establish a number
of time weighted estimates of energy type (still
involving only critical norms). 
We will in particular get accurate enough 
information on $\t u,$ which will spare us going to Lagrangian coordinates. 
In fact, in contrast with recent works on similar issues
(see e.g.\cite{DM1,DM2}) our proof of uniqueness is  performed directly on the original system (INS): we estimate
the difference of velocities in the energy space 
and, by means of a duality argument, the difference of
densities in  $\dot H^{-1}(\R^d).$
In dimension $d=3,$ we do not need the density to be positive. 
In the two-dimensional case,   the space $\dot H^{1}(\R^2)$
fails to be embedded in any Lebesgue space, 
which complicates the proof, unless 
the density has a positive lower bound. 
If it is not the case, then one can 
combine a suitable logarithmic interpolation 
inequality with Osgood lemma so as to get 
a uniqueness result in some cases where the density 
vanishes. However, we have to  strengthen slightly our regularity
requirement on the velocity (see the end of Section 5). 
\medbreak
The rest of the paper unfolds as follows. The a priori estimates leading to global existence for  Theorems  \ref{themd2} and \ref{them1d3}  are performed in the next two sections. 
 Section \ref{section4} is devoted to   the proof of the global existence. Section \ref{section5} is dedicated to  the proof of 
 various stability estimates and uniqueness statements
 that, in particular, imply  Theorem \ref{thm:uniqueness} and 
 the  uniqueness part of Theorems  \ref{themd2}, \ref{them1d3} and \ref{them:PZuniqueness}. 
 For reader's convenience, we present in Appendix
 the maximal regularity result in Lorentz spaces 
 of \cite{DM2} adapted to the Stokes system, 
 recall a few properties of Besov and Lorentz spaces
 and prove a critical bilinear estimate with a logarithmic loss
  that is needed  for uniqueness in dimension $d=2.$

%%%%%%%%%%%%%%%%%%%%%%%%%%%%%%%%%%%%%%

\section{A priori estimates in the 2D case} \label{section2}

This part is devoted to the proof of a priori estimates 
for (INS) in the 2D case.
We shall first establish estimates for $u$ in the critical 
regularity space $\dot W^{2,1}_{p,(q,1)}(\R_+\times\R^2)$
with $1/q+1/p=3/2$ defined in \eqref{eq:W}, which actually suffices to get the global existence of a solution. 
Then, we will prove time weighted estimates 
both of energy type and in critical Besov spaces
that are needed for uniqueness. 
The last statement of the section points out higher order time weighted estimates, of independent interest.
%We have to keep in mind that, for $p\leq2,$ we have %%$\B^{-1+2/p}_{p,1}(\R^{2})\hookrightarrow L_{2}(\R^{2}).$ 
%Hence the solutions that we shall consider will have finite energy
\smallbreak
%We shall 
%first prove the estimates that are stated in Theorem \ref{themd2}, then establish 
%supplementary  time weighted estimates -- of energy type -- that will be needed in the proof of uniqueness. 
\begin{proposition}\label{prop0d2}
Let $(\rho,u)$ be a smooth solution of  (INS) on $[0,T]\times\R^2$
with sufficiently decaying velocity, and density satisfying
\begin{equation}\label{eq:smallrho1}\sup_{t\in[0,T]}\norm{\rho(t)-1}_{L_{\infty}( \R^{2})}\leq c\ll1.\end{equation}
  Then, it holds that 
\begin{equation}\label{eq:L2}
\norm{u}^{2}_{L_{\infty}(0,T;L_{2}(\R^{2}))}+ 2\mu\norm{\nabla u}^{2}_{L_{2}(0,T\times\R^{2})}\leq \norm{u_{0}}^{2}_{L_{2}(\R^{2})}
\end{equation}
and, for all $1<p,q<2$ with $1/p+1/q=3/2,$  
\begin{multline}\label{eq:u}
\mu^{\frac1p-\frac12}\norm{u}_{L_{\infty}(0,T;
\B^{-1+2/p}_{p,1}(\R^{2}))}+\norm{u_{t}, \mu\nabla^{2} u,\nabla P}_{L_{q,1}(0,T;L_{p}( \R^{2}))}+\mu^{\frac12}\norm{u}_{L_{s,1}(0,T;L_{m}(\R^{2}))}
\\
\leq C \mu^{\frac1p-\frac12}\norm{u_{0}}_{\B^{-1+2/p}_{p,1}(\R^{2})}
e^{C\mu^{-2}\norm{u_{0}}^{2}_{L_2(\R^{2})}},\end{multline}
for a constant $C$ independent of $T$ and  $\mu,$
with $m$ and $s$ being the conjugate exponents of $q$ and
$p,$ respectively.
Furthermore, we have
  \begin{equation}\label{eq:dotu2} 
    \norm{\t{u}}_{L_{q,1}(0,T;L_{p}(\R^{2}))} \leq C \mu^{\frac1p-\frac12}\norm{u_{0}}_{\B^{-1+2/p}_{p,1}(\R^{2})}
    e^{C\mu^{-2}\norm{u_{0}}_{L_2(\R^{2})}^2}
\end{equation}
and 
 \begin{equation}\label{eq:uLinfty}\mu^{\frac12}\norm{u}_{L_2(0,T;L_{\infty}(\R^{2}))} \leq C\mu^{\frac1p-\frac12}\norm{u_{0}}_{\B^{-1+2/p}_{p,1}(\R^{2})}
 e^{C\mu^{-2}\norm{u_{0}}^{2}_{L_{2}(\R^{2})}}\cdotp\end{equation}
\end{proposition}
\begin{proof} 
Putting together the energy balance \eqref{eq:energy} and \eqref{eq:smallrho1}
clearly ensures \eqref{eq:L2} provided~$c$ has been chosen small enough. 
\medbreak
For proving the other inequalities, note that, 
thanks to  the following  rescaling:
   \begin{equation}\label{eq:rescaling}
   (\wt\rho,\wt u,\wt P)(t,x):= (\rho,\mu^{-1}u, \mu^{-2} P)(\mu^{-1}t,x),\qquad
      (\wt\rho_0,\wt u_0)(x):= (\rho_0,\mu^{-1}u_0)(x),
   \end{equation}
  one may assume with no loss of generality  that $\mu=1.$
\medbreak
In order to prove \eqref{eq:u}, let us observe that
\begin{equation}\label{s4e1}
u_{t}-\Delta u+\nabla P=-(\rho-1)u_{t}-\rho u\cdot \nabla u,\qquad\div u=0.
\end{equation}
%and \begin{equation}\label{edu1}
 %\rho \t{u}=\mu \Delta u -\nabla P \with \t{u}:= u_t+u\cdot\nabla u.
%\end{equation}
%Because $\div u=0,$ applying $\div$ to \eqref{s4e1} yields: 
%\begin{equation}\label{eqnp}
 %\nabla P=-\nabla (\Delta)^{-1}\div [(\rho-1)u_{t}+\rho u\cdot \nabla u],
%\end{equation}
%which implies 
%\begin{equation}\label{s4e2}
%u_{t}-\mu \Delta u=-\p [(\rho-1)u_{t}+\rho u\cdot \nabla u ],
%\end{equation}
%where $\p={\rm Id}-\nabla (\Delta)^{-1}\div$ is the Helmholtz projector on divergence free vector field.
Looking at \eqref{s4e1} as a Stokes   equation with source term,   Proposition \ref{propregularity} gives us
 \begin{multline}\label{esb2d}
     \norm{u}_{L_{\infty}(0,T;
\B^{-1+2/p}_{p,1}(\R^{2}))}+\norm{u_{t}, \nabla^{2} u,\nabla P}_{L_{q,1}(0,T;L_{p}( \R^{2}))}+\norm{u}_{L_{s,1}(0,T;L_{m}(\R^{2}))}\\
\leq C\bigl(\norm{u_{0}}_{\B^{-1+2/p}_{p,1}(\R^{2})}+\norm{(\rho-1)u_{t}+\rho u\cdot \nabla u}_{L_{q,1}(0,T;L_{p}( \R^{2}))}\bigr)\cdotp
 \end{multline}
By H\"older inequality, we have
%\begin{equation*}\begin{aligned}
$$\displaylines{\quad
\norm{(\rho-1)u_{t}+\rho u\cdot \nabla u}_{L_{q,1}(0,T;L_{p}( \R^{2}))}\leq
%&\,\norm{ [(\rho-1)u_{t}+\rho u\cdot \nabla u ]}_{L_{2}(0,T\times\R^{2})}\\\lesssim &\,
\norm{\rho-1}_{L_{\infty}(0,T\times\R^{2})}\norm{u_{t}}_{L_{q,1}(0,T;L_{p}( \R^{2}))}\hfill\cr\hfill+\norm{\rho}_{L_{\infty}(0,T\times\R^{2})}\norm{u\cdot \nabla u}_{L_{q,1}(0,T;L_{p}( \R^{2}))}.\quad}
$$
If $c$ is small enough in \eqref{eq:smallrho1}, then the first part in the right-hand side can be absorbed by the left-hand side of \eqref{esb2d}.
 For  the last term, we  have by H\"older inequality,
$$\norm{u\cdot \nabla u}_{L_{q,1}(0,T;L_{p}( \R^{2}))}\leq \norm{u}_{L_{s,1}(0,T;L_{m}(\R^{2}))}\norm{\nabla u}_{L_{2}(0,T;L_{2}(\R^{2}))}.$$
Hence, there exists a  (small) constant $\alpha>0$ such that if 
\begin{equation}\label{gub}
    \norm{\nabla u}_{L_{2}(0,T;L_{2}(\R^{2}))}\leq \alpha,
\end{equation}
then \eqref{esb2d} implies  that 
$$\norm{u}_{L_{\infty}(0,T;
\B^{-1+2/p}_{p,1}(\R^{2}))}\!+\!\norm{u_{t}, \nabla^{2}\!u,\nabla\! P}_{L_{q,1}(0,T;L_{p}( \R^{2}))}+\norm{u}_{L_{s,1}(0,T;L_{m}(\R^{2}))}
\!\lesssim\! \norm{u_{0}}_{\B^{-1+2/p}_{p,1}(\R^{2})}.$$
If \eqref{gub} is not satisfied then we follow the method used for proving  \cite[Theorem 3.1]{DM2}
and  split $[0,T]$  into a finite number $K$ of intervals $[T_{k-1},T_{k})$ with $T_{0}=0,$ $T_{K}=T$,
and $T_1,\cdots,T_{K-1}$ defined by:
\begin{equation*}
    \begin{array}{cc}
       \norm{\nabla u}_{L_{2}((T_{k-1},T_{k})\times \R^{2})}=\alpha  &\hbox{if }  1\leq k\leq K-1; \\[1ex]
       \norm{\nabla u}_{L_{2}((T_{k-1},T_{k})\times \R^{2})}\leq \alpha  &\hbox{for }  k=K.
    \end{array}
\end{equation*}
For fixed $\alpha,$ we calculate the value of $K$ by
$$\begin{aligned}
K\alpha^{2}\geq \sum^{K}_{k=1}\norm{\nabla u}^{2}_{L_{2}((T_{k-1},T_{k})\times \R^{2})}&=\norm{\nabla u}^{2}_{L_{2}(0,T\times \R^{2})}\\
&>\sum^{K-1}_{k=1}\norm{\nabla u}^{2}_{L_{2}((T_{k-1},T_{k})\times \R^{2})}=(K-1)\alpha^{2},\end{aligned}$$
which gives
\begin{equation}\label{valueK}
    K=\lceil \alpha^{-2}\norm{\nabla u}^{2}_{L_{2}(0,T\times \R^{2})}\rceil.
\end{equation}
Then, we adapt \eqref{esb2d} to  each interval  $[T_k,T_{k+1})$ getting 
$$\displaylines{
\quad
\norm{u}_{L_{\infty}(T_{k},T_{k+1};
\B^{-1+2/p}_{p,1}(\R^{2}))}+\norm{u_{t}, \nabla^{2} u,\nabla P}_{L_{q,1}(T_{k},T_{k+1};L_{p}( \R^{2}))}+\norm{u}_{L_{s,1}(T_{k},T_{k+1};L_{m}(\R^{2}))}
\hfill\cr\hfill
\leq C \norm{u(T_{k})}_{\B^{-1+2/p}_{p,1}(\R^{2})}.}$$
Arguing by induction,  taking $K$  according to \eqref{valueK} and  using \eqref{eq:L2}
so as to bound $\|\nabla u\|_{L_2(0,T\times\R^2)}$, we conclude that
\begin{multline}\label{eq:u2}
\norm{u}_{L_{\infty}(0,T;
\B^{-1+2/p}_{p,1}(\R^{2}))}+\norm{u_{t},\nabla^{2} u,\nabla P}_{L_{q,1}(0,T;L_{p}( \R^{2}))}+\norm{u}_{L_{s,1}(0,T;L_{m}(\R^{2}))}
\\
\leq C \norm{u_{0}}_{\B^{-1+2/p}_{p,1}(\R^{2})}\exp(C\norm{u_{0}}^{2}_{L_2(\R^{2})}).\end{multline}
%\begin{equation*}    \begin{aligned}
 % \norm{\nabla P}_{L_{q,1}(\R_{+};L_{p}( \R^{2}))}&\leq C(\norm{u_{t}}_{L_{q,1}(\R_{+};L_{p}(\R^{2}))}+\norm{u\cdot \nabla u}_{L_{q,1}(\R_{+};L_{p}(\R^{2}))}) \\
  %&\leq  C\norm{u_{0}}_{\B^{-1+2/p}_{p,1}(\R^{2})}(1+\norm{u_{0}}_{\B^{-1+2/p}_{p,1}(\R^{2})})\exp{(C\norm{u_{0}}^{2}_{\B^{-1+2/p}_{p,1}(\R^{2})})}.\end{aligned}\end{equation*}
In order to prove \eqref{eq:dotu2}, it suffices to use the fact that
$$\begin{aligned}
  \norm{\t{u}}_{L_{q,1}(0,T;L_{p}(\R^{2}))} &\leq  \norm{u_{t}}_{L_{q,1}(0,T;L_{p}(\R^{2}))}+\norm{u\cdot \nabla u}_{L_{q,1}(0,T;L_{p}(\R^{2}))}\\
  &\leq \norm{u_{t}}_{L_{q,1}(0,T;L_{p}(\R^{2}))}+\norm{u}_{L_{s,1}(0,T;L_{m}(\R^{2}))}\norm{\nabla u}_{L_{2}(0,T;L_{2}(\R^{2}))}.
  \end{aligned}$$
  Then, bounding the right-hand side according to \eqref{eq:L2} and \eqref{eq:u2}
  yields \eqref{eq:dotu2}. 
  \smallbreak
  Finally, as a consequence of Gagliardo-Nirenberg inequality and embedding,  we have:
  \begin{equation}\label{eq:lil2esd2}
  \norm{z}_{L_{\infty}(\R^2)}\lesssim \norm{z}^{1-q/2}_{L_{2}(\R^2)}\norm{\nabla ^{2}z}^{q/2}_{L_{p}(\R^2)}\lesssim \norm{z}^{1-q/2}_{\B^{-1+2/p}_{p,1}(\R^2)}\norm{\nabla ^{2}z}^{q/2}_{L_{p}(\R^2)}.
  \end{equation}

Hence,   using Inequality \eqref{eq:u2}, we find that 
\begin{equation}\label{esulinfty}
    \begin{aligned}
    \int_{0}^{T}\norm{u}^{2}_{L_{\infty}(\R^2)}\,dt &\leq C\int_{0}^{T}\norm{u}^{2-q}_{\B^{-1+2/p}_{p,1}(\R^2)}\norm{\nabla^{2}u}^{q}_{L_{p}(\R^2)}\,dt\\
    &\leq C\norm{u}^{2-q}_{L_{\infty}(0,T;\B^{-1+2/p}_{p,1}(\R^2))}\norm{\nabla^{2}u}^{q}_{L_{q}(0,T;L_{p}(\R^2))} \\
    &\leq C\norm{u_{0}}^{2}_{\B^{-1+2/p}_{p,1}(\R^2)}\exp{(C\norm{u_{0}}^{2}_{L_{2}(\R^2)})}.    \end{aligned}\end{equation}
  This completes the proof of the proposition.\end{proof}
\medbreak
{}For better readability, we drop  from now on $\R^2$ in  the norms.
\begin{proposition}\label{prop1}
Under the  assumptions of Proposition  \ref{prop0d2}, we have
$$\displaylines{
 \mu\norm{tu}_{L_{\infty}(0,T;\B^{2-2/s}_{m,1})}+\mu^{\frac1s}\norm{(tu)_{t}, \mu\nabla^{2}(tu),\nabla (tP)}_{L_{s,1}(0,T;L_{m})}
 +\mu^{\frac1s}\norm{t\dot u}_{L_{s,1}(0,T;L_{m})}
\hfill\cr\hfill
\leq C\cE_0\norm{u_{0}}_{\B^{-1+2/p}_{p,1}}\with
\cE_0:=\exp\Bigl(C\mu^{-s}\norm{u_{0}}^{s}_{\B^{-1+2/p}_{p,1}}
e^{C\mu^{-2}\norm{u_{0}}^{2}_{L_2}}\Bigr)
 \cdotp}$$
\end{proposition}
\begin{proof}
Again, we use the rescaling \eqref{eq:rescaling} to reduce the 
proof to the case $\mu=1.$ Now, multiplying both sides of \eqref{s4e1} by time $t$ yields
$$(tu)_{t}-\Delta (tu)+\nabla(tP)=-(\rho-1)(tu)_{t}+\rho u-\rho  u\cdot \nabla  tu,\qquad \div(tu)=0.$$
Then,  taking advantage of  of  Proposition \ref{propregularity}  with Lebesgue indices $m$ and $s$     gives
$$\displaylines{\quad
\norm{tu}_{L_{\infty}(0,T;\B^{2-2/s}_{m,1})}+\norm{(tu)_{t}, \nabla^{2}(tu),\nabla (tP)}_{L_{s,1}(0,T;L_{m})}
\hfill\cr\hfill
\leq\norm{\rho-1}_{L_{\infty}(0,T\times\R^{2})}\norm{(tu)_{t}}_{L_{s,1}(0,T;L_{m})}
\hfill\cr\hfill +\norm{\rho}_{L_{\infty}(0,T\times\R^{2})}
\left (\norm{u}_{L_{s,1}(0,T;L_{m})}+\norm{tu\cdot \nabla u}_{L_{s,1}(0,T;L_{m})}\right)\cdotp}$$
Owing to \eqref{eq:smallrho1},
the second line may be absorbed by the first one. 
Next, as $2/m=1-2/s,$ combining  H\"older inequality and 
the following embedding:
\begin{equation}\label{eq:embed}
\dot B^{2/m}_{m,1}(\R^2)\hookrightarrow L_\infty(\R^2)
\end{equation} yields
\begin{equation*}
    \begin{aligned}
    \norm{tu\cdot \nabla u}_{L_{s,1}(0,T;L_{m})}&\leq \norm{t \nabla u}_{L_{\infty}(0,T\times \R^{2})} \norm{u}_{L_{s,1}(0,T;L_{m})}\\
    %&\lesssim   \norm{t \nabla u}_{L_{\infty}(0,T;\B^{2/m}_{m,1}( \R^{2}))} %\norm{u}_{L_{s,1}(0,T;L_{m})}\\
    &\lesssim \norm{t u}_{L_{\infty}(0,T;\B^{2-2/s}_{m,1})} \norm{u}_{L_{s,1}(0,T;L_{m})}. 
    \end{aligned}
\end{equation*}
Hence, there exists a (small) positive constant $\beta$ such that, if 
$$\norm{u}_{L_{s,1}(0,T;L_{m})}\leq \beta,$$
then we have
\begin{equation}\label{estud2}
 \norm{tu}_{L_{\infty}(0,T;\B^{2-2/s}_{m,1})}+\norm{(tu)_{t}, \nabla^{2}(tu),\nabla (tP)}_{L_{s,1}(0,T;L_{m})}
\leq C\norm{u}_{L_{s,1}(0,T;L_{m})}.
\end{equation}
If  $\norm{u}_{L_{s,1}(0,T;L_{m})}>\beta,$  then one can argue as in the proof of the previous proposition:
there exists a finite sequence $0=T_{0}<T_{1}<\cdots<T_{K-1}<T_{K}=T$ such that
\begin{equation}\label{ulsmtk}
    \begin{array}{cc}
       \norm{\nabla u}_{L_{s,1}((T_{k-1},T_{k});L_m)}=\beta  &\hbox{if }  1\leq k\leq K-1; \\[1ex]
       \norm{\nabla u}_{L_{s,1}((T_{k-1},T_{k};L_m))}\leq \beta  &\hbox{for }  k=K.
    \end{array}
\end{equation}
%\begin{equation}\label{ulsmtk}
%\norm{u}_{L_{s,1}(T_{k},T_{k+1};L_{m})}\leq \beta\quad\hbox{for all }\ k\in \{0,1,\cdots, K-1\}.\end{equation}
Indeed, from Remark \ref{lorentzdef2}, we have
$$\norm{U(t)}_{L_{s,1}(0,T)}=s\int_{0}^{\infty} \abs{\{t\in(0,T): \abs{U(t)}>\lambda\}}^{1/s}\,d\lambda\with 
U(t):=\norm{u(t,\cdot)}_{L_{m}}$$ 
which, together with Lebesgue dominated theorem gives
$$%\norm{U(t)}_{L_{s,1}(0,T)}=s
\int_{0}^{\infty} \abs{\{t\in (T_1,T_2): \abs{U(t)}>\lambda\}}^{1/s}\,d\lambda\to 0\quad \text{as}\quad T_2-T_1\to 0,$$ 
which allows to construct a family $(T_k)_{0\leq k\leq K}$
satisfying \eqref{ulsmtk}.
Now, by H\"older inequality (with 
exponents $s$ and $p$) we have for all $\lambda>0,$
$$\begin{aligned}\sum_{k=1}^K 
\abs{\{t\in (T_{k-1},T_k): \abs{U(t)}>\lambda\}}^{1/s}
&\leq K^{1/p}\biggl(\sum_{k=1}^K 
\abs{\{t\in (T_{k-1},T_k): \abs{U(t)}>\lambda\}}\biggr)^{1/s}\\
&= K^{1/p}\abs{\{t\in (0,T): \abs{U(t)}>\lambda\}}^{1/s}.
\end{aligned}$$
Hence, integrating with respect to $\lambda$
and using \eqref{ulsmtk} yields  $K\lesssim \beta^{-s}\norm{u}_{L_{s,1}(0,T;L_{m})}^s.$
% Let $$X_k:=\underset{t\in(T_k,T_{k+1})}{\sup}\norm{tu}_{\B^{2-2/s}_{m,1}}+\norm{(tu)_{t}, \nabla^{2}(tu),\nabla (tP)}_{L_{s,1}(T_k,T_{k+1};L_{m})}.$$
 %Then, \eqref{estud2} and \eqref{ulsmtk} ensure that 
 %$$X_k\leq C(\beta+X_{k-1}) \quad \text{for all} \quad k\in %\{1,2,\cdots, K-1\}.$$
 %Hence, arguing by induction, we get for all $n\in \{ 0,1,\cdots, K-1\},$
 %$$ X_n\leq C\beta \sum^n_{l=0}\leq \frac{C^{K+1}}{{C-1}} \beta.$$
 Arguing by induction, we thus obtain 
 $$
 \norm{tu}_{L_{\infty}(0,T;\B^{2-2/s}_{m,1})}+\norm{(tu)_{t}, \nabla^{2}(tu),\nabla (tP)}_{L_{s,1}(0,T;L_{m})}
 \leq C\norm{u}_{L_{s,1}(0,T;L_{m})}e^{C\norm{u}_{L_{s,1}(0,T;L_{m})}^s}\cdotp$$
  In the end, using the first estimate of Proposition \ref{prop0d2}, one may conclude that
\begin{multline}\label{eq:uu}
\norm{tu}_{L_{\infty}(0,T;\B^{2-2/s}_{m,1})}+\norm{(tu)_{t}, \nabla^{2}(tu),\nabla (tP)}_{L_{s,1}(0,T;L_{m})}\\
\leq C \norm{u_{0}}_{\B^{-1+2/p}_{p,1}}\:
\exp\Bigl(C\norm{u_{0}}^{s}_{\B^{-1+2/p}_{p,1}}
\exp\bigl(C\norm{u_{0}}^{2}_{L_2}\bigr)\Bigr)\cdotp\end{multline}
To bound $t\dot u,$ we just 
have to observe that $t\dot u= (tu)_t - u + tu\cdot\nabla u.$ Hence, by H\"older inequality and \eqref{eq:embed}, we get
$$\norm{t\t{u}}_{L_{s,1}(0,T;L_{m})}\leq  \norm{(tu)_{t}}_{L_{s,1}(0,T;L_{m})}+\norm{u}_{L_{s,1}(0,T;L_{m})}+
\norm{tu}_{L_\infty(0,T; \dot B^{2-2/s}_{m,1})} \norm{u}_{L_{s,1}(0,T;L_{m})}.$$
At this stage, using Inequalities \eqref{eq:u} and \eqref{eq:uu} gives the desired result.
\end{proof}
\begin{corollary}\label{coro1d2}
With the notation of  Proposition  \ref{prop1}, we have:
\begin{align}\label{eq:tnablau1}
\mu\int_{0}^{T} \norm{\nabla u}_{L_{\infty}}\,dt&\leq C
\norm{u_{0}}_{\B^{-1+2/p}_{p,1}}\,\cE_0\,
e^{C\mu^{-2}\norm{u_{0}}^{2}_{L_2}},\\
\label{eq:tnablau2}
 \mu\biggl(\int_{0}^T t\norm{\nabla u}^{2}_{L_{\infty}}\,dt\biggr)^{1/2}&\leq C\norm{u_{0}}_{\B^{-1+2/p}_{p,1}}\,\cE_0\, e^{C\mu^{-2}\norm{u_{0}}^{2}_{L_2}}, \\
\label{estulilid2}\underset{t\in[0,T]}
\sup (\mu t)^{1/2} \|u(t)\|_{L_\infty}&\leq C\|u_0\|_{L_2}^{1/2}
\norm{u_{0}}_{\B^{-1+2/p}_{p,1}}^{1/2} \,\cE_0.\end{align}
\end{corollary}
\begin{proof} Just consider the case $\mu=1.$
  From the following Gagliardo-Nirenberg inequality
$$\norm{z}_{L_{\infty}}\lesssim \norm{ \nabla z}^{1-2/m}_{L_{p}}\norm{\nabla z}^{2/m}_{L_{m}},$$
and H\"older estimates in Lorentz spaces
(see Proposition \ref{p:lorentz}), we gather that
\begin{equation*}
\begin{aligned}
\int^{T}_{0}\norm{\nabla u}_{L_{\infty}}\,dt&\lesssim \int^{T}_{0} t^{-2/m} \norm{ \nabla^{2} u}^{1-2/m}_{L_{p}}\norm{t\nabla^{2} u}^{2/m}_{L_{m}}\,dt\\
&\lesssim\norm{t^{-2/m}}_{L_{m/2,\infty}(0,T)}\norm{\nabla^{2} u}^{1-2/m}_{L_{q,1}(0,T;L_{p})}\norm{t \nabla^{2} u}^{2/m}_{L_{s,1}(0,T;L_{m})}.\end{aligned}
\end{equation*}
As $t\mapsto t^{-2/m}\in L_{m/2,\infty}(\R_+)$
and  the other terms of
the  right-hand side may be bounded by means of Propositions
\ref{prop0d2} and \ref{prop1}, we get \eqref{eq:tnablau1}.
Next, by virtue of \eqref{eq:embed}, we have
\begin{equation*}
\begin{aligned}
\int_{0}^{T}t\norm{\nabla u}^{2}_{L_{\infty}}\,dt& \leq \int_{0}^{T} t\norm{\nabla u}_{\B^{2/m}_{m,1}}\norm{\nabla u}_{L_{\infty}}\,dt\\
&\lesssim \int_{0}^{T} \norm{tu}_{\B^{2-2/s}_{m,1}}\norm{\nabla u}_{L_{\infty}}\,dt\\
&\lesssim\norm{tu}_{L_{\infty}(0,T;\B^{2-2/s}_{m,1})} \norm{\nabla u}_{L_{1}(0,T;L_{\infty})},
\end{aligned}
\end{equation*}
whence the second inequality. 
\medbreak
%\begin{equation}\label{estulilid2}\begin{aligned}
  %\norm{t^{1/2}u}_{L_{\infty}(0,T\times\R^{2})}&\leq \norm{u}^{1/2}_{L_{\infty}(0,T;L_{2})}\norm{tu}^{1/2}_{L_{\infty}(0,T;\B^{1+2/m}_{m,1})}\\
    %&\leq C\norm{u_{0}}^{1/2}_{L_{2}} \exp(C\norm{u_{0}}^{s}_{\B^{-1+2/p}_{p,1}})e^{C\norm{u_{0}}^{2}_{L_{2}}}.
    %\end{aligned}\end{equation}
Finally, by interpolation, we have for all $t\in[0,T],$ 
$$t^{1/2} \|u(t)\|_{L_\infty}\lesssim \|u(t)\|_{L_2}^{1/2}\|t u(t)\|_{\dot B^{1+2/m}_{m,1}}^{1/2}$$
which, in light of \eqref{eq:L2} and of  Proposition  \ref{prop1} completes the proof. 
\end{proof}
\bigbreak
The rest of this section is devoted to establishing supplementary time weighted estimates of energy type that will be needed to prove the uniqueness of solutions of (INS). For expository purpose, we shall always assume that $\mu=1.$
\begin{proposition}\label{prop2}
Under the assumptions of Proposition  \ref{prop0d2}, we have  for  all $t\in [0, T],$ 
$$\displaylines{\quad
t\int_{\R^{2}}\abs{\nabla u(t)}^{2}\,dx+\int_{0}^{t}\!\!\int_{\R^{2}} \tau \bigl(\rho \abs{\t{u}}^{2} +\abs{\nabla^{2}u}^{2}+\abs{\nabla P}^{2}\bigr)\,dx \,d\tau
\hfill\cr\hfill
\leq C\norm{u_{0}}^{2}_{L_{2}}\,
\exp\left(C\norm{u_{0}}_{L_{2}}\norm{u_{0}}_{\B^{-1+2/p}_{p,1}}
\cE_0^2\right)\cdotp}$$
\end{proposition}

\begin{proof}
Let us rewrite the velocity equation as:
 \begin{equation}\label{edu1}
 \rho \t{u}=\Delta u -\nabla P \with \t{u}:= u_t+u\cdot\nabla u.
\end{equation}
As $\div u=0,$ testing \eqref{edu1} by $t\t{u}$ yields
$$
\int_{\R^2}\rho t|\t{u}|^2\,dx=t\int_{\R^2}\Delta u\cdot u_t\,dx
-t\int_{\R^2}\nabla P\cdot u_t\,dx + 
t\int_{\R^2}\bigl(\Delta u-\nabla P)\cdot(u\cdot\nabla u)\,dx$$
whence, integrating by parts and using again \eqref{edu1}, 
$$\frac{1}{2}\frac{d}{dt}\int_{\R^{2}}t\abs{\nabla u}^{2}\,dx+\int_{\R^{2}} \rho t\abs{\t{u}}^{2}\,dx=\int_{\R^{2}}\rho t\t{u}\cdot(u\cdot \nabla u)\,dx + \frac12\int_{\R^2}|\nabla u|^2\,dx.$$
Performing a time integration, we get for all $0\leq t\leq T,$
$$\frac{t}{2}\int_{\R^{2}}\!\abs{\nabla u(t)}^{2}dx+\int_{0}^{t}\!\!\int_{\R^{2}} \tau \rho \abs{\t{u}}^{2}dx\,d\tau=\int_{0}^{t}\!\!\int_{\R^{2}}\!\! \tau \rho \t{u}\cdot(u\cdot \nabla u)\,dx\, d\tau+\frac{1}{2}\int_{0}^{t} \!\!\int_{\R^{2}}  \abs{\nabla u(\tau)}^{2}dx d\tau.$$
To bound the right-hand side, we use the fact that 
%\begin{equation*}\label{esrh1}
   $$ \begin{aligned}
    \int_{0}^{t}\!\!\int_{\R^{2}}\! \tau \rho \t{u}\cdot 
    (u\cdot \nabla u)\,dx\,d\tau&\leq \int_{0}^{t} \norm{\sqrt{\rho\tau}\t{u}}_{L_{2}}\norm{\sqrt{\rho\tau}\,u\cdot \nabla u}_{L_{2}}\,d\tau\\
    &\leq \frac12 \int_{0}^{t} \norm{\sqrt{\rho \tau}\t{u}}^{2}_{L_{2}}\,d\tau
    +\frac{\|\rho_0\|_{L_\infty}}{2}\int_{0}^{t}\! \norm{u}^{2}_{L_{\infty}}\norm{\tau^{1/2}\nabla u}^{2}_{L_{2}}\,d\tau.
    \end{aligned}$$
%\end{equation*}
Observe that,  thanks to \eqref{edu1}, we have for some constant $C$ depending only on $\|\rho_0\|_{L_\infty},$ 
\begin{equation}\label{ed2}\norm{\nabla^{2}u}_{L_{2}}^2+\norm{\nabla P}_{L_{2}}^2\leq C\norm{\sqrt{\rho}\t{u}}_{L_{2}}^2.\end{equation}
Hence, applying Gronwall lemma yields some constant $C$
depending only on $\|\rho_0\|_{L_\infty},$ and such that
$$\displaylines{\quad
t\int_{\R^{2}}\abs{\nabla u(t)}^{2}\,dx+\int_{0}^{t}\!\!\int_{\R^{2}} \tau \rho \abs{\t{u}}^{2}\,dx\, d\tau +\int_{0}^{t}\!\!\int_{\R^{2}} \tau\bigl(\abs{\nabla^{2}u}^{2}+\abs{\nabla P}^{2}\bigr)dx\, d\tau
\hfill\cr\hfill
\leq C \int_0^t\norm{\nabla u}^{2}_{L_{2}}\exp \biggl(C\int_{\tau}^{t} \norm{u}^{2}_{L_{\infty}}\,d\tau'\biggr)d\tau.}$$
Putting  together with \eqref{eq:L2} and \eqref{esulinfty} completes the proof of the proposition.
%$$\displaylines{\quadt\int_{\R^{2}}\abs{\nabla u(t)}^{2}\,dx+\int_{0}^{t}\int_{\R^{2}} \tau \rho \abs{\t{u}}^{2}\,dx d\tau +\int_{0}^{t}\int_{\R^{2}} \tau\abs{\nabla^{2}u}^{2}\,dx d\tau +\int_{0}^{t}\int_{\R^{2}} \tau\abs{\nabla P}^{2}\,dx d\tau
%\hfill\cr\hfill
%\leq C\norm{u_{0}}^{2}_{L_{2}}\exp\left(\norm{u_{0}}^{2}_{\B^{-1+2/p}_{p,1}}\exp{(C\norm{u_{0}}^{2}_{L_{2}})}\right).}$$
\end{proof}

\begin{proposition}\label{prop3}  Under the assumptions of Proposition  \ref{prop0d2},
there exists a constant $C_0$ depending only on $p$ and on $\norm{u_{0}}_{\B^{-1+2/p}_{p,1}},$ such 
that for all $t\in [0,T],$ 
$$\int_{\R^{2}} t^{2}\Bigl(\rho\bigl(\abs{u_t}^{2}+\abs{\t{u}}^{2}\bigr)+
\abs{\nabla^{2} u}^{2}+\abs{\nabla P}^{2}\Bigr)dx
+\int_{0}^{t}\!\!\int_{\R^{2}} \tau^{2}
\bigl(\abs{\nabla u}^{2}+\abs{\nabla \t{u}}^{2}\bigr)dx\,d\tau\leq C_0.$$
\end{proposition}
\begin{proof}
From \eqref{eq:embed}, the definition of $\t{u}$ and
H\"older inequality, one can write
\begin{equation*}
    \begin{aligned}
    \norm{t\t{u}- t{u}_t}_{L_{\infty}(0,T;L_{2})}
%&\leq \norm{t u\cdot \nabla u}_{L_{\infty}(L_{2})}\\
    &\leq  \norm{t \nabla u}_{L_{\infty}(0,T\times\R^2)}\norm{u}_{L_{\infty}(0,T;L_{2})}\\
    &\leq C\norm{t u}_{L_{\infty}(0,T;\B^{1+2/m}_{m,1})}\norm{u}_{L_{\infty}(0,T;L_{2})}
    \end{aligned}
\end{equation*}
and
\begin{equation*}
    \begin{aligned}
    \norm{t\nabla \t{u}-t\nabla u_t&}_{L_{2}(0,T\times\R^2)}
    %&\leq \norm{t\nabla {u}_{t}}_{L_{2}(L_{2})}+\norm{t\nabla (u\cdot \nabla u)}_{L_{2}(L_{2})}\\
    \leq  \norm{t\nabla u\otimes \nabla u}_{L_{2}(0,T\times\R^2)}+\norm{t u\otimes \nabla^{2} u}_{L_{2}(0,T\times\R^2)}\\
    &\leq \norm{t\nabla u}_{L_{\infty}(0,T\times\R^2)}\norm{\nabla u}_{L_{2}(0,T\times\R^2)}+\norm{t\nabla^{2} u}_{L_{\infty}(0,T;L_{2})}\norm{u}_{L_{2}(0,T;L_{\infty})}\\
    &\leq C\norm{t u}_{L_{\infty}(0,T;\B^{1+2/m}_{m,1})}\norm{\nabla u}_{L_{2}(0,T\times\R^2)}+\norm{t\nabla^{2} u}_{L_{\infty}(0,T;L_{2})}\norm{u}_{L_{2}(0,T;L_{\infty})}.
    \end{aligned}
\end{equation*}
Furthermore, \eqref{ed2} implies that 
\begin{equation}\label{ed2b}\norm{t\nabla^{2}u}_{L_{2}}^2+\norm{t\nabla P}_{L_{2}}^2\leq C\norm{t\sqrt{\rho}\t{u}}_{L_{2}}^2.
\end{equation}
Hence, to complete the proof, it is only a matter of 
showing that 
$$\norm{t{u}_{t}}_{L_\infty(0,T;L_2)}
+\norm{t\nabla {u}_{t}}_{L_{2}(0,T\times \R^2)}\leq C_0.$$
To do so,  apply $\d_{t}$ to the momentum equation of (INS). We get
\begin{equation}\label{eq:utt}
    \rho u_{tt}+\rho u\cdot \nabla u_{t}-\Delta u_{t}+\nabla P_{t}=-\rho_{t}\t{u}-\rho u_{t}\cdot \nabla u.
\end{equation}
As $\div u_t=0,$ by taking the $L_{2}(\R^2;\R^2)$ scalar product of \eqref{eq:utt}
with  $t^{2}u_t,$ we obtain
$$\displaylines{
\quad
\frac{1}{2}\frac{d}{dt}\int_{\R^{2}}\rho t^{2}\abs{u_{t}}^{2}\,dx+\int_{\R^{2}}t^{2}\abs{\nabla u_{t}}^{2}\,dx
\hfill\cr\hfill
\leq \int_{\R^{2}} t\rho\abs{u_{t}}^{2}\,dx-\int_{\R^{2}}t^{2} \rho_{t}\t{u}\cdot u_{t}\,dx-\int_{\R^{2}}t^{2}\rho (u_{t}\cdot \nabla u)\cdot u_{t}\,dx.}$$
Then, integrating with respect to  time yields for all $t\in[0,T],$
\begin{multline}\label{eq:II2}
\frac12\underset{\tau\leq t}{\sup}\norm{t\sqrt{\rho} u_{t}}^{2}_{L_{2}}+\norm{t\nabla u_{t}}^{2}_{L_{2}(0,t\times\R^2}
\leq \int_{0}^{t}\!\!\int_{\R^{2}} \tau\rho\abs{u_{\tau}}^{2}\,dx\,d\tau\\-\int_{0}^{t}\!\!\int_{\R^{2}}\tau^{2} \rho_{\tau}\t{u}\cdot u_{\tau}\,dx\,d\tau-\int_{0}^{t}\!\!\int_{\R^{2}}\tau^{2}\rho (u_{\tau}\cdot \nabla u)\cdot u_{\tau}\,dx\,d\tau=:I_{1}+I_{2}+I_{3}.\end{multline}
For term $I_{2},$ the mass equation of (INS) and integration by parts
yield
\begin{equation*}
    \begin{aligned}
    I_{2}%&=-\int_{0}^{t}\!\!\int_{\R^{2}}\tau^{2} \rho_{\tau}\t{u}\cdot u_{\tau}\,dx\,d\tau\\
    &=\int_{0}^{t}\!\!\int_{\R^{2}}\tau^{2} \div(\rho u)\,\t{u}\cdot u_{\tau}\,dx\,d\tau\\
    &=-\int_{0}^{t}\!\!\int_{\R^{2}}\tau^{2} (\rho u\cdot \nabla{\t{u}}) \cdot u_{\tau}\,dx\,d\tau
    -\int_{0}^{t}\!\!\int_{\R^{2}}\tau^{2} (\rho u\cdot \nabla{u_{\tau})\cdot  \t{u} }\,dx\,d\tau\\
    &=:I_{21}+I_{22}.
    \end{aligned}
\end{equation*}
 Since $\t{u}=u_{t}+u\cdot \nabla u$, we may write
\begin{equation*}
    \begin{aligned}
    I_{21}=&-\int_{0}^{t}\!\!\int_{\R^{2}}\tau^{2}  (\rho u\cdot \nabla{u}_{\tau}) \cdot u_{\tau}\,dx\,d\tau
    -\int_{0}^{t}\!\!\int_{\R^{2}}\tau^{2} (\rho u\cdot \nabla(u\cdot \nabla u)) \cdot u_{\tau}\,dx\,d\tau\\
    =&-\int_{0}^{t}\!\!\int_{\R^{2}}\tau^{2} (\rho u\cdot \nabla u_{\tau}) \cdot u_{\tau}\,dx\,d\tau
    -\int_{0}^{t}\!\!\int_{\R^{2}}\tau^{2} \bigl(\rho u\cdot (\nabla^{2}u\cdot u) \cdot u_{\tau}\bigr)dx\,d\tau\\
    &\hspace{6cm}-\int_{0}^{t}\!\!\int_{\R^{2}}\tau^{2} \bigl(\rho u\cdot 
    (\nabla u\cdot\nabla u)\bigr) \cdot u_{\tau}\,dx\,d\tau
    \end{aligned}    \end{equation*}
    and 
    $$ I_{22}=-\int_{0}^{t}\!\!\int_{\R^{2}} \tau^{2}(\rho u\cdot \nabla u_{\tau})\cdot u_{\tau}\,dx\,d\tau
    -\int_{0}^{t}\!\!\int_{\R^{2}} \tau^{2}(\rho u\cdot \nabla u_\tau)\cdot ( u\cdot \nabla u)\,dx\,d\tau.$$
Applying Young's inequality and remembering that $\rho$ is bounded  gives
for all $\eps>0,$ 
    \begin{equation*}
  \begin{aligned}
   I_{21} &\lesssim \int_{0}^{t} \norm{u}_{L_{\infty}}\norm{\tau \nabla u_{\tau}}_{L_{2}}
    \norm{\tau\sqrt{\rho} u_{\tau}}_{L_{2}}\,d\tau\\
    &\qquad\qquad+\int_{0}^{t}\tau\norm{u}^{2}_{L_{\infty}}\norm{\nabla ^{2} u}_{L_{2}}\norm{\tau \sqrt{\rho} u_{\tau}}_{L_{2}}\,d\tau\\
    &\qquad\qquad+\int_{0}^{t} \norm{\tau^{1/2}\nabla u}_{L_{2}}\norm{\tau^{1/2}\nabla u}_{L_{\infty}} \norm{\tau\sqrt{\rho}u_{\tau}}_{L_{2}} \norm{ u}_{L_{\infty}}\,d\tau\\
    &\leq C\eps^{-1} \int_{0}^{t} \norm{u}^{2}_{L_{\infty}}
    \norm{\tau \sqrt{\rho} u_{\tau}}^{2}_{L_{2}}\,d\tau+C\int_{0}^{t}\norm{\tau^{1/2}u}^{2}_{L_{\infty}}\norm{\tau^{1/2}\nabla ^{2} u}^{2}_{L_{2}}\,d\tau\\
    &\qquad\qquad+C\int_{0}^{t} \norm{\tau^{1/2}\nabla u}^{2}_{L_{2}}
   \norm{\tau^{1/2}\nabla u}^{2}_{L_{\infty}}\,d\tau+\eps \int_{0}^{t}\norm{\tau \nabla u_{\tau}}^{2}_{L_{2}}\,d\tau,
    \end{aligned}
\end{equation*}
and 
\begin{equation*}
    \begin{aligned}
   I_{22} &\lesssim \int_{0}^{t} \norm{\tau \sqrt{\rho} u_{\tau}}_{L_{2}}\norm{\tau \nabla u_{\tau}}_{L_{2}}\norm{u}_{L_{\infty}} \,d\tau
    +\int_{0}^{t}\tau\norm{\tau \nabla u_{\tau}}_{L_{2}}\norm{u}^{2}_{L_{\infty}}\norm{\nabla u}_{L_{2}}\,d\tau\\
   %  &\leq C_{\eps}\left(\int_{0}^{t} \norm{\tau \sqrt{\rho} u_{\tau}}^{2}_{L_{2}}\norm{u}^{2}_{L_{\infty}}\,d\tau+\int_{0}^{t}\tau^{2}\norm{u}^{4}_{L_{\infty}}\norm{\nabla u}^{2}_{L_{2}} \,d\tau\right)\\
  %  &+\eps \int_{0}^{t}\norm{\tau \nabla u_{\tau}}^{2}_{L_{2}}\,d\tau\\
    &\leq C\eps^{-1}\biggl(\int_{0}^{t} \norm{\tau \sqrt{\rho} u_{\tau}}^{2}_{L_{2}}\norm{u}^{2}_{L_{\infty}}\,d\tau\\
        &\qquad\qquad\qquad+\int_{0}^{t}\norm{u}^{2}_{L_{\infty}}\norm{\tau^{1/2}u}^{2}_{L_{\infty}}\norm{\tau^{1/2}\nabla u}^{2}_{L_{2}} \,d\tau\biggr) +\eps \int_{0}^{t}\norm{\tau \nabla u_{\tau}}^{2}_{L_{2}}\,d\tau.
    \end{aligned}
\end{equation*}
For $I_{3}$, one has
$$ I_{3}=-\int_{0}^{t}\int_{\R^{2}}\tau ^{2}(\rho u_{\tau}\cdot \nabla u)\cdot u_{\tau}\,dx\,d\tau
    \leq \int_{0}^{t} \norm{\tau\sqrt{\rho} u_{\tau}}^{2}_{L_{2}}\norm{\nabla u}_{L_{\infty}}\,d\tau.$$
    Taking $\eps$ small  enough, then  
 reverting to \eqref{eq:II2}  and applying Gronwall inequality  gives 
$$\displaylines{
\quad
\underset{\tau\leq t}{\sup}\norm{t\sqrt{\rho}u_{t}}^{2}_{L_{2}}+\int_0^t\norm{t\nabla u_{t}}^{2}_{L_{2}}\,d\tau \leq C\exp\biggl(\int_{0}^{t} \norm{u}^{2}_{L_{\infty}}+\norm{\nabla u}_{L_{\infty}}\,d\tau\biggr)
\hfill\cr\hfill \biggl(\int_{0}^{t} \norm{\tau^{1/2}\nabla u}^{2}_{L_{2}}
   \norm{\tau^{1/2}\nabla u}^{2}_{L_{\infty}}\,d\tau+\int_{0}^{t}\!\!\int_{\R^{2}} \tau\rho\abs{u_{\tau}}^{2}\,dx\,d\tau
    \hfill\cr\hfill
   +\int_{0}^{t}\!\norm{u}^{2}_{L_{\infty}}\norm{\tau^{1/2}u}^{2}_{L_{\infty}}\norm{\tau^{1/2}\nabla u}^{2}_{L_{2}} \,d\tau +\int_{0}^{t}\!\norm{\tau^{1/2}u}^{2}_{L_{\infty}}\norm{\tau^{1/2}\nabla ^{2} u}^{2}_{L_{2}}\,d\tau\biggr)\cdotp}$$
Combining with Propositions \ref{prop1} and  \ref{prop2}, 
Inequality  \eqref{esulinfty} and
Corollary \ref{coro1d2}
allows to bound the right-hand side by $C_0$
 for all $t\in[0,T],$
and using also \eqref{ed2b} completes the proof. 
\end{proof}

 In order to get a higher order time weighted estimate, one has
 to consider  the evolutionary equation for $\dot u.$ 
So we take  the convective derivative of \eqref{edu1}, getting
 $$\frac{D}{Dt} (\rho \t{u})-\frac{D}{Dt}\Delta u+\frac{D}{Dt}\nabla P=0.$$
Observe that
 $$\begin{aligned}-\frac{D}{Dt} \Delta u&=-\Delta \t{u}+\Delta u \cdot \nabla u+2\nabla u \cdot \nabla^{2} u
  \with(\nabla u \cdot \nabla^{2} u)^{i}:=\underset{1\leq j,k\leq d}{\Sum} \d_{k}u^{j} \,\d_{j}\d_{k} u^{i},\\
 \frac{D}{Dt}\nabla P&=\nabla \t{P}-\nabla u\cdot \nabla P,\\
 \frac{D}{Dt}(\rho \t{u})&=\rho\ddot u\with 
 \ddot u:=\frac D{Dt}\t u.\end{aligned}$$
 Hence,  we have
 \begin{equation}\label{doubleD}
     \rho \ddot u-\Delta \t{u}+\nabla \t{P}=
     f\with f:=-\Delta u\cdot \nabla u-2\nabla u \cdot \nabla^{2} u+\nabla u\cdot \nabla P.\end{equation}

 \begin{proposition}\label{prop5}
 Under the  assumptions of Proposition \ref{prop1}, it holds that
 $$\norm{t\t{u}}_{L_{\infty}(0,T;
\B^{-1+2/p}_{p,1}(\R^2))}+\norm{(t\t{u})_{t}, t\nabla^{2} \t{u}}_{L_{q,1}(0,T;L_{p}( \R^{2}))}+\norm{t\dot u}_{L_2(0,T;L_\infty(\R^2))}\leq C_{0}.$$
 \end{proposition}
 \begin{proof} From \eqref{doubleD}, we get  
%\begin{equation}\label{eqtuf} (\t{u})_{t}-\Delta \t{u}+\nabla \t{P}=-(\rho-1)(\t{u})_{t}-\rho u\cdot \nabla \t{u}+f.\end{equation}
%After applying $\div $ to \eqref{eqtuf}, from \eqref{divtu} and $\div u=0,$ we discover that
 %$$\nabla \t{P}=\Q[-(\rho-1)(\t{u})_{t}-\rho u\cdot \nabla \t{u}+f]-\nabla (\nabla u\cdot \nabla u)+\Q(u_{t}\cdot\nabla u+u\cdot \nabla u_{t}),$$
% which implies
 %$$(\t{u})_{t}-\Delta \t{u}=\p[-(\rho-1)(\t{u})_{t}-\rho u\cdot \nabla \t{u}+f]-\nabla (\nabla u\cdot \nabla u)+\Q(u_{t}\cdot\nabla u+u\cdot \nabla u_{t}).$$
 %where $\p$ and $\Q$ are Helmholtz projectors. 
the following  equation for $t\t{u}$:  
\begin{equation}\label{eq:ttu}\rho(t\t{u})_t-\Delta(t\t{u})+\nabla(t\t{P}) = -t\rho u\cdot\nabla\t{u}+\rho\t{u} + tf.\end{equation}
Since  $\div \t u \not=0,$ one cannot apply  
directly Proposition \ref{propregularity}. 
Now, let us introduce the Helmholtz projectors on divergence free
and gradient like vector-fields, namely,
\begin{equation}\label{eq:PQ}\p:={\rm Id}+\nabla (-\Delta)^{-1}\div\andf \q:=-\nabla (-\Delta)^{-1}\div
\end{equation}
We  observe that  
$$\nabla(t\t{P})=\q\Bigl( -t\rho u\cdot\nabla\t{u}+\rho\t{u} + tf - \rho(t\t{u})_t+\Delta(t\t{u})\Bigr)\cdotp$$
Hence, reverting to \eqref{eq:ttu} implies that
\begin{equation}\label{eq:nnnn}
\rho(t\t{u})_t-\Delta(t\t{u}) =\p\bigl(\rho\dot u+tf -t\rho u\cdot\nabla\dot u\bigr) +\q\bigl(\rho(t\t{u})_t-\Delta(t\t{u}) \bigr).
\end{equation}
Using the fact that $\div u=0,$ we easily get
 \begin{equation}\label{divtu}\div \t{u}=\underset{1\leq i,j\leq d}{\sum}\d_{i} u^{j}\d_{j} u^{i}=
 {\rm Tr}(\nabla u \cdot \nabla u),\end{equation}
whence  
$$\q(t\Delta\dot u) =  t\nabla{\rm Tr} (\nabla u\cdot\nabla u)$$
and since
$$\q((\rho(t\dot u)_t) =\q\bigl( (\rho-1)(t\dot u)_t +\dot u + tu\cdot\nabla u_t +tu_t\cdot\nabla u \bigr),$$ 
we get  in the end, 
\begin{multline}\label{eqttu}(t\t{u})_{t}-\Delta t \t{u} 
 =\p[ (1-\rho)(t\t{u})_{t}-t\rho u\cdot \nabla \t{u}+\rho\dot u+tf]\\
 +\Q(\dot u+ t u_{t}\cdot\nabla u+t u\cdot \nabla u_{t})- \nabla{\rm Tr} (t\nabla u\cdot\nabla u).\end{multline}
At this point, we use the maximal regularity estimate \emph{for the heat equation} stated in \cite[Prop. 2.1]{DM2} 
as well as the continuity of $\p$ and $\q$ on $L_p$ to  conclude that
$$\displaylines{
    \norm{t\t{u}}_{L_{\infty}(0,T;\B^{-1+2/p}_{p,1})}+\norm{(t\t{u})_{t}, \nabla^{2} t\t{u}}_{L_{q,1}(0,T;L_{p})}\hfill\cr\hfill
\lesssim \norm{(1-\rho)(t\t{u})_{t}-t\rho u\cdot \nabla \t{u}+\rho\t{u}+tf}_{L_{q,1}(0,T;L_{p})}\hfill\cr\hfill
+\norm{\dot u + t u_{t}\cdot\nabla u+t u\cdot \nabla u_{t}}_{L_{q,1}(0,T;L_{p})}
+\norm{t\nabla u\otimes \nabla^2u}_{L_{q,1}(0,T;L_{p})}.}$$
As usual, owing to \eqref{eq:smallrho}, the first term in the right-hand side may be absorbed by the left-hand side. 
Now, using \eqref{eq:embed} and the definition of $f$ in 
\eqref{doubleD}, we get
 \begin{equation*}
    \begin{aligned}
    \norm{t f}_{L_{q,1}(0,T;L_{p})}    &\leq C\norm{t\nabla u}_{L_{\infty}(0,T\times\R^2)}
    (\norm{ \nabla^{2}u}_{L_{q,1}(0,T;L_{p}))}+\norm{ \nabla P}_{L_{q,1}(0,T;L_{p})})\\
    &\leq C \norm{tu}_{L_{\infty}(0,T;\B^{1+2/m}_{m,1})}
    (\norm{ \nabla^{2}u}_{L_{q,1}(0,T;L_{p})}+\norm{ \nabla P}_{L_{q,1}(0,T;L_{p})}).
     \end{aligned}
\end{equation*}
Next,  $\norm{\t{u}}_{L_{q,1}(0,T;L_{p})}$ may be bounded according to 
 Inequality \eqref{eq:dotu2}. 
 Finally, we have 
$$\begin{aligned}\norm{ t \rho u\cdot \nabla \t{u}}_{L_{q,1}(0,T;L_{p})}&\leq C \norm{ t \nabla \t{u}}_{L_{2}(0,T\times \R^{2})}\norm{u}_{L_{s,1}(0,T;L_{m})},\\
\norm{ t u\cdot \nabla u_{t}}_{L_{q,1}(0,T;L_{p})}&\leq C \norm{t\nabla u_{t}}_{L_{2}(0,T\times\R^{2})}\norm{u}_{L_{s,1}(0,T;L_{m})},\\
\norm{t u_{t}\cdot \nabla u}_{L_{q,1}(0,T;L_{p})}&\leq C \norm{tu_{t}}_{L_{s,1}(0,T;L_{m})}\norm{\nabla u}_{L_{2}(0,T\times\R^{2})},\\
    \norm{t\nabla u\otimes\nabla^2 u}_{L_{q,1}(0,T;L_{p})}&\leq C\norm{t \nabla u}_{L_{\infty}(0,T\times \R^{2})}\norm{\nabla^{2} u}_{L_{q,1}(0,T;L_{p})}\\
    &\leq C\norm{t u}_{L_{\infty}(0,T;\B^{1+2/m}_{m,1})}\norm{\nabla^{2} u}_{L_{q,1}(0,T;L_{p})}.
    \end{aligned}$$
Then,  putting all together with Proposition \ref{prop0d2}, Inequality \eqref{eq:dotu2},   Proposition \ref{prop1} and  Proposition \ref{prop3}, we discover that
$$\norm{t\t{u}}_{L_{\infty}(0,T;
\B^{-1+2/p}_{p,1})}+\norm{(t\t{u})_{t}, t\nabla^{2} \t{u}}_{L_{q,1}(0,T;L_{p})}\leq C_0\cdotp$$
%$$\displaylines{\quad\norm{tu_{t}}_{L_{\infty}(0,T;
%\B^{-1+2/p}_{p,1})}+\norm{(tu_{t})_{t}, t\nabla^{2} u_{t}}_{L_{q,1}(0,T;L_{p})}
%\leq C\exp(C_{0})\norm{u_{0}}^{2}_{\B^{-1+2/p}_{p,1}}\hfill\cr\hfill+\norm{u_{0}}_{\B^{-1+2/p}_{p,1}}\exp\left(C(\norm{u_{0}}^{s}_{\B^{-1+2/p}_{p,1}}+\norm{u_{0}}^{2}_{\B^{-1+2/p}_{p,1}})\right).}$$
Finally, Inequality \eqref{eq:lil2esd2} enables us to conclude that
 $$\|t\dot u\|_{L_2(0,T;L_\infty)}\leq \norm{t\dot u}^{\frac{2-q}{2}}_{L_{\infty}(0,T;\B^{-1+2/p}_{p,1})}\norm{t\nabla^{2}\dot u}^{\frac 2q}_{L_{q,1}(0,T;L_{p})}\leq C_0,$$ which completes the proof.
 \end{proof}
 \medbreak
 We end this section by stating higher order
 energy type time weighted estimates (that are not required
 for proving the uniqueness).  
 \begin{proposition}\label{prop4}
 Under the assumptions of  Proposition \ref{prop1},  we have  for all $t\in[0,T],$
 $$\underset{\tau\in[0,t]}\sup\norm{\tau^{3/2}\nabla \t{u}}^{2}_{L_{2}}+
\int_0^t\norm{\tau^{3/2}\nabla^{2} \t{u},t^{3/2}\nabla \t{P},t^{3/2}\sqrt{\rho}\ddot u}^{2}_{L_{2}}\,d\tau\leq C_0
%\bigl(\norm{u_{0}}^{4+1/s}_{\B^{-1+2/p}_{p,1}}+\norm{u_{0}}^{3+1/s}_{\B^{-1+2/p}_{p,1}}+\norm{u_{0}}^{2}_{\B^{-1+2/p}_{p,1}}\bigr)\cdotp}
$$
where  $C_0$ depends  only on $p$ and on $\norm{u_{0}}_{\B^{-1+2/p}_{p,1}}.$ 
 \end{proposition}
 \begin{proof}
 Taking the $L_{2}(\R^2;\R^2)$ inner product of \eqref{doubleD} with $t^{3}\ddot u$ then integrating on $[0,t]$ yields
     \begin{multline}\label{eq:nabladotu}
     \frac{t^{3}}{2}\int_{\R^{2}}\abs{\nabla \t{u}}^{2}\,dx+\int_{0}^{t}\!\!\int_{\R^{2}} \tau^{3} \rho \abs{\ddot u}^{2}\,dx\,d\tau
    = \int_{0}^{t}\!\!\int_{\R^{2}}\frac{3\tau^{2}}{2}\abs{\nabla \t{u}}^{2}\,dx\,d\tau\\+\int_{0}^{t}\!\!\int_{\R^{2}}\Delta \t{u}\cdot \tau^{3} u\cdot \nabla \t{u}\,dx\,d\tau-\int_{0}^{t}\!\!\int_{\R^{2}} \nabla \t{P} \cdot \bigl(\tau^{3}u\cdot \nabla \t{u}\bigr)\,dx\,d\tau\\
    \qquad+\int_{0}^{t}\!\!\int_{\R^{2}} \nabla \t{P} \cdot \bigl(\tau^{3}u_{t}\cdot \nabla u\bigr)\,dx\,d\tau
    +\int_{0}^{t}\!\!\int_{\R^{2}} \nabla \t{P} \cdot \bigl(\tau^{3}u\cdot \nabla u_t\bigr)\,dx\,d\tau\\
    +\int_{0}^{t}\int_{\R^{2}}f\cdot \tau^{3} \ddot u\,dx\,d\tau=:\underset{1\leq k\leq 6}{J_{k}}.
     \end{multline}
      In order to bound  $J_{2},J_{3},J_{4},J_{5},$ we proceed  as follows: 
 \begin{equation*}
     \begin{aligned}
     J_{2}&=\int_{0}^{t}\!\!\int_{\R^{2}}\Delta \t{u}\cdot \tau^{3} u\cdot \nabla \t{u}\,dx\,d\tau\\
     %&\leq \int_{0}^{t} \norm{\tau^{3/2} \nabla^{2} \t{u}}_{L_{2}}\norm{\tau \nabla \t{u}}_{L_{2}}\norm{\tau^{1/2}u}_{L_{\infty}}\,d\tau\\
     &\leq \norm{\tau^{3/2} \nabla^{2} \t{u}}_{L_{2}(0,t\times\R^2)}\norm{\tau \nabla \t{u}}_{L_{2}(0,t\times\R^2)}\norm{\tau^{1/2}u}_{L_{\infty}(0,t\times\R^{2})},
     \end{aligned}
 \end{equation*}
  \begin{equation*}
     \begin{aligned}
     J_{3}&=-\int_{0}^{t}\!\!\int_{\R^{2}} \nabla \t{P} \cdot 
     (\tau^{3}u\cdot \nabla \t{u})\,dx\,d\tau\\
    %&\leq \int_{0}^{t} \norm{\tau^{3/2} \nabla \t{P}}_{L_{2}}\norm{\tau \nabla \t{u}}_{L_{2}}\tau^{1/2}\norm{u}_{L_{\infty}}\,d\tau\\
     &\leq \norm{\tau^{3/2} \nabla \t{P}}_{L_{2}(0,t\times\R^2)}\norm{\tau \nabla \t{u}}_{L_{2}(0,t\times\R^2)}
     \norm{\tau^{1/2}u}_{L_{\infty}(0,T\times\R^{2})},
     \end{aligned}
 \end{equation*}
 \begin{equation*}
     \begin{aligned}
     J_{4}&=\int_{0}^{t}\!\!\int_{\R^{2}} \nabla \t{P} \cdot (\tau^{3}u_{\tau}\cdot \nabla u)\,dx\,d\tau\\
     %&\leq \int_{0}^{t} \norm{\tau^{3/2} \nabla \t{P}}_{L_{2}} \norm{\tau u_{\tau}}_{L_{2}}\norm{\tau^{1/2}\nabla u}_{L_{\infty}}\,d\tau\\
     &\leq \norm{\tau^{3/2} \nabla \t{P}}_{L_2(0,t\times\R^2)}\norm{\tau u_{\tau}}_{L_{\infty}(0,t;L_{2})}\norm{\tau^{1/2}\nabla u}_{L_{2}(0,t;L_{\infty})},
     \end{aligned}
 \end{equation*}
 \begin{equation*}
     \begin{aligned}
     J_{5}&=\int_{0}^{t}\!\!\int_{\R^{2}} \nabla \t{P} \cdot \bigl(\tau^{3}u\cdot \nabla u_{t}\bigr)dx\,d\tau\\
     %&\leq \int_{0}^{t} \norm{\tau^{3/2} \nabla \t{P}}_{L_{2}} \norm{\tau \nabla u_{t}}_{L_{2}}\norm{\tau^{1/2}u}_{L_{\infty}}\,d\tau\\
     &\leq \norm{\tau^{3/2} \nabla \t{P}}_{L_2(0,t\times\R^2)} \norm{\tau \nabla u_{t}}_{L_{2}(0,t\times\R^2)}
     \norm{\tau^{1/2}u}_{L_{\infty}(0,t\times\R^{2})}.
     \end{aligned}
 \end{equation*}
 At this point, we have to explain how to bound $t^{3/2}\nabla^{2} \t{u}$ and 
 $t^{3/2}\nabla \t{P}$ in $L_2(0,T\times\R^2).$
 Observe that \eqref{doubleD} and \eqref{divtu} 
 ensure that
 \begin{equation}\label{eq:dotP}
\nabla \t{P}=\Q f+\Q\bigl(\rho \ddot u\bigr)+\nabla {\rm Tr}(\nabla u \cdot\nabla u).
\end{equation} 
Hence, owing to  the continuity of $\Q$ on $L_2,$ 
 %\begin{equation}\label{ntpl2}
 %\norm{\nabla \t{P}}_{L_{2}}\leq %C\bigl(\norm{f}_{L_{2}}+\norm{\sqrt{\rho}\ddot u}_{L_{2}}+ \norm{\nabla %u\otimes\nabla^2 u}_{L_{2}}\bigr)\cdotp \end{equation}
 we have  for all $t\in[0,T],$ 
 $$\|t^{3/2}\nabla \t{P}(t)\|_{L_2} \lesssim \|\sqrt\rho\, t^{3/2}\ddot u(t)\|_{L_2}+\|t^{3/2}(\nabla u\otimes\nabla^2u)(t)\|_{L_2}
 +\|t^{3/2}f(t)\|_{L_2}.$$
 Hence, since
 $$t^{3/2}\Delta\t{u} =  t^{3/2}\nabla\t{P} + \rho t^{3/2}\ddot u + t^{3/2}\Delta u\cdot \nabla u + 2t^{3/2}\nabla u\cdot\nabla^2u
 -t^{3/2}\nabla u\cdot\nabla P,$$
we easily get 
 $$\begin{aligned}
    \norm{t^{3/2}\nabla^{2} \t{u},t^{3/2}\nabla \t{P}}_{L_{2}(0,t\times\R^2)}
    &\lesssim\norm{t^{3/2}\sqrt{\rho}\ddot u}_{L_{2}(0,t\times\R^2)}
 +\norm{t^{3/2}\nabla u\otimes \nabla^{2}u}_{L_2(0,t\times\R^2)}\\
 &\hspace{4cm}+\norm{t^{3/2}\nabla u\cdot \nabla P}_{L_{2}(0,t\times\R^2)}\\
   &\lesssim\norm{t^{3/2}\sqrt{\rho}\ddot u}_{L_{2}(0,t\times\R^2)}\\
   &\quad\qquad+\norm{t\nabla^{2} u, t\nabla P}_{L_{\infty}(0,t;L_{2})}\norm{t^{1/2}\nabla u}_{L_2(0,t;L_{\infty})}.
    \end{aligned}$$
 Thanks to  Corollary \ref{coro1d2} and Proposition \ref{prop3}, we thus end up with     
    \begin{equation}\label{tn2tues}
     \norm{t^{3/2}\nabla^{2} \t{u},t^{3/2}\nabla \t{P}}_{L_{2}(0,t\times\R^2)}\lesssim\norm{t^{3/2}\sqrt{\rho}\ddot u}_{L_{2}(0,t\times\R^2)}+C_{0}.   \end{equation}
Reverting to the above inequalities for $J_2$ to $J_5$ and taking
advantage of  Corollary \ref{coro1d2}, Proposition \ref{prop2} and  Proposition \ref{prop3},
we conclude that there exists  some constant   $C_0$ depending  only on $p$ and on $\norm{u_{0}}_{\B^{-1+2/p}_{p,1}},$
and such that 
 \begin{eqnarray}\label{esj2345}
    \sum_{k=2}^5J_{k} &\!\!\!\leq\!\!\!& C_0 \left(\norm{t^{3/2}\sqrt{\rho}\,\ddot u}_{L_{2}(0,t\times\R^2)}+C_{0}\right)\nonumber\\
    &\!\!\!\leq\!\!\!& \frac{1}{4}\norm{t^{3/2}\sqrt{\rho}\,\ddot u}^{2}_{L_{2}(0,t\times\R^2)}+2C_{0}^2.
     \end{eqnarray}
For $J_{6},$ we write that
\begin{equation*}
    \begin{aligned}
    J_{6}&=\int_{0}^{t}\!\!\int_{\R^{2}}f \cdot \tau^{3} \ddot u\,dx\,d\tau\\
    &=\int_{0}^{t}\!\!\int_{\R^{2}}
    (-\Delta u\cdot \nabla u-2\nabla u \cdot \nabla^{2} u+\nabla u\cdot \nabla P)\cdot \tau^{3} \ddot u\,dx\,d\tau\\
    %&\lesssim \int_{0}^{t} \norm{\tau^{3/2} \ddot u}_{L_{2}}\norm{\tau^{1/2} \nabla u}_{L_{\infty}}( \norm{\tau \nabla^{2}u}_{L_{2}}+ \norm{\tau \nabla P}_{L_{2}})\,d\tau\\
    &\lesssim\norm{\tau^{3/2} \ddot u}_{L_{2}(0,t\times\R^2)} 
    \norm{\tau^{1/2} \nabla u}_{L_2(0,t;L_{\infty})}\bigl(\norm{\tau \nabla^{2}u}_{L_{\infty}(0,t;L_{2})}
    +\norm{\tau \nabla P}_{L_{\infty}(0,t;L_{2})}\bigr),
    \end{aligned}
\end{equation*}
which along with Proposition \ref{prop3}, Corollary \ref{coro1d2} and \eqref{eq:smallrho} gives
$$J_{6}\leq C_0\norm{\tau^{3/2} \sqrt\rho\ddot u}_{L_{2}(0,t\times\R^2)} \leq 
\frac14 \norm{\tau^{3/2} \sqrt\rho \ddot u}^{2}_{L_{2}(0,t\times\R^2)} +2C_{0}^2.$$
Inserting  the above inequality and  \eqref{esj2345} in
\eqref{eq:nabladotu}, we get
 $$t^{3}\int_{\R^{2}}\abs{\nabla \t{u}}^{2}\,dx+\int_{0}^{t}\!\!\int_{\R^{2}} \tau^{3} \rho \abs{\ddot u}^{2}\,dx\,d\tau
      \leq 3\int_{0}^{t}\!\!\int_{\R^{2}}\tau^{2}\abs{\nabla \t{u}}^{2}\,dxd\tau+C_{0}$$
 which, by virtue of  Proposition \ref{prop3}, completes the proof. 
  \end{proof}

%%%%%%%%%%%%%%%%%%%%%%%%%%%%%%%%%%%
 
 \section{Estimates  in the three-dimensional case \label{section3} }
 Here we establish  the inequalities 
 that are needed to prove Theorem \ref{them1d3}. 
 The first two propositions are required for proving the existence
 of a global solution, while the last one is  needed for uniqueness.
 \begin{proposition}\label{propcsd3}
 Let $(\rho,u)$ be a smooth solution of $(INS)$ on $[0,T]\times\R^3,$
 with $u$ sufficiently decaying at infinity and $\rho$ such that
 \begin{equation}\label{eq:smallrho3}
 \sup_{t\in[0,T]}\|\rho(t)-1\|_{L_\infty(\R^3)}\leq c\ll1.
 \end{equation}
  Then, for all indices $1<m,p,q,s<\infty$ satisfying
 \begin{equation}\label{eq:relation}
 \frac  3p+\frac2q=3\!\andf\!\frac3m+\frac2s=1,\with
 p<m<\infty\andf q<s<\infty,\end{equation}
 the following
 inequalities hold true:
 \begin{multline}\label{eq:u3d1}
\mu^{\frac3{2p}-\frac12}\norm{u}_{L_{\infty}(0,T;
\B^{-1+3/p}_{p,1}(\R^{3}))}
+\mu^{\frac3{2p}-\frac12+\frac1s}\norm{ u}_{L_{s,1}(0,T;L_{m}(\R^{3}))}\\+\norm{\t{u}, u_{t}, \mu\nabla^{2} u,\nabla P}_{L_{q,1}(0,T;L_{p}( \R^{3}))}\leq C\mu^{\frac3{2p}-\frac12} \norm{u_{0}}_{\B^{-1+3/p}_{p,1}(\R^{3})},\end{multline}
\begin{equation}\label{eq:u3d2}
\andf\mu^{\frac12}\norm{u}_{L_2(0,T;L_{\infty}(\R^{3}))}\leq C  \norm{u_{0}}_{\B^{-1+3/p}_{p,1}(\R^{3})}\cdotp\qquad\qquad\end{equation}
 \end{proposition}
 \begin{proof}
 For notational simplicity, we omit to specify the dependence of the norms with respect to $\R^3$ in the proof. As usual, we only consider 
 the case $\mu=1.$
 Now,  applying  Proposition \ref{propregularity}   to System \eqref{s4e1} yields
\begin{multline}\label{esb3d}
 \norm{u}_{L_{\infty}(0,T;\B^{-1+3/p}_{p,1})}+\norm{u_{t}, \nabla^{2} u,\nabla P}_{L_{q,1}(0,T;L_{p})}+\|u\|_{L_{s,1}(0,T;L_{m})}
 \\
\leq C\bigl(\norm{u_{0}}_{\B^{-1+3/p}_{p,1}}+\norm{(\rho-1)u_{t}
+\rho u\cdot \nabla u}_{L_{q,1}(0,T;L_{p})}\bigr)\cdotp \end{multline}
By H\"older inequality, we have
$$\displaylines{
\norm{(\rho-1)u_{t}+\rho u\cdot \nabla u}_{L_{q,1}(0,T;L_{p})}\hfill\cr\hfill\leq
    \norm{\rho-1}_{L_{\infty}(0,T\times\R^{3})}\norm{u_{t}}_{L_{q,1}(0,T;L_{p})}
+\norm{\rho}_{L_{\infty}(0,T\times\R^{3})}\norm{u\cdot \nabla u}_{L_{q,1}(0,T;L_{p})}.}$$
Owing to \eqref{eq:smallrho3}, the first term  can be absorbed by the left-hand side of \eqref{esb3d}. 
 For  term $\norm{u\cdot \nabla u}_{L_{q,1}(0,T;L_{p}( \R^{3}))}$, by embedding 
 \begin{equation}\label{eq:embedL3}
 \B^{-1+3/p}_{p,1}(\R^{3})\hookrightarrow L_{3}(\R^{3})\end{equation} and 
 \begin{equation}\label{eq:W1p}
 \W^{1}_{p}(\R^{3})\hookrightarrow L_{p^{*}}(\R^{3})
\with \frac{1}{p^{*}}=\frac{1}{p}-\frac{1}{3},\end{equation} we obtain
\begin{equation*}
\begin{aligned}
\norm{u\cdot \nabla u}_{L_{q,1}(0,T;L_{p})}&\leq \norm{u}_{L_{\infty}(0,T;L_{3})}\norm{\nabla u}_{L_{q,1}(0,T;L_{p^{*}})}\\
&\lesssim \norm{u}_{L_{\infty}(0,T;\B^{-1+3/p}_{p,1})}\norm{\nabla^{2} u}_{L_{q,1}(0,T;L_{p})}.
\end{aligned}
\end{equation*}
Denoting  $\Phi_0:= \|u_0\|_{\B^{-1+3/p}_{p,1}}$ and
$$\Phi:=\norm{u}_{L_{\infty}(0,T;
\B^{-1+3/p}_{p,1})}+\norm{u_{t}, \nabla^{2} u,\nabla P}_{L_{q,1}(0,T;L_{p})}+
\|u\|_{L_{s,1}(0,T;L_{m})},$$
 we can conclude that 
$$\Phi\leq C(\Phi_{0}+\Phi^{2}).$$
Hence, if \begin{equation}\label{ini2d3}
   4C\Phi_{0} < 1, \end{equation}
then one can assert that  \begin{equation}\label{eq:Phi}
\Phi\leq 2 \Phi_{0}.\end{equation}
Clearly, $\t{u}$ satisfies the same inequality
since $\Phi$ is small and, by H\"older inequality, 
\begin{equation*}
\norm{\t{u}}_{L_{q,1}(0,T;L_{p})} \leq  \norm{u_{t}}_{L_{q,1}(0,T;L_{p})}+
 \norm{u}_{L_{\infty}(0,T;\B^{-1+3/p}_{p,1})}\norm{\nabla^{2} u}_{L_{q,1}(0,T;L_{p})}\leq C\Phi(1+\Phi).
\end{equation*} 
Finally, 
as a consequence of Gagliardo-Nirenberg inequality and embedding,  we have:
 \begin{equation}\label{eq:lil2esd3}
 \norm{z}_{L_{\infty}}\lesssim \norm{z}^{1-q/2}_{L_{3}}\norm{\nabla ^{2}z}^{q/2}_{L_{p}}\lesssim \norm{z}^{1-q/2}_{\B^{-1+3/p}_{p,1}}\norm{\nabla ^{2}z}^{q/2}_{L_{p}},
 \end{equation}
whence
\begin{equation}\label{esulinfty3d}
    \begin{aligned}
    \int_{0}^{T}\norm{u}^{2}_{L_{\infty}}\,dt &\leq C\int_{0}^{T}\norm{u}^{2-q}_{\B^{-1+3/p}_{p,1}}\norm{\nabla^{2}u}^{q}_{L_{p}}\,d\tau\\
    &\leq C\norm{u}^{2-q}_{L_{\infty}(0,T;\B^{-1+3/p}_{p,1})}\norm{\nabla^{2}u}^{q}_{L_{q,1}(0,T;L_{p})}\\
    &\leq C \Phi^{2}\cdotp
    \end{aligned}
\end{equation}
Owing to \eqref{eq:Phi}, this yields \eqref{eq:u3d2}. 
\end{proof}

\begin{proposition}\label{prop1d3}
Under the  assumptions Proposition \ref{propcsd3}, we have
$$\mu\norm{tu}_{L_{\infty}(0,T;\B^{1+3/m}_{m,1}(\R^{3}))}
+\mu^{\frac1s}\norm{(tu)_{t}, \mu\nabla^{2}(tu),\nabla(tP)}_{L_{s,1}(0,T;L_{m}(\R^{3}))}\leq C \|u_0\|_{\B^{-1+3/p}_{p,1}(\R^3)}.$$
Moreover, the following inequalities hold true: 
$$\mu\!\int_{0}^{T}\!\! \norm{\nabla u}_{L_{\infty}(\R^{3})}\,dt\leq C \|u_0\|_{\B^{-1+3/p}_{p,1}(\R^{3})}\!\andf\! \mu\!\int_{0}^{T} \!\!t\norm{\nabla u}^{2}_{L_{\infty}(\R^{3})}\,dt\leq C \|u_0\|_{\B^{-1+3/p}_{p,1}(\R^{3})}^2.$$
\end{proposition}
 \begin{proof} Assume that $\mu=1.$
Multiplying both sides of \eqref{s4e1}  by time $t$ yields
$$(tu)_{t}-\Delta (tu)+\nabla(tP)=(1-\rho)(tu)_{t}+\rho u- \rho  u\cdot \nabla  tu.$$
Then, taking  advantage of Proposition \ref{propregularity},  we get:
$$\displaylines{\quad
\norm{tu}_{L_{\infty}(0,T;\B^{1+3/m}_{m,1})}+\norm{(tu)_{t}, \nabla^{2}(tu),\nabla(tP)}_{L_{s,1}(0,T;L_{m})}
\hfill\cr\hfill
\lesssim\norm{\rho-1}_{L_{\infty}(0,T\times\R^{3})}\norm{(tu)_{t}}_{L_{s,1}(0,T;L_{m})}\hfill\cr\hfill
+\norm{\rho}_{L_{\infty}(0,T\times\R^{3})}
\left (\norm{u}_{L_{s,1}(0,T;L_{m})}+\norm{tu\cdot \nabla u}_{L_{s,1}(0,T;L_{m})}\right)\cdotp}$$
Owing to \eqref{eq:smallrho3}, the first term of the right-hand side may be bounded by the left-hand side, 
and we deduce from H\"older inequality and the embedding \begin{equation}\label{eq:embed3}\B^{3/m}_{m,1}(\R^{3})\hookrightarrow L_{\infty}(\R^{3})\end{equation} that
\begin{equation*}
\begin{aligned}
 \norm{tu\cdot \nabla u}_{L_{s,1}(0,T;L_{m})}
    &\leq \norm{ u}_{L_{s,1}(0,T;L_{m})} \norm{t\nabla u}_{L_{\infty}(0,T\times\R^{3})}\\
   & \leq C\norm{ u}_{L_{s,1}(0,T;L_{m})}\norm{tu}_{L_{\infty}(0,T;\B^{1+3/m}_{m,1})}.
\end{aligned}
\end{equation*}
Remember that Proposition \ref{propcsd3} allows to bound $u$ in $L_{s,1}(0,T;L_{m}(\R^{3}))$ by $\Phi_0.$  Hence,
setting  
$$\Pi:=\norm{tu}_{L_{\infty}(0,T;\B^{1+3/m}_{m,1})}
+\norm{(tu)_{t},\mu \nabla^{2}(tu),\nabla(tP)}_{L_{s,1}(0,T;L_{m})},$$ the above calculations imply that
   $\Pi\leq C(1+\Pi)\Phi_0$ and, as $ \Phi_0$ is small, this completes the proof of the first part of the proposition. 
   \medbreak
 Bounding $\nabla u$ relies on the following interpolation inequality
 (as \eqref{eq:relation} implies that $p<3<m$):
$$\norm{u}_{L_{\infty}(\R^{3})}\leq \norm{\nabla u}^{\frac{p(m-3)}{3(m-p)}}_{L_{p}(\R^{3})}\norm{\nabla u}^{\frac{m(3-p)}{3(m-p)}}_{L_{m}(\R^{3})}.$$
Hence, applying  H\"older inequality in Lorentz spaces
with exponents:
$$(p_1,r_1)=\biggl(\frac{3(m-p)}{m(3-p)},\infty\biggr),\!\quad
(p_2,r_2)=\biggl(\frac{3q(m-p)}{p(m-3)},\frac{p_2}q\biggr),\!\quad
(p_3,r_3)=\biggl(\frac{3s(m-p)}{m(3-p)},\frac{p_3}s\biggr),$$
 using the fact that   $t^{-\alpha}$ with $\alpha=m(3-p)/(3(m-p))$  
 is in $L_{1/\alpha,\infty}(\R_+),$ 
 \eqref{eq:u3d1} and the first inequality of Proposition \ref{prop1d3},
 we end up with
\begin{equation*}
    \begin{aligned}
    \int_{0}^{T}\norm{\nabla u}_{L_{\infty}}\,dt
    %&\leq \int_{0}^{\infty}\norm{\nabla^{2} u}^{\frac{p(m-3)}{3(m-p)}}_{L_{p}}\norm{\nabla^{2} u}^{\frac{m(3-p)}{3(m-p)}}_{L_{m}}\,dt\\
    &\leq \int_{0}^{T} t^{-\frac{m(3-p)}{3(m-p)}} \norm{\nabla^{2} u}^{\frac{p(m-3)}{3(m-p)}}_{L_{p}}\norm{t\nabla^{2} u}^{\frac{m(3-p)}{3(m-p)}}_{L_{m}}\,dt\\
    &\leq C \norm{\nabla^{2} u}^{\frac{p(m-3)}{3(m-p)}}_{L_{q,1}(0,T;L_{p})}\norm{t\nabla^{2} u}^{\frac{m(3-p)}{3(m-p)}}_{L_{s,1}(0,T;L_{m})}\\
    &\leq C \|u_0\|_{\B^{-1+3/p}_{p,1}}.
    \end{aligned}
\end{equation*}
Furthermore,  we deduce from \eqref{eq:embed3}  that
%\begin{equation*}\label{esnul2li}
  $$  \begin{aligned}
    \int_{0}^{T} t\norm{\nabla u}^{2}_{L_{\infty}}\,dt
    %&\leq  \int_{0}^{\infty}t\norm{\nabla u}_{L_{\infty}}\norm{\nabla u}_{L_{\infty}}\,dt\\
     &\leq  \int_{0}^{T}t\norm{\nabla u}_{\B^{3/m}_{m,1}}\norm{\nabla u}_{L_{\infty}}\,dt\\
     &\leq \norm{tu}_{L_{\infty}(0,T;\B^{1-3/m}_{m,1})}\int_{0}^{T}\norm{\nabla u}_{L_{\infty}}\,dt\\&\leq 
     C \|u_0\|_{\B^{-1+3/p}_{p,1}}^{2},
    \end{aligned}$$
%\end{equation*}
by virtue of the inequality we proved just before. 
\end{proof}

To prove the uniqueness, the following time weighted estimate  is required. 
 \begin{proposition}\label{prop2d3}
 Under the  assumptions of Proposition \ref{propcsd3}, it holds that
 \begin{multline}\label{eq:prop33}\mu^{\frac3{2p}-\frac12}\norm{t\t{u}}_{L_{\infty}(0,T;
\B^{-1+3/p}_{p,1})}+\norm{(t\t{u})_{t}, \mu t\nabla^{2} \t{u}}_{L_{q,1}(0,T;L_{p})}\\
+\mu^{\frac3{2p}-\frac12+\frac1s}\norm{ t\t{u}}_{L_{s,1}(0,T;L_{m})}\leq C\mu^{\frac3{2p}-\frac12}\|u_0\|_{\B^{-1+3/p}_{p,1}}.\end{multline}
%where $(m,p,q,s)$ are as in  Proposition \ref{propcsd3}.
\medbreak
Furthermore, we have 
\begin{equation}\label{eq:coro2d3}
 \mu^{\frac12}\norm{t\nabla \t{u}}_{L_{2}(0,T;L_{3}(\R^{3}))}+
 \mu^{\frac12}\norm{t\t{u}}_{L_{2}(0,T;L_{\infty}(\R^{3}))}\leq C \|u_0\|_{\B^{-1+3/p}_{p,1}(\R^{3})}. 
\end{equation}
 \end{proposition}
 \begin{proof}
 We know that $t\t{u}$ satisfies \eqref{eqttu} and we also observe, owing to $\div u=\div u_t=0,$ that
 $$\q(u\cdot\nabla u_t)=\q(u_t\cdot\nabla u).$$ 
  Hence, using the maximal regularity estimates 
 in Lorentz spaces for the heat equation 
 (cf \cite[Prop. 2.1]{DM2}) and  the continuity of the Helmholtz projectors on ${L_{q,1}(0,T;L_{p})},$  we
 get
 $$\displaylines{\norm{t\t{u}}_{L_{\infty}(0,T;
\B^{-1+2/p}_{p,1})}+\norm{(t\t{u})_{t}, t\nabla^{2}\t{u}}_{L_{q,1}(0,T;L_{p})}+\norm{ t\t{u}}_{L_{s,1}(0,T;L_{m})}\lesssim
\|(1-\rho)(t\t{u})_{t}\|_{L_{q,1}(0,T;L_{p})}\hfill\cr\hfill
+\|t\rho u\cdot \nabla \t{u}\|_{L_{q,1}(0,T;L_{p})}
+\|\rho\dot u\|_{L_{q,1}(0,T;L_{p})}+\|tf\|_{L_{q,1}(0,T;L_{p})}\hfill\cr\hfill
 +\|\dot u+ t u_{t}\cdot\nabla u\|_{L_{q,1}(0,T;L_{p})}
 + \|t\nabla{\rm Tr} (\nabla u\cdot\nabla u)\|_{L_{q,1}(0,T;L_{p})}.}$$ 
 Owing to \eqref{eq:smallrho3}, the first term in the right-hand side may be absorbed by the left-hand side, 
 and Proposition \ref{propcsd3} allows to bound  $\norm{\t{u}}_{L_{q,1}(0,T;L_{p})}.$
 Also recall that $$f=-\Delta u\cdot \nabla u-2\nabla u \cdot \nabla^{2} u+\nabla u\cdot \nabla P.$$ 
 Hence, thanks to \eqref{eq:embed3} and to Propositions \ref{propcsd3},
 \ref{prop1d3},
 $$    \begin{aligned}
    \norm{t f}_{L_{q,1}(0,T;L_{p})}\!+\!\|t\nabla{\rm Tr} (\nabla u\cdot\nabla u)\|_{L_{q,1}(0,T;L_{p})}  &\lesssim \norm{t\nabla u}_{L_{\infty}(0,T\times \R^{3})}    \norm{ \nabla^{2}u,\nabla P}_{L_{q,1}(0,T;L_{p})}\\
    &\lesssim  \norm{tu}_{L_{\infty}(0,T;\B^{1+3/m}_{m,1})}\norm{ \nabla^{2}u,\nabla P}_{L_{q,1}(0,T;L_{p})}\\
     &\lesssim \|u_0\|_{\B^{-1+3/p}_{p,1}}^2.
     \end{aligned}$$
     Using the H\"older inequality in Lorentz spaces,  the embeddings \eqref{eq:W1p} and  
    \eqref{eq:embedL3}, and Propositions \ref{propcsd3}, \ref{prop1d3}, we obtain
\begin{equation*}
    \begin{aligned}
    \norm{ t \rho u\cdot \nabla \t{u}}_{L_{q,1}(0,T;L_{p})}&\lesssim\norm{ t \nabla \t{u}}_{L_{q,1}(0,T;L_{p^{*}})}\norm{u}_{L_{\infty}(0,T;L_{3})}\\
    &\lesssim\norm{ t \nabla^{2} \t{u}}_{L_{q,1}(0,T;L_{p})} \|u_0\|_{\B^{-1+3/p}_{p,1}},
    \end{aligned}
\end{equation*}
\begin{equation*}
    \begin{aligned}
    \norm{t u_{t}\cdot \nabla u}_{L_{q,1}(0,T;L_{p})}&\leq  \norm{u_{t}}_{L_{q,1}(0,T;L_{p})}\norm{t\nabla u}_{L_{\infty}(0,T\times\R^{3})}\\
    &\lesssim\norm{u_{t}}_{L_{q,1}(0,T;L_{p})}\norm{t u}_{L_{\infty}(0,T;\B^{1+3/m}_{m,1})}\\
    &\lesssim \|u_0\|_{\B^{-1+3/p}_{p,1}}^2.
    \end{aligned}
\end{equation*}
Putting the above inequalities together, we conclude  that
 $$\displaylines{\norm{t\t{u}}_{L_{\infty}(0,T;
\B^{-1+3/p}_{p,1})}+\norm{(t\t{u})_{t}, t\nabla^{2}\t{u}}_{L_{q,1}(0,T;L_{p})}+\norm{ t\t{u}}_{L_{s,1}(0,T;L_{m})}\hfill\cr\hfill
\lesssim \|u_0\|_{\B^{-1+3/p}_{p,1}}^2
+\bigl(1+ \norm{ t \nabla^{2} \t{u}}_{L_{q,1}(0,T;L_{p})}\bigr) \|u_0\|_{\B^{-1+3/p}_{p,1}}.}$$
 Since $ \|u_0\|_{\B^{-1+3/p}_{p,1}}$ is small, we have
 \eqref{eq:prop33}.   
 \medbreak
 In order to prove Inequality \eqref{eq:coro2d3},
 let us first  consider  the case $3/2<p<3$ (which implies that $1<q<2$).  Then,  Inequality  
\eqref{eq:lil2esd3} ensures that 
$$\norm{\dot u}_{L_{\infty}}\leq C\norm{\dot u}^{1-\frac{q}{2}}_{L_{3}}\norm{\nabla^{2}\dot u}^{\frac{q}{2}}_{L_p}\leq C\norm{\dot u}^{1-\frac{q}{2}}_{\B^{-1+3/p}_{p,1}}\norm{\nabla^{2}\dot u}^{\frac{q}{2}}_{L_{p}}.$$
Consequently, $$\norm{t\t{u}}_{L_{2}(0,T;L_{\infty})}\leq C\norm{t\t{u}}^{1-\frac{q}{2}}_{L_{\infty}(0,T;\B^{-1+3/p}_{p,1})}\norm{\nabla^{2}(t\t{u})}^{\frac{q}{2}}_{L_{q}(0,T;L_p)},$$
Then, applying \eqref{eq:prop33}   gives
the second part of \eqref{eq:coro2d3}.
\medbreak
In order to complete the proof of  \eqref{eq:coro2d3},
it suffices to apply 
Proposition \ref{prop:for existence} with $r=3$ to $t\t{u}$ 
(keeping in mind that $-1+3/p=2-2/q$) then H\"older
inequality with respect to the time variable. 
In the end, as $p\in(3/2,3),$ we get
$$\norm{t\nabla \t{u}}_{L_{2}(0,T;L_{3})}\lesssim  \norm{t\t{u}}_{L_{\infty}(0,T;\B^{-1+3/p}_{p,1})}^\theta
\norm{t\nabla^{2}\t{u}}_{L_{q}(0,T;L_p)}^{1-\theta}
\with \theta=\frac{2p-3}{3p-3}\cdotp$$
Then, applying   the first part of the proposition gives the desired result.
\smallbreak
The case $1<p\leq 3/2$ reduces to the case we treated before since  $\B^{-1+\frac{3}{p}}_{p,1}\hookrightarrow \B^{-1+\frac{3}{p_{1}}}_{p_{1},1}$
for some $p_1\in(3/2,3).$ 
\end{proof}

%%%%%%%%%%%%%%%%%%%%%%%%%%%%%%%%%%%%%%%%%%%%%

 \section{Existence} \label{section4}
This  section is devoted to the proof of   existence of a global solution under our assumptions
(both in dimensions $2$ and $3$). 

As a first step, we shall smooth out the data so as  to apply prior results ensuring the existence of a sequence $(a^n,u^n,\nabla P^n)_{n\in\N}$ of  strong 
(relatively) smooth solutions to \eqref{rins2}. 
The estimates of Sections \ref{section2} 
and \ref{section3} will guarantee that  
the solution $(a^n,u^n,\nabla P^n)_{n\in\N}$ is global and  uniformly 
bounded in the expected spaces. 
In order to pass to the limit, we shall take advantage 
of compactness arguments.
A technical point is that  Lorentz spaces $L_{q,1}$ are \emph{nonreflexive},  so that one cannot directly use the classical results,  
like Aubin-Lions' lemma.
To overcome the difficulty, we shall  look at the approximate
solutions in the slightly larger (but reflexive) 
 space 
 $$\W^{2,1}_{p,r}(\R_+ \times \R^d):=
 \bigl\{u\in\cC_b(\R_+;\dot B^{2-2/r}_{p,r}(\R^d)\,:\,
 u_t,\nabla^2u\in L_r(\R_+;L_p(\R^d))\bigr\}
 $$ for some $ 1<r<\infty,$
 then check afterward that the constructed solution has  the desired
 regularity. 
 \medbreak
% From now on,  we drop  $\R^d$  in the notations for  the norms. 
As a first, let us smooth out the initial
 data $a_0$ and $u_0$ by means of non-negative mollifiers, to get a 
 sequence $(a_0^n, u_0^n)_{n\in{\mathbb N}}$ of smooth data such that
 %$u^n_{0} \in \B^{-1+d/p}_{p,r}$  for all $1<r<\infty,$ and 
\begin{equation}\label{eq:unifbound}
\norm{a^n_{0}}_{L_\infty}\leq  \norm{a_{0}}_{L_\infty},\quad \norm{u^n_{0}}_{\B^{-1+d/p}_{p,1}}\leq C \norm{u_{0}}_{\B^{-1+d/p}_{p,1}}
\end{equation}
with, in addition, $$a^{n}_{0}\rightharpoonup a_0  \quad \text{weak * in}\quad L_\infty \andf 
u_{0}^n \to u_{0} \quad \text{strongly   in}\quad \B^{-1+d/p}_{p,1}. $$ 

According to e.g. \cite{DR2004}, there exists $T>0$ such that System \eqref{rins2} supplemented
with initial data $(a_0^n, u_0^n)$ admits a unique smooth local solution  $(a^n,u^n,\nabla P^n)$ on $[0,T]\times\R^d.$
In particular, the energy balance is satisfied 
(in the cases where $u_0$ is in $L_2(\R^d)$), 
$a^n\in\cC_b([0,T]\times\R^d)$ and
$(u^n,\nabla P^n)$ is in the space 
 $$E^{p,r}_{T}=\bigl\{(u,\nabla P)\with  u\in \W^{2,1}_{p,r}(0,T\times \R^d) \andf \nabla P \in L_{r}(0,T;L_p)\bigr\}\quad\hbox{for all }\ r\geq1.$$
Let us denote by $T^n$ the maximal time of existence  of  $(a^n,u^n,\nabla P^n).$
 Since the calculations of the previous sections  just follow from  the properties of the heat flow and transport equation, basic functional analysis and integration 
 by parts, each $(a^{n},u^{n},\nabla P^n)$ satisfies the estimates 
 therein \emph{up to time $T^n$},  and thus 
 \begin{equation}\label{eq:anbound}
 \norm{a^{n}(t)}_{L_\infty}= \norm{a^{n}_0}_{L_\infty} \leq \norm{a_0}_{L_\infty} \quad \text{for all} \quad t\in [0,T^n)
 \end{equation}
and 
 \begin{equation}\label{eq:unlorentz}
 \norm{u^n}_{\W^{2,1}_{p,(q,1)}(0,T^n\times\R^d)}+\norm{\nabla P^n}_{L_{q,1}(0,T^n;L_p)}\leq C\norm{u^{n}_0}_{\B^{-1+d/p}_{p,1}}\leq C\norm{u_0}_{\B^{-1+d/p}_{p,1}}\cdotp
 \end{equation}
Furthermore, taking any $r\in(1,\infty)$
and 
applying Proposition \ref{propregularity} with $q=r$ to 
$$\d_{t} u^{n}-\Delta u^n+\nabla P^n=-a^n \d_t u^n-(1+a^n) u^n\cdot \nabla u^n,\qquad \div u^n=0,$$
yields for all $T<T^n,$
\begin{equation*}
\begin{aligned}
\norm{u^n, \nabla P^n}_{E^{p,r}_T}&\leq C(\norm{u^n_0}_{\B^{2-2/r}_{p,r}}+\norm{a^n \d_t u^n+(1+a^n) u^n\cdot \nabla u^n}_{L_{r}(0,T;L_p)})\\
 &\leq C (\norm{u^{n}_0}_{\B^{2-2/r}_{p,r}}+\norm{a^n}_{L_{\infty}(0,T\times \R^d)}\norm{\d_t u^n}_{L_{r}(0,T;L_p)}\\
&\qquad\qquad\qquad+(1+\norm{a^n}_{L_{\infty}(0,T\times \R^d)})\norm{u^n\cdot \nabla u^n}_{L_{r}(0,T;L_p)}).
\end{aligned}
\end{equation*}
In light of \eqref{inidr} or \eqref{ini1d3}, and \eqref{eq:anbound}, 
the above inequality becomes 
 $$\norm{u^n, \nabla P^n}_{E^{p,r}_T}\lesssim \norm{u^n_0}_{\B^{2-2/r}_{p,r}}+ \norm{u^n\cdot \nabla u^n}_{L_{r}(0,T;L_p)},$$
 whence
 \begin{equation*}
 \begin{aligned}
 \norm{u^n, \nabla P^n}^r_{E^{p,r}_T}&\lesssim  \norm{u^n_0}^{r}_{\B^{2-2/r}_{p,r}}+\int_{0}^{T} \norm{u^n\cdot \nabla u^n}^{r}_{L_p}\,dt\\
 &\lesssim \norm{u^n_0}^{r}_{\B^{2-2/r}_{p,q}}+\int_{0}^{T} \norm{u^n}^{r}_{L_{\frac{dr}{r-1}}}\norm{\nabla u^n}^{r}_{L_{\beta}}\,dt
 \with \frac{1}{\beta}+\frac{1}{d}-\frac{1}{dr}=\frac{1}{p}\cdotp
 \end{aligned}
 \end{equation*}
Combining with Proposition \ref{prop:for existence} and Young's inequality, one gets
$$\norm{u^n, \nabla P^n}^r_{E^{p,r}_T}\lesssim  \norm{u^n_0}^r_{\B^{2-2/r}_{p,r}} +\eps \int_0^T 
    \norm{\nabla^{2} u^n}^{r}_{L_p}\,dt+C_\eps \int_0^T \norm{u^n}^{2r}_{L_{\frac{dr}{r-1}}}
    \norm{u^n}^{r}_{\B^{2-2/r}_{p,r}}\,dt\cdotp$$
Then, taking $\eps$ small enough and using Gronwall's inequality yields
\begin{equation}\label{eq:unest}
\norm{u^n, \nabla P^n}^r_{E^{p,r}_T}\leq C \norm{u^n_0}^r_{\B^{2-2/r}_{p,r}} \exp\biggl(C\int_0^T \norm{u^n}^{2r}_{L_{\frac{dr}{r-1}}} \,dt \biggr) \cdotp
\end{equation}
In the end, Gagliardo-Nirenberg inequality and embedding give
$$\norm{u^n}_{L_{\frac{dr}{r-1}}}\leq \norm{u^n}^{\frac 1r}_{L_{\infty}}\norm{u^n}^{1-\frac 1r}_{\B^{-1+d/p}_{p,1}},$$
which implies that
\begin{equation}\label{eq:new}\int_0^T \norm{u^n}^{2r}_{L_{\frac{dr}{r-1}}} \,dt 
\leq \norm{u^n}^2_{L_2(0,T;L_\infty)}\norm{u^n}^{2(r-1)}_{L_\infty(0,T;\B^{-1+d/p}_{p,1})}.\end{equation}
Now, we deduce from 
 Proposition \ref{prop0d2} (case $d=2$) or 
 Proposition \ref{propcsd3} (case $d=3$) 
 that the two factors of the right-hand side of \eqref{eq:new} 
 are bounded by $\|u_0\|_{\dot B^{-1+d/p}_{p,1}}.$
 Hence,  reverting to \eqref{eq:unest} and 
 using a classical continuation argument allows
 to conclude that  the solution is global and
 belongs to all spaces 
 $W^{2,1}_{p,r}(\R_+\times\R^d)$ with $1<r<\infty.$
 Furthermore, since the solution is smooth and
 \eqref{eq:unifbound} is satisfied, all the a priori  estimates
 of Sections \ref{section2} and \ref{section3}
 are satisfied uniformly with respect to $n.$
 
 In particular, 
$(u^{n},\nabla P^n)_{n\in\N}$ is  bounded in 
$E^{p,q}_T$ for all $T\geq0.$
%$\W^{2,1}_{p,r}(\R_+\times \times\R^{d}).$
%This ensures that $T^n=\infty$ and, finally, 
%that $(u^{n})_{n\in\N}$ is bounded in the space $E^{p,r}$ (that is %$E^{p,r}_T$ with $T=\infty$)
%with a bound  that only depends on $\norm{u_0}_{\B^{-1+d/p}_{p,1}}.$
%\smallbreak
 This,  together with   \eqref{eq:anbound} ensures that 
 there exists a subsequence, still denoted by $(a^{n},u^{n},\nabla P^{n})_{n\in{\mathbb N}},$ and $(a,u,\nabla P)$ with
 $$a\in L_{\infty}(\R_{+}\times\R^{d}),\quad \nabla P\in L_{q}(\R_{+};L_{p}(\R^{d}))\andf u\in \W^{2,1}_{p,q}(\R_{+}\times \R^d)$$
 such that 
 \begin{equation}\label{eq:weak limit}
 \begin{aligned}
 &a^{n}\rightharpoonup a\quad \text{weak *  in } \quad L_{\infty}(\R_{+}\times\R^{d}),\\
  &u^{n}\rightharpoonup u\quad \text{weak *  in } \quad L_{\infty}(\R_{+};\dot B^{2-2/q}_{p,q}),\\
 &(\d_{t}u^{n}, \nabla^{2}u^{n})\rightharpoonup (\d_{t}u, \nabla^{2}u)\quad  \text{weakly in} \quad 
 L_{q}(\R_{+};L_{p}),\\
 &\nabla P^{n} \rightharpoonup \nabla P\quad  \text{weakly in} \quad  
 L_{q}(\R_{+};L_{p}).
 \end{aligned}
 \end{equation}
% Since our proof is valid for all $r\in(1,\infty),$ one 
% can take $r=q_1$ or $r=q_2$ with $1<q_1<q<q_2<\infty$ such that
% $\frac1q=\frac12(\frac1{q_1}+\frac1{q_2})\cdotp$
% Hence,  using the following results of interpolation:
% $$\W^{2,1}_{p,(q,1)}(\R_{+}\times %\R^d)=(\W^{2,1}_{p,q_1}(\R_{+}\times %\R^d),\W^{2,1}_{p,q_2}(\R_{+}\times \R^d))_{\frac12,1},$$ 
% and (see  \cite[Th2:1.18.6]{HT}): %$$L_{q,1}(\R_{+};L_p)=(L_{q_1}(\R_+;L_p), %L_{q_2}(\R_+;L_p))_{\frac12,1},$$
Furthermore, as all the spaces under consideration 
in the previous sections have the Fatou
property, the estimates proved therein
as still valid. For example, 
one gets
  $$u\in \W^{2,1}_{p,(q,1)}(\R_{+}\times \R^d)\quad \text{and} \quad \nabla P\in L_{q,1}(\R_+;L_p (\R^d)).$$
Note that the fact  that $(\d_t u^n)_{n\in {\mathbb N}}$ is bounded in $L_q(\R_{+};L_p)$ enables us 
to take advantage of  Arzel\`a-Ascoli Theorem 
in order to get strong convergence 
results for $u$ like, 
for instance, for all small enough $\varepsilon>0,$  
\begin{equation}\label{eq:strong limit}
\begin{aligned}
&u^{n}\rightarrow u \ \text{strongly in}\  L_{\infty,loc}(\R_{+};L_{d-\varepsilon,loc}(\R^d)),\\
&\nabla u^{n}\rightarrow \nabla u \ \text{strongly  in }\ L_{q,loc}(\R_{+};L_{p^*-\varepsilon,loc})\with\frac1{p^*}
=\frac1p-\frac1d\cdotp\end{aligned}
\end{equation}
This allows to pass to the limit in 
the convection term of the velocity 
equation of \eqref{rins2}.
In order to pass
to the limit in the terms containing $a^n$
and conclude that $(a,u,\nabla P)$ is a global weak solution, 
one can argue  as in \cite{HPZ2013, DFM}. 
Since  $a\in L_\infty(\R_+\times\R^d)$ and  $\nabla u\in L_{2q}(\R_{+};L_{\frac{dq}{2q-1}}),$ the Di Perna - Lions
theory in \cite{DL1989} ensures that  $a$ is the only solution to
the mass  equation of (INS) and  that
$$a^{n}\rightarrow a \ \text{strongly  in }\ L_{\alpha,loc}(\R_{+}\times \R^d)\quad\hbox{for all }\ 1<\alpha<\infty.$$ 
Then, using the uniform bounds for $(a^n,u^n,P^n)_{n\in\N}$ 
one  can pass to the limit in all the terms of the  following equations, that are satisfied by  construction of $(a^{n},u^{n},\nabla P^{n})$:
  $$\int_{0}^{\infty}\!\!\!\int_{\R^{d}}a^{n}(\d_{t}\phi+u^{n}\cdot \nabla \phi)\,dx\,dt+\int_{\R^{d}}\phi(0,x)a^n_{0}(x)\,dx=0,$$
 $$\int_{0}^{\infty}\!\!\!\int_{\R^{d}}\phi \,\div{u^{n}}\,dx\,dt=0$$
 for all functions $\phi\in \mathcal{C}^{\infty}_{c}(\R_+\times \R^{d}),$ and
 $$\displaylines{
\quad
\int_{0}^{\infty}\!\!\!\int_{\R^{d}}\bigl((1+a^{n})(u^{n}\cdot \d_{t}\Phi+(u^{n}\otimes u^{n}:\nabla \Phi))+(\Delta u^{n}-\nabla P^{n})\cdot \Phi\bigr)\,dx\,dt
\hfill\cr\hfill+\int_{\R^{d}}u^{n}_{0}\cdot \Phi(0,x)\,dx=0,}$$
  for all $\Phi\in \mathcal{C}^{\infty}_{c}(\R_+\times \R^{d};\R^d)$ with $\div \Phi\equiv 0.$ 
\smallbreak
This ensures that $(a,u,\nabla P)$ is a distributional solution of \eqref{rins2}, which 
completes the proof of the existence parts of  Theorems \ref{themd2} and  \ref{them1d3}.

%%%%%%%%%%%%%%%%%%%%%%%%%%%%%%%%%%%%

 \section{Uniqueness results}\label{section5} 
 
 The   goal of this section is to prove Theorem \ref{thm:uniqueness}, 
and to show that it  implies the uniqueness part  of 
  Theorems \ref{themd2}, \ref{them1d3} and \ref{them:PZuniqueness}.
 Theorem  \ref{thm:uniqueness} will come up as a consequence
 of the stability estimates  of Proposition \ref{propunid3} (in the 3D case)
 and of Propositions \ref{propunid2}, and \ref{propunid2bis} (2D case). 
 \smallbreak
% It  is based on a priori estimates  for $t^{-1}\dr$ (with $\dr:=\rho_1-\rho_2$) and $\du:=u_1-u_2$
%in the spaces $L_\infty(0,T;\dot H^{-1}(\R^d))$ and $L_\infty(0,T; L_2(\R^d))\cap L_2(0,T;\dot H^1(\R^d))$ or 
%$L_{\infty,loc}(0,T; L_2(\R^d))\cap L_{2,loc}(0,T;\dot H^1(\R^d)),$  respectively.
    Throughout this part, we denote
\begin{equation}\label{eq:notdensity}
r_0:=\inf_{x\in\R^d} \rho_0(x)\andf R_0:= \sup_{x\in\R^d} \rho_0(x).
\end{equation}

\subsection{Uniqueness in the three-dimensional case}

Let us first state general stability estimates in the 3D case. 
\begin{proposition} \label{propunid3}Let $(\rho_1,u_1,P_1)$ and  $(\rho_2,u_2,P_2)$ be two finite energy 
solutions of $(INS)$ on $[0,T]\times\R^3$ corresponding to the same initial density $\rho_0$ but, possibly, 
two different initial velocities $u_{1,0}$ and $u_{2,0}.$   Let $\dr:=\rho_1-\rho_2,$  $\du:=u_1-u_2,$
$$ f(t):=\|t\dot u_2\|_{L_\infty(\R^3)}+\|t\nabla \dot u_2\|_{L_3(\R^3)}
 \andf  g(t):=\|\nabla u_2\|_{L_\infty(\R^3)}.$$
There exists an absolute constant $C$ such that  the following inequality holds true for  all $t\in[0,T]$:
$$\displaylines{
\underset{\tau\in [0,t]}{\sup} \tau^{-1}\norm{\dr(\tau)}_{\H^{-1}(\R^{3})}
\leq  R_0^{1/2}  \|\sqrt{\rho_0}\,\du_{0}\|_{L_2(\R^3)} 
e^{\frac C2R_0\int_0^tf^2 \,d\tau\: \exp(2\int_0^tg\,d\tau)} e^{2\int_0^t g\,d\tau},\cr
\|\sqrt{\rho_1(t)}\du(t)\|^2_{L_2(\R^3)} + \!\int_0^t\!\|\nabla\du\|^2_{L_2(\R^3)}\,d\tau
\leq \|\sqrt{\rho_0}\,\du_{0}\|^2_{L_2(\R^3)}e^{CR_0\int_0^t\!f^2 \,d\tau\: \exp(2\!\int_0^tg\,d\tau)} e^{2\!\int_0^t g\,d\tau}
\cdotp}$$
\end{proposition}
\begin{proof}
The beginning of the proof is independent of the dimension $d.$ 
In sharp contrast with   \cite{DM1,RD,DM2}, our stability estimates
 are performed directly in Eulerian coordinates: we consider 
 the  following system that is satisfied by $\dr$, $\du$ and $\dP:=P_1-P_2,$   denoting $\dot u_2:=(u_2)_t+u_2\cdot\nabla u_2,$  
\begin{equation}\label{insu} \left\{\begin{aligned}
&(\dr)_{t}+\du\cdot \nabla \rho_{1}+ u_{2}\cdot \nabla \dr=0, \\
 & \rho_1(\du)_{t}+\rho_1u_{1}\cdot \nabla \du- \Delta \du+\nabla \dP=-\dr \,\t{u_{2}}-\rho_1\du\cdot \nabla u_{2},\\
&\dr|_{t=0}=0, \qquad\du|_{t=0}=\du_0.
\end{aligned}\right.
\end{equation}

Let us set $\phi:=-(-\Delta)^{-1}\dr$ 
(so that  $\norm{\dr}_{\H^{-1}(\R^{d})}=\norm{\nabla \phi}_{L_{2}(\R^{d})}$).
Testing the first equation of \eqref{insu} by $\phi$ yields after integrating
by parts and using that $\div u_1=\div u_2=0,$
\begin{equation*}
\begin{aligned}
\frac{1}{2}\frac{d}{dt}\norm{\nabla \phi}^{2}_{L_{2}(\R^{d})}&\!\leq \int_{\R^{d}} \nabla u_{2}:(\nabla \phi \otimes\nabla \phi)\,dx-\int_{\R^{d}} \rho_{1}\du\cdot \nabla \phi\,dx\\
 &\!\leq \norm{\nabla u_{2}}_{L_{\infty}(\R^{d})}\norm{\nabla  \phi\! \otimes \!\nabla \phi}_{L_{1}(\R^{d})}+\|\rho_1\|_{L_\infty(\R^d)}^{1/2}\norm{\sqrt{\rho_1}\,\du}_{L_{2}(\R^{d})}\norm{\nabla \phi}_{L_{2}(\R^{d})}.
\end{aligned}
\end{equation*}
After time  integration, we find that for all $t\in [0,T]$,
\begin{equation}\label{eq:dphi}
\norm{\nabla \phi(t)}_{L_{2}(\R^{d})}\leq \int_{0}^{t} \norm{\nabla u_{2}}_{L_{\infty}(\R^{d})}\norm{\nabla \phi}_{L_{2}(\R^{d})}\,d\tau+\int_{0}^{t} \norm{\rho_{1}}^{1/2}_{L_{\infty}(\R^{d})}\norm{\sqrt{\rho_1}\,\du}_{L_{2}(\R^{d})}\,d\tau.\end{equation}
For all $t\in[0,T],$ set 
\begin{align}\label{def:XY}
X(t)&:=\underset{\tau\in [0,t]}{\sup} \tau^{-1}\norm{\dr(\tau)}_{\H^{-1}(\R^{d})}\nonumber\\
\andf  Y(t)&:=\Bigl(\underset{\tau\in[0,t]}\sup\norm{(\sqrt{\rho_1}\du)(\tau)}_{L_{2}(\R^{d})}^2
+\norm{\nabla \du}_{L_{2}(0,t\times \R^{d})}^2\Bigr)^{1/2}.\end{align}
From \eqref{eq:dphi}  and the mass conservation, we end up with
\begin{equation}\label{esxt}X(t) \leq \int_0^t g X\,d\tau + R_0^{1/2} Y(t)
\end{equation}
and that Gronwall lemma thus gives (since $Y$ is nondecreasing)
\begin{equation}\label{eq:XY}
X(t)\leq R_0^{1/2}Y(t) e^{\int_0^t g\,d\tau}\cdotp
\end{equation}
 In order to control $Y,$ we test \eqref{insu} by $\du$ and  find that
\begin{equation}\label{eq:Y} \frac{1}{2}\frac{d}{dt} \int_{\R^{d}}\rho_1\abs{\du}^{2}\,dx+\int_{\R^{d}}\abs{\nabla \du}^{2}\,dx
 = -\int_{\R^{d}}\dr \, \t{u_{2}}\cdot \du\,dx-\int_{\R^{d}}\rho_1(\du\cdot \nabla u_{2})\cdot \du\,dx.\end{equation}
%Integrating on $[0,t]$  gives
%\begin{multline}\label{eq:Y}
%\frac12\norm{(\sqrt{\rho_1}\du)(t)}^{2}_{L_{2}(\R^{d})}+\int_{0}^{t} \!\!\int_{\R^{d}}
%\abs{\nabla \du}^{2}\,dx\,d\tau\\
%=-\int_{0}^{t}\!\!\int_{\R^{d}}\dr \: \t{u_{2}}\cdot \du\,dx\,d\tau-\int_{0}^{t}\!\!\int_{\R^{d}}\rho_1(\du \cdot \nabla u_{2})\cdot \du\,dx\,d\tau
%=:B_{1}+B_{2}.\end{multline}
   Bounding the last term  is straightforward: we just write that 
   \begin{equation}\label{eq:B2}
    -\int_{\R^{d}}\rho_1(\du\cdot \nabla u_{2})\cdot \du\,dx \leq   \norm{\nabla u_{2}}_{L_{\infty}(\R^{d})}\norm{\sqrt{\rho_1}\,\du}^{2}_{L_{2}(\R^{d})}.
    \end{equation}
In order to estimate   the term with $\dr \,\dot u_2\cdot\du,$   we argue by duality, writing  that
\begin{eqnarray}\label{eq:duality}
-\int_{\R^{d}}\dr \: \t{u_{2}}\cdot \du\,dx&\!\!\!\leq\!\!\!& \|\dr\|_{\dot H^{-1}(\R^d)} \|{\dot u_{2}}\cdot \du\|_{\dot H^1(\R^d)}\nonumber\\
&\!\!\!\leq\!\!\!& X\bigl(\|\tau\nabla \t{u}_2 \cdot \du\|_{L_2(\R^d)}  + \|\tau \t{u}_2\cdot\nabla\du\|_{L_2(\R^d)}\bigr)\cdotp\end{eqnarray}
Assuming in the rest of the proof that $d=3,$ and using H\"older inequality and 
the  embedding $\H^{1}(\R^{3})\hookrightarrow L_{6}(\R^{3}),$ we get
$$\begin{aligned}
    -\int_{\R^{3}}\dr \: \t{u_{2}}\cdot \du\,dx&\leq  X \bigl(\norm{\tau\nabla \t{u}_{2}}_{L_{3}(\R^{3})}\norm{\du}_{L_{6}(\R^{3})}
    +  \norm{\tau\t{u_{2}}}_{L_{\infty}(\R^{3})} \norm{\nabla \du}_{L_{2}(\R^{3})}\bigr)\\
    &\leq  C X \norm{\nabla \du}_{L_{2}(\R^{3})} \bigl(\tau\norm{\nabla \t{u}_{2}}_{L_{3}(\R^{3})}
    +  \norm{\tau\t{u_{2}}}_{L_{\infty}(\R^{3})}\bigr)\\
       &\leq  \frac12\|\nabla\du\|_{L_2(\R^3)}^2 +  \frac{C^2}2 X^2 \bigl(\norm{\tau\nabla \t{u}_{2}}_{L_{3}(\R^{3})}
    +  \norm{\tau\t{u_{2}}}_{L_{\infty}(\R^{3})}\bigr)^2 
    \cdotp 
    \end{aligned}$$
    Hence, plugging \eqref{eq:B2} and the above inequality in \eqref{eq:Y} and using the notation of the statement, we get
  (for another constant $C$):
    $$\frac{d}{dt} \int_{\R^{3}}\rho_1\abs{\du}^{2}\,dx+\int_{\R^{3}}\abs{\nabla \du}^{2}\,dx\leq 
   2 \norm{\nabla u_{2}}_{L_{\infty}(\R^{d})}\norm{\sqrt{\rho_1}\,\du}^{2}_{L_{2}(\R^{3})}
+  C f^2 X^2.$$
 Integrating on $[0,t],$ the above inequality becomes
$$Y^2(t)\leq Y^2(0) + 2\int_0^t gY^2\,d\tau + C\int_0^t f^2 X^2\,d\tau\cdotp$$
Hence, Gronwall lemma gives
$$Y^2(t)\leq e^{2\int_0^tg}\biggl( Y^2(0) + C\int_0^t e^{-2\int_0^\tau g\,d\tau'} f^2 X^2\,d\tau\biggr)\cdotp$$
Plugging  \eqref{eq:XY} in the above inequality, we discover that
$$
Y^2(t)\leq e^{2\int_0^tg\,d\tau}\biggl( Y^2(0) + CR_0\int_0^t f^2 Y^2\,d\tau\biggr)\cdotp$$
Hence, applying again Gronwall inequality, we end up with 
$$Y^2(t) \leq  e^{CR_0(\int_0^tf^2 \,d\tau)\: \exp(2\int_0^tg\,d\tau)} e^{2\int_0^t g\,d\tau} Y^2(0).$$ 
Inserting this latter inequality in \eqref{eq:XY}  completes the proof. 
%Hence $$\displaylines{
%Y(t) \leq e^{\frac C2R_0\int_0^tf^2 \,d\tau\: \exp(2\int_0^tg\,d\tau)} e^{\int_0^t g\,d\tau} Y^2(0)\hfill\cr\hfill
%\andf X(t)\leq R_0^{1/2} Y(0)  e^{\frac C2R_0\int_0^tf^2 \,d\tau\: \exp(2\int_0^tg\,d\tau)} e^{2\int_0^t g\,d\tau}.}$$
\end{proof}
 We claim that the above proposition implies the uniqueness part of Theorem \ref{them1d3}. 
 As a first, we have to explain why the map  $t\mapsto t^{-1}{\dr}$ belongs to  $L_{\infty}(0,T;\H^{-1}(\R^{3})).$
 In fact,  for all $t\in[0,T],$ integrating the mass equation of (INS) on $[0,t]$ yields
$$\rho_i(t)-\rho_{0}=-\int_{0}^{t}\div(\rho_i u_i)\,d\tau,\qquad i=1,2.$$
Hence,
$$\begin{aligned}\norm{\rho_i(t)-\rho_{0}}_{\H^{-1}(\R^{3})}&\leq \int_{0}^{t}\norm{\div(\rho_i u_i)}_{\H^{-1}(\R^{3})}\,d\tau\\
&\leq t\norm{\sqrt{\rho_i}u_i}_{L_{\infty}(\R_{+};L_{2}(\R^{3}))}\norm{\sqrt{\rho_i}}_{L_{\infty}(\R_{+}\times\R^{3})},\end{aligned}$$
and thus, thanks to the energy balance and the mass equation,
$$t^{-1}\|\dr(t)\|_{\dot H^{-1}(\R^3)}\leq  \sqrt{R_0}\,\|\sqrt{\rho_0}u_0\|_{L_2(\R^3)}.$$
 Next, we have to show that $\du$ is in the energy space.
 If $2<p<3,$ then  this is guaranteed by the assumption of Theorem \ref{them1d3}. 
If $1<p\leq2,$ then we argue as follows. By construction,  $\d_t{u}$ is in $L_{q,1}([0,T];L_p(\R^3))$ and $u$
 is in $C_b([0,T];\B^{-1+3/p}_{p,1}(\R^3))$, whence
 $$u(t)-u_0 \in  C([0,T];L_p(\R^3))\cap C([0,T];\B^{-1+3/p}_{p,1}(\R^3)).$$
 Hence $u(t)-u_0 \in   C([0,T]; B^{-1+3/p}_{p,1}(\R^3))$
 (nonhomogeneous Besov space).  Owing to the classical 
 embedding 
  $$B^{-1+3/p}_{p,1}(\R^3)\hookrightarrow  H^{1/2}(\R^3)\quad\hbox{for }\ 1<p\leq2,$$ we obtain  $$u(t)-u_0 \in  C([0,T];L_2(\R^3)).$$
 Now, taking $(s,m)=(4,2)$ in
 \eqref{eq:maxreg2},   we 
 see that  $\nabla u_i$ belongs to $L_{4,1}([0,T];L_2),$ hence to 
 $L_2(0,T\times\R^d).$
From this, we eventually conclude  
 that 
 $\du$ is in  $L_{\infty}(0,T; L_2)\cap L_{2}(0,T;\dot H^1).$
 \smallbreak 
 Finally,  the solution constructed in Theorem \ref{them1d3}
satisfies $t\nabla\t u\in L_2(\R_+;L_3(\R^3))$ and $\nabla u\in L_1(\R_+;L_\infty(\R^d)).$
Hence, all the assumptions of Theorem  \ref{thm:uniqueness} are satisfied by the solutions constructed 
in Theorem \ref{them1d3}, which are thus unique. 
 \bigbreak
  Another  corollary Proposition \ref{propunid3} is the uniqueness of  P. Zhang's 
  solutions constructed  in \cite{Zhang19}.
 Indeed, if  $(\rho,u)$ stands for a solution of Theorem \ref{them:PZuniqueness} then it satisfies \eqref{eq:PZsolutions}.
 Therefore, thanks to the embeddings $$\B^{1/2}_{2,1}(\R^{3})\hookrightarrow L_3(\R^3), \quad \B^{3/2}_{2,1}(\R^{3})\hookrightarrow L_\infty(\R^3),$$ 
% and to
%$$\norm{t\t{u}}_{\title L^{2}(\R_+;\B^{3/2}_{2,1})}\leq C\norm{u_0}_{ \B^{\frac{1}{2}}_{2,1}},$$
on can write that for all $T>0,$ we have
 $$t\nabla\dot u \in L_2(\R_+;L_3(\R^3))\andf t\dot u \in L_2(\R_+;L_\infty(\R^3)).$$
 Hence, if we prove in addition that 
 $\nabla u\in L_1(0,T;L_\infty(\R^3))$ for all $T>0,$
 then  Proposition \ref{propunid3} will ensure uniqueness. 
\smallbreak
In order to prove this latter property, let us 
observe as in \cite{Zhang19} that if $(\rho,u,\nabla P)$
is a solution to (INS) on 
$[0,T]\times\R^3,$ 
%first  consider the linear system 
%$$\left\{\begin{aligned}
% &\rho \d_t u+\rho v\cdot \nabla u-\Delta u+\nabla P=0, \\
%&\div u=0,\\
%&u|_{t=0}=u_0,
%\end{aligned}\right.$$
%where (for the time being) $\rho$ and $v$ are  smooth functions %such that 
%$\rho$ is bounded away from zero,  $\d_t \rho+v\cdot \nabla %\rho=0$ and $ \div v=0.$ In order to get the Lipschitz property,
and if 
we look at the following \emph{linear} Stokes system with convection:
 \begin{equation*}
\left\{\begin{aligned}
 &\rho \d_t u_j+\rho u\cdot \nabla u_j-\Delta u_j+\nabla P_j=0, \\
&\div u_j=0,\\
&u_j|_{t=0}=\dot \Delta_j u_0,
\end{aligned}\right.
\end{equation*}
then, by uniqueness,  we have 
%Since the initial data and the transport field are smooth 
%and $\rho$ is smooth and bounded away from $0,$ proving 
%the uniqueness is not an issue, and 
\begin{equation}\label{eq:sumup}
u=\sum_{j\in\Z}u_j  \andf \nabla P=\sum_{j\in \Z}\nabla P_j \cdotp
\end{equation}
%{\bf R. I am not sure one can apply uniqueness now since at this stage we do not know that the solution $u$ is unique. 
%Probably we need one more step : first we replace $u$ with a smooth $v$ in the convection term and 
%$\rho$ such that $\d_t\rho+v\cdot\nabla\rho=0.$ This is linear system for which we have uniqueness 
%as well as the estimates of Corollary 1.18. For this solution, we prove the Lipschitz property. 
%And, as a consequence, it still holds for $u=v.$}
%\bigbreak

 In  \cite{Zhang19}, under assumptions \eqref{eq:Z1} and \eqref{eq:Z2}, 
% that $\|u_0\|_{\dot B^{1/2}_{2,1}(\R^3)}$ is small 
% enough and that $\rho_0^{\pm1}$ is bounded, 
% P. Zhang proved 
 the following  time weighted estimates
 have been proved
 (see Corollaries 3.1, 3.2 and 4.2 and Inequalities
 (2.10), (3.8) and (3.23)):
\begin{align}\label{eq:timeweighted1}
\|\sqrt{t}\nabla^2 \!u_j\|_{L_2(0,T\times\R^3)}\!+\!\|t\nabla \d_t u_j\|_{L_2(0,T\times\R^3)}\!+\!\|t \nabla^2\! u_j\|_{L_\infty(0,T;L_2)}&\lesssim d_j^1 2^{-\frac j2}\norm{u_0}_{ \B^{{1}/{2}}_{2,1}} \\\label{eq:timeweighted2}
 \|\nabla^2 u_j\|_{L_2(0,T\times\R^3)}+\|\sqrt{t}\nabla^2 u_j\|_{L_\infty(0,T;L_2)}+\|t\nabla \d_t u_j\|_{L_\infty(0,T;L_2)}&\lesssim d_j^2 2^{\frac j2}\norm{u_0}_{ \B^{{1}/{2}}_{2,1}},
\end{align}
with  $\{d_j^1\}_{j\in \Z}$  and 
 $\{d_j^2\}_{j\in \Z}$ in  the unit ball of $\ell_1(\Z).$
\medbreak
This, together with the following interpolation result (see \cite[Th2:1.18.6]{HT}):
$$L_{4,1}(0,T;L_2)
=\bigl(L_{2}(0,T;L_2),L_{\infty}(0,T;L_2)\bigr)_{1/2,1}$$ yields
\begin{equation}\label{eq:esl41l2}
\|\sqrt{t}\nabla^2 u_j\|_{L_{4,1}(0,T;L_2)}\leq C d_j \norm{u_0}_{ \B^{{1}/{2}}_{2,1}}\with
\sum_{j\in\Z} d_j=1.\end{equation}
Next, since 
$$-\Delta u_j+\nabla P_j=-\rho\d_tu_j-\rho v\cdot\nabla u_j,
$$
we have in light of the standard elliptic
regularity result for the Stokes system: 
$$\|t\nabla^2 u_j, t\nabla P_j\|_{L_{2}(0,T;L_6)} \leq
C\bigl(\|\d_tu_j\|_{L_{2}(0,T;L_6)}
+ \|v\cdot\nabla u_j\|_{L_{2}(0,T;L_6)}\bigr)\cdotp$$
Hence, as $\norm{u_0}_{ \B^{{1}/{2}}_{2,1}}$ is small, 
using H\"older inequality, 
 $\dot H^1(\R^3)\hookrightarrow L_6(\R^3)$
and  $\dot B^{3/2}_{2,1}(\R^3)\hookrightarrow L_\infty(\R^3)$
and, eventually, \eqref{eq:PZsolutions} and \eqref{eq:timeweighted1}, we get
$$\begin{aligned}
\|t\nabla^2 u_j, t\nabla P_j\|_{L_{2}(0,T;L_6)} 
    &\leq C \bigl(\|t\d_t u_j\|_{L_{2}(0,T;L_6)}+\|tu\cdot \nabla u_j\|_{L_{2}(0,T;L_6)} \bigr)\\
    &\leq  C \bigl(\|t\nabla \d_t u_j\|_{L_{2}(0,T\times \R^3)}+\|u\|_{L_{2}(0,T;L_\infty)}\|t\nabla^2 u_j\|_{L_{\infty}(0,T;L_2)} \bigr)\\
    &\leq C d_j^12^{-j/2}\norm{u_0}_{ \B^{{1}/{2}}_{2,1}}.
\end{aligned}$$
Similarly, we deduce from 
elliptic regularity,  embedding, \eqref{eq:PZsolutions} 
  and \eqref{eq:timeweighted2}
 that 
$$\begin{aligned}
\|t\nabla^2 u_j, t\nabla P_j\|_{L_{\infty}(0,T;L_6)} 
    &\!\leq\! C \bigl(\|t\d_t u_j\|_{L_{\infty}(0,T;L_6)}+\|tu\cdot \nabla u_j\|_{L_{\infty}(0,T;L_6)} \bigr)\\
    &\!\leq\!  C \bigl(\|t\nabla\d_t u_j\|_{L_{\infty}(0,T;L_2)}\!+\!\|\sqrt t u\|_{L_{\infty}(0,T\times\R^{3})}\|\sqrt t \nabla^2 u_j\|_{L_{\infty}(0,T;L_2)} \bigr)\\
    &\leq  C
     d_j^22^{j/2}\norm{u_0}_{\B^{{1}/{2}}_{2,1}}.
\end{aligned}$$
Together with the interpolation property $$L_{4,1}(0,T;L_6)=\bigl(L_2(0,T;L_6),L_{\infty}(0,T;L_6)\bigr)_{1/2,1},$$ 
this yields
\begin{equation}\label{eq:l41l6es}
\|t\nabla^2 u_j\|_{L_{4,1}(0,T;L_6)}\leq C d_j \|u_0\|_{\B^{1/2}_{2,1}}\with\sum_{j\in\Z} d_j=1.
\end{equation}
Summing up on  all $j\in \Z$ we deduce from \eqref{eq:sumup}, \eqref{eq:esl41l2} and \eqref{eq:l41l6es} that
\begin{equation}\label{eq:last}\|t \nabla^2 u\|_{L_{4,1}(0,T;L_6)} + \| \sqrt{t} \nabla^2 u\|_{L_{4,1}(0,T;L_2)}\lesssim \norm{u_0}_{ \B^{{1}/{2}}_{2,1}}.\end{equation}
It is now easy  to bound  $\nabla u$ in $L_1(0,T;L_\infty)$
for all $T>0.$ Recall  the following  Gagliardo-Nirenberg inequality: 
$$\|z\|_{L_\infty}\leq \norm{\nabla z}^{1/2}_{L_2}\norm{\nabla z}^{1/2}_{L_6},$$
Then, combining with Proposition \ref{p:lorentz} (items (iii), (iv) and (v)) and \eqref{eq:last} yields
$$\begin{aligned}
\int^{T}_{0} \|\nabla u\|_{\infty}\,dt 
  &\leq C \int_0^T \norm{\nabla^2 u}^{1/2}_{L_2}\norm{\nabla^2 u}^{1/2}_{L_6}\,dt\\
  &\leq C \int_0^T t^{-3/4}\norm{\sqrt{t}\nabla^2 u}^{1/2}_{L_2}\norm{t\nabla^2 u}^{1/2}_{L_6}\,dt\\
  &\leq C\|t^{-3/4}\|_{L_{4/3,\infty}(\R_+)}\norm{\sqrt{t}\nabla^2 u}^{1/2}_{L_{4,1}(0,T;L_2)}\norm{t\nabla^2 u}^{1/2}_{L_{4,1}(0,T;L_6)}\\
  &\leq C\norm{u_0}_{ \B^{{1}/{2}}_{2,1}}\cdotp
\end{aligned}$$
This completes the proof of the  Lipschitz 
regularity for the velocity. Now, applying 
Proposition \ref{propunid3} yields uniqueness in Theorem 
 \ref{them:PZuniqueness}.
 %The same estimate true for the nonlinear system (INS): just %take $u=v$
%and remember that  $\|u\|_{L_2(0,T;L_\infty)}$ and $\|\sqrt %t\,u\|_{L_\infty(0,T\times\R^3)}$
%may be bounded in terms of $\norm{u_0}_{\B^{\frac{1}{2}}_{2,1}}.$
%\medbreak
%Finally, by using Proposition \ref{propunid3} for two Zhang's %solutions $(\rho_1,u_1, P_1)$ and $(\rho_2,u_2, P_2)$ with same %initial data we deduce that $(\rho_1,u_1, P_1)\equiv %(\rho_2,u_2, P_2).$ This completes the uniqueness part of P. %Zhang's solutions and of Theorem \ref{them:PZuniqueness}.

%%%%%%%%%%%%%%%%%%%%%%%%%%%%%%%%%%%%%%%%%%%%%%%%%%%%%%%%%%%%%%%%%

\subsection{Stability and uniqueness in the two-dimensional case}

Let us  present a  first result that requires the density to be bounded 
away from zero (that is $r_0>0$ in \eqref{eq:notdensity}). 
\begin{proposition} \label{propunid2}
Let $(\rho_1,u_1,P_1)$ and  $(\rho_2,u_2,P_2)$ be two 
solutions of $(INS)$ on $[0,T]\times\R^2$ corresponding to the same initial density $\rho_0$
\emph{bounded away from $0$} and, possibly, 
two different initial velocities $u_{1,0}$ and $u_{2,0}.$   
Denote $g(t):=\|\nabla u_2(t)\|_{L_\infty(\R^2)}$ and 
$$ f_1(t):=\| t\dot u_2\|^2_{L_\infty(\R^2)}+\|t\nabla^2\dot u_2\|^q_{L_p(\R^2)}\
\hbox{ for some }\ 1<p,q<2\ \hbox{ such that }\ \frac 1p+\frac 1q=\frac 32\cdotp$$

There exists a constant $C$ depending only $p$ such that   the functions $\dr:=\rho_1-\rho_2$ and
  $\du:=u_1-u_2$ satisfy  for  all $t\in[0,T]$:
$$\displaylines{
\underset{\tau\in [0,t]}{\sup} \tau^{-1}\norm{\dr(\tau)}_{\H^{-1}(\R^{2})}\hfill\cr\hfill
\leq  R_0^{1/2}  \|\sqrt{\rho_0}\,\du_{0}\|_{L_2(\R^2)} 
\exp\biggl(\int_0^t(2g+{\textstyle\frac{Cf_1}{2r_0}})\,d\tau\biggr)
\exp\biggl({\textstyle\frac{CR_0}2}e^{2\int_0^tg\,d\tau}\int_0^tf_1\,d\tau\biggr),\cr
\|\sqrt{\rho_1(t)}\du(t)\|^2_{L_2(\R^2)} + \!\int_0^t\!\|\nabla\du\|^2_{L_2(\R^2)}\,d\tau\hfill\cr\hfill
\leq \|\sqrt{\rho_0}\,\du_{0}\|^2_{L_2(\R^2)}
\exp\biggl(\int_0^t(2g+Cr_0^{-1}f_1)\,d\tau\biggr)
\exp\biggl(CR_0e^{2\int_0^tg\,d\tau}\int_0^tf_1\,d\tau\biggr)\cdotp
%e^{R_0r_0(1-\exp(Cr_0^{-1}\int_0^t\!f_1 \,d\tau))\: %\exp(\int_0^t(2g+Cr_0^{-1}f_1)\,d\tau)}\cdotp
}$$
\end{proposition}
\begin{proof}
 Let us define the functions $X$ and $Y$ according to \eqref{def:XY}.
Compared to the three-dimensional case, the only change is in the treatment of the  first term of the 
right-hand side of \eqref{eq:duality}.  Now, using
 H\"older inequality, the  embedding 
 $$\W^{1}_{p}(\R^2)\hookrightarrow L_{m}(\R^2)\with\frac 1m=\frac 1p-\frac 12,$$ 
 and Gagliardo-Nirenberg inequality, we get (with $1/s:=1-1/p$): 
 $$\begin{aligned}
  \|\tau\nabla\t{u_{2}}\cdot \du\|_{L_2(\R^2)} \leq & \norm{\tau\nabla \t{u}_{2}}_{L_{m}(\R^{2})}\norm{\du}_{L_{s}(\R^{2})}\\
    \leq & C  \norm{t\nabla^{2}\t{u_2}}_{L_p(\R^2)}\norm{\du}^{\frac 2s}_{L_2(\R^2)}\norm{\nabla\du}^{\frac 2m}_{L_2(\R^2)}.
        \end{aligned}$$
    Hence, reverting to \eqref{eq:duality} and using  Young inequality, we obtain that 
    $$\begin{aligned}
    -\int_{\R^{2}}\dr \: \t{u_{2}}\cdot \du\,dx \leq  &\frac 12 \norm{\nabla\du}_{L_2(\R^2)}^2+C X^2\norm{t \t{u_2}}^{2}_{L_\infty(\R^2)}
    +C X^q \norm{t\nabla^{2}\t{u_2}}^q_{L_p(\R^2)}\norm{\du}^{2-q}_{L_2(\R^2)}\\
   \leq  &\frac 12 \norm{\nabla\du}_{L_2(\R^2)}^2+C X^2 \bigl(\norm{t \t{u_2}}^{2}_{L_\infty(\R^2)}+\norm{t\nabla^{2}\t{u_2}}^q_{L_p(\R^2)}\bigr)\\
    &\hspace{5cm}+C\norm{t\nabla^{2}\t{u_2}}^q_{L_p(\R^2)}\norm{\du}^{2}_{L_2(\R^2)}\cdotp 
    \end{aligned}$$
Then, substituting \eqref{eq:B2} and the above inequality into \eqref{eq:Y} yields
 $$\frac{d}{dt} \int_{\R^{2}}\rho_1\abs{\du}^{2}\,dx+\int_{\R^{2}}\abs{\nabla \du}^{2}\,dx\leq 
   2 g \norm{\sqrt{\rho_1}\,\du}^{2}_{L_{2}(\R^{2})}
+  C f_1 X^2+C f_1 \norm{\du}^{2}_{L_{2}(\R^{2})}\cdotp$$
As the density is bounded from below by $r_0>0,$ after integrating on $[0,t],$ we get
$$Y^2(t)\leq \int_0^t (2 g(\tau)+Cr_0^{-1} f_1(\tau))Y^{2}(\tau) \,d\tau+C\int_0^t X^2(\tau) f_{1}(\tau)\,d\tau+Y^2(0).$$
Hence, applying Gronwall's inequality  yields
$$Y^{2}(t)\leq e^{\int_0^t (2 g+Cr_0^{-1} f_1) \,d\tau}\bigl(Y^2(0)+\int_0^t C X^2 f_1 e^{-\int_0^\tau (2 g+C r_0^{-1}f_1) \,d\tau'}\,d\tau\bigr),$$ 
which together with \eqref{eq:XY} implies
$$Y^{2}(t)\leq e^{\int_0^t (2 g+Cr_0^{-1} f_1) \,d\tau}\bigl(Y^2(0)+C R_{0}\int_0^t  Y^2 f_1 e^{-\int_0^\tau Cr_0^{-1}f_1\,d\tau'}\,d\tau\bigr)\cdotp$$
Hence, we deduce from Gronwall's inequality that
$$Y^{2}(t)\leq  Y^2(0)\,
\exp\biggl(\int_0^t(2g+Cr_0^{-1}f_1)\,d\tau\biggr)
\exp\biggl(CR_0e^{2\int_0^tg\,d\tau}\int_0^tf_1\,d\tau\biggr)\cdotp $$ 
Inserting this latter inequality in the inequality for $X$ completes the proof. 
\end{proof}
\smallbreak
Proposition \ref{propunid2} implies the uniqueness part of Theorem \ref{themd2}. 
Indeed,  the density of the   solutions constructed therein is bounded away from zero, 
the gradient of their velocity is in $L_1(\R_+;L_\infty(\R^d)),$ 
we have  $t\nabla^2\t u\in L_q(\R_+;L_p(\R^2))$ with $1<p,q<2$ such that $1/p+1/q=3/2$
and $t\t u\in L_2(\R_+;L_\infty(\R^2)),$
the  solutions have finite energy, the map 
$t\mapsto t^{-1}\dr$ is in $L_\infty(0,T;\dot H^{-1}(\R^3))$ for all $T>0$
(the proof is exactly the same as in the 3D case). 
%to the following Gagliardo-Nirenberg inequality:
%$$\norm{z}_{L_{\infty}(\R^{2})}\lesssim %\norm{z}^{1/2}_{L_{2}(\R^{2})}\norm{\nabla^2 z}^{1/2}_{L_{2}(\R^{2})}$$
%we have for all $t>0,$
%$$\int_0^t \|\tau \dot u\|_{L_\infty(\R^2)}^2\,d\tau \lesssim 
%\int_0^t   \norm{\tau^{1/2}\dot %u}_{L_{2}(\R^{2})}\norm{\tau^{3/2}\nabla^2\dot %u}_{L_{2}(\R^{2})}\,d\tau<\infty.$$
\medbreak
Having in mind the results in the three-dimensional case, it  is natural to address the uniqueness 
issue \emph{without assuming that the density has a positive lower bound}. 
The following result ensures uniqueness in the case of periodic boundary conditions, without 
making any particular assumption on the density.  
\begin{proposition} \label{propunid2bis}
Let $(\rho_1,u_1,P_1)$ and  $(\rho_2,u_2,P_2)$ be two 
solutions of $(INS)$ on $[0,T]\times\T^2$ corresponding to the same  data 
$(\rho_0,u_0)$ such that  $M:=\int_{\T^2}\rho_0\,dx$ is positive. 
\smallbreak
 If,  in addition,
 $$\nabla u_2\in L_1(0,T;L_\infty(\T^2))\andf \bigl[t\mapsto t\,\log(e+t^{-1})
\t{u}_2\bigr]\in L_{2}(0,T;L_\infty(\T^2)\cap H^1(\T^2)),$$ 
 then  $(\rho_1,u_1,P_1)=(\rho_2,u_2,P_2)$ on $[0,T]\times\T^2.$
\end{proposition}
\begin{proof}
Compared to the previous proposition, the only change is in the treatment of the  first term of the 
right-hand side of \eqref{eq:duality}.  Thanks to Inequality \eqref{eq:product1} adapted to the periodic
setting\footnote{The Littlewood-Paley decomposition that
is required for proving \eqref{eq:product1} may 
be adapted to the periodic setting, see e.g. \cite{D-cours}.}, 
we have
\begin{align}\label{eq:uniq23}
-\int_{\T^{2}}\dr \: \t{u_{2}}\cdot \du\,dx &\leq \|\du\|_{H^1(\T^2)}\norm{\dr \cdot \dot u_2}_{H^{-1}(\T^2)}\nonumber\\
&\lesssim  \|\du\|_{H^1(\T^2)}\|\dr\|_{H^{-1}(\T^2)}
\log^{\frac12}\biggl(1+\frac{\|\dr\|_{L_2(\T^2)}}{\|\dr\|_{H^{-1}(\T^2)}}\biggr)
\|\dot u_2\|_{H^1(\T^2)\cap L_\infty(\T^2)}.\end{align}
Note that one cannot bound directly  $\|\du\|_{H^1(\T^2)}$ from $\|\nabla\du\|_{L_2(\T^2)}$
since  $\int_{\T^2} \du\,dx$ need not be zero and, in the periodic setting, 
$$\|\du\|_{H^1(\T^2)}\simeq \biggl|\int_{\T^2} \du\,dx\biggr| +\|\nabla\du\|_{L_2(\T^2)}.$$
To bound the first term we write that by virtue of Cauchy-Schwarz and Poincar\'e inequalities,
$$\begin{aligned}
\biggl|M\int_{\T^2} \du\,dx\biggr|
&=\biggl| \int_{\T^2}\rho_1\du\,dx +\int_{\T^2}(M-\rho_1)\biggl(\du-\int_{\T^2}\du\,dx\biggr)dx\biggr|\\
&\leq \sqrt M  \,\|\sqrt{\rho_1} \du\|_{L_2(\T^2)}
+ C\|M-\rho_1\|_{L_2(\T^2)}\|\nabla\du\|_{L_2(\T^2)}.
\end{aligned}
$$
Therefore, there exists a constant $C$ depending only on $M$ and on $R_0,$ and such that
$$\|\du\|_{H^1(\T^2)}\leq C\bigl(\|\sqrt{\rho_1} \du\|_{L_2(\T^2)}+\|\nabla\du\|_{L_2(\T^2)}\bigr)\cdotp$$
 Let us denote
 \begin{equation}\label{eq:f2}
 g(t):=\|\nabla u_2(t)\|_{L_\infty(\T^2)}\andf f_2(t):=\|t \dot u_2\|_{L_\infty(\T^2)\cap H^1(\T^2)}.\end{equation} 
Plugging the above  inequality in \eqref{eq:uniq23} and using 
\eqref{eq:f2},  \eqref{def:XY} 
and \eqref{eq:notdensity}, and, finally,   Young's inequality for the second line  yields 
    $$\begin{aligned}
    -\int_{\T^{2}}\dr \: \t{u_{2}}\cdot \du\,dx &\leq 
    C\bigl(\|\sqrt{\rho_1} \du\|_{L_2(\T^2)}+\|\nabla\du\|_{L_2(\T^2)}\bigr)
X\,\log^{\frac12}\biggl(1+\frac{R_0}{\|\dr\|_{H^{-1}(\T^2)}}\biggr) f_2\\
     &\leq  \frac12\|\nabla\du\|_{L_2(\T^2)}^2+Y^2+CX^2f_2^2\,\log\biggl(1+\frac{R_0}{tX}\biggr)\cdotp
    \end{aligned}$$
Still thanks to \eqref{eq:notdensity}, we see that there exists a constant $C$ such that 
$$\sup_{t\in[0,T]} \|\dr(t)\|_{L_2}\leq CR_0.$$
 Hence, reverting to \eqref{eq:Y} and using the notation of \eqref{def:XY} yields
%  $$ \frac{d}{dt}Y^2\leq  
% 2\|\nabla u_2(t)\|_{L_\infty}  Y^2 +2 %Y^2+CX^2f_2^2\,\log\biggl(1+\frac{R_0}{tX}\biggr). $$
% Hence, reverting to \eqref{eq:Y} yields
  $$ \frac{d}{dt}Y^2\lesssim (1+g)Y^2  + X^2f_2^2\, \log\biggl(1+\frac{R_0}{tX}\biggr)\cdotp$$
 Remembering \eqref{eq:XY} and integrating the above inequality yields
  \begin{equation*}
Y^2(t)\lesssim \int_0^t (1+g) Y^2\,d\tau 
+R_0\int_0^t e^{2\int_0^\tau g\,d\tau'}
f^2_2 \, Y^2 \log\biggl(1+\frac{R_0^{1/2}}{Y \tau e^{\int_0^\tau g\,d\tau'}}\biggr)d\tau. 
\end{equation*}
 Hence, taking  advantage of the following basic inequality:
$$\log(1+aY^{-1})\leq \log(1+a)(1-\log Y), \quad a\geq 0,\quad  Y\in (0,1),$$ 
we get
 \begin{equation*}
Y^2(t)\lesssim \int_0^t (1+g) Y^2\,d\tau 
+CR_0\int_0^t \log(1+R_0^{1/2}\tau^{-1})\, e^{2\int_0^\tau g\,d\tau'}
f^2_2 \, Y^2 \,(1-\frac12\log Y^2)\,d\tau. 
\end{equation*}
Our assumptions ensure that both $g$ and 
$\tau\mapsto \log(1+R_0^{1/2}\tau^{-1})\, e^{2\int_0^\tau g\,d\tau'}
f^2_2(\tau)$ are integrable on $[0,T].$
Furthermore, 
the function $r\to r(1-\frac12\log r)$ is increasing near $0^+$ 
and satisfies
$$\int_0^1\frac{dr}{ r(1-\frac12\log r)}=\infty.$$
Hence 
%applying Gronwall's inequality  yields
%$$Y^2(t)\leq Ce^{\int_0^t 2g\,d\tau}\int_0^t f^2_2 X^2 \log(1+\tau^{\frac 12})(1-\log \tau^{\frac 12})\log %(1+Z(\tau))(1-\log X(\tau))\,d\tau\cdotp$$
%Plugging  \eqref{eq:XY} in the above inequality, we find that
%$$X^2(t)\leq C R_0e^{\int_0^t 4g\,d\tau} \int_0^t f^2_2 X^2 \tau^{\frac 12}(1-\log \tau^{\frac 12})\log %(1+Z(\tau))(1-\log X(\tau))\,d\tau.$$
%Hence, as is increasing, and satisfies
%$$\underset{t\to 0}{\lim}\,t\,(1-\log t)=0,$$
one can apply Osgood lemma (see e.g. \cite[Lemma 3.4]{BCD}) 
so as to conclude that  $Y\equiv 0$ on $[0,T],$ and thus, owing  to \eqref{eq:XY}, 
we have $X\equiv 0,$ too. 
\end{proof}

%{\bf Something about  div-curl lemma.
%For  $\phi=-(-\Delta)^{-1}\dr$ such that  $\norm{\dr}_{\H^{-1}(\R^{d})}=\norm{\nabla \phi}_{L_{2}(\R^{d})},$ %then $$\begin{aligned}
%-\int_{\R^{2}}\dr \: \t{u_{2}}\cdot \du\,dx\leq &\int_{\R^{2}}\nabla \phi \: \nabla \t{u_{2}}\cdot \du\,dx %+\int_{\R^{2}}\nabla \phi \:  \t{u_{2}}\cdot \nabla \du\,dx\\
%\leq & \|\nabla \phi  \nabla \t{u_{2}}\|_{\mathcal{H}^1}\|\du\|_{BMO}+\|\nabla \phi\|_{L_2}\| %\t{u_{2}}\|_{L_\infty}\|\nabla \du\|_{L_2}\cdotp
%\end{aligned}$$
%Even though $curl(\nabla \phi)=0$ and $curl(\nabla \t{u_{2}})=0$, but no one is divergence free. }

%%%%%%%%%%%%%%%%%%%%%%%%%%%%%%%%%%%%%%%%%

\subsection* {Acknowledgments:}
 %The authors  are  indebted to the anonymous referee for  his/her careful reading and numerous suggestions that
 %contributed to improve a lot the present article. 
The first author  is  partially supported by the ANR project INFAMIE (ANR-15-CE40-0011).
 The second author has been partly funded by the B\'ezout Labex, funded by ANR,  reference ANR-10-LABX-58.

\appendix
\section{}
For the reader's convenience, we here list some results  involving Besov spaces and Lorentz spaces, 
prove maximal regularity estimates 
in Lorentz spaces for \eqref{eq:stokes}, 
and product estimates that were needed at 
the end of the last section.
\medbreak
The following properties of Lorentz spaces may be found in e.g. \cite{LG}:
\begin{proposition}[Properties of  Lorentz spaces] \label{p:lorentz}
There holds:
\begin{enumerate}
\item {\rm Interpolation}: For all $1\leq r,q\leq \infty$ and $\theta\in (0,1)$, we have 
$$\left(L_{p_{1}}(\R_{+};L_{q}(\R^{d}));L_{p_{2}}(\R_{+};L_{q}(\R^{d}))\right)_{\theta,r}=L_{p,r}(\R_{+};L_{q}(\R^{d}))),$$
where $1<p_{1}<p<p_{2}<\infty$ are such that  $\frac{1}{p}=\frac{(1-\theta)}{p_{1}}+\frac{\theta}{p_{2}}\cdotp$
\item {\rm Embedding}: $L_{p,r_{1}}\hookrightarrow L_{p,r_{2}} \ \text{if}\  r_{1}\leq r_{2},$ and $L_{p,p}=L_{p}.$ 
\item {\rm H\"older inequality}: for $1<p,p_{1},p_{2}<\infty$ and $1 \leq r,r_{1},r_{2}\leq \infty,$
 we have $$\norm{fg}_{L_{p,r}}\lesssim \norm{f}_{L_{p_{1},r_{1}}}\norm{g}_{L_{p_{2},r_{2}}}\quad\hbox{if}\quad \frac{1}{p}=\frac{1}{p_{1}}+\frac{1}{p_{2}}\andf\frac{1}{r}=\frac{1}{r_{1}}+\frac{1}{r_{2}}\cdotp$$ 
 This still holds for couples $(1,1)$ and $(\infty,\infty)$ with the convention $L_{1,1}=L_{1}$ and $L_{\infty,\infty}=L_{\infty}.$
\item For any $\alpha>0$ and nonnegative measurable function $f,$ we have  $\norm{f^{\alpha}}_{L_{p,r}}=\norm{f}^{\alpha}_{L_{p\alpha,r\alpha}}$.
\item For any  $k>0$, we have $\norm{x^{-k}1_{\R_+}}_{L_{1/k,\infty}}=1.$ 
\end{enumerate}
\end{proposition}
Next, let us  state a few classical properties of Besov spaces.
\begin{proposition}[Besov embedding]\label{p:A1} There holds: \begin{enumerate}
\item  For any $(p,q)$ in $[1,\infty]^{2}$ such that $p\leq q,$ we have
 $$\B^{d/p-d/q}_{p,1}(\R^{d})\hookrightarrow L_{q}(\R^{d}).$$
\item  Let $1\leq p_{1}\leq p_{2}\leq \infty$ and $1\leq r_{1}\leq r_{2}\leq \infty.$ Then, for any real number $s$, 
$$\B^{s}_{p_{1},r_{1}}(\R^{d})\hookrightarrow \B^{s-d(\frac{1}{p_{1}}-\frac{1}{p_{2}})}_{p_{2},r_{2}}(\R^{d}).
%\andf B^{s}_{p_{1},r_{1}}(\R^{d})\hookrightarrow %B^{s-d(\frac{1}{p_{1}}-\frac{1}{p_{2}})}_{p_{2},r_{2}}(\R^{d}). 
$$
\end{enumerate}
\end{proposition}
The interpolation theory in  Besov spaces played  an important role in our  paper.  
Below are listed some results that we used (see details in \cite[Prop. 2.22]{BCD} or in 
 \cite[chapter 2.4.2]{HT}). 
\begin{proposition}[Interpolation]\label{interpolation}
A constant $C$ exists that  satisfies the following properties. If $s_{1}$ and $s_{2}$ are real numbers such that $s_{1}<s_{2}$
and $\theta\in ]0,1[,$ then we have, for any $(p,r)\in [1,\infty]^{2}$ and any tempered distribution $u$ satisfying \eqref{eq:lf}, 
$$\norm{u}_{\B^{\theta s_{1}+(1-\theta)s_{2}}_{p,r}(\R^{d})}\leq \norm{u}^{\theta}_{\B^{s_{1}}_{p,r}}\norm{u}^{1-\theta}_{\B^{s_{2}}_{p,r}(\R^{d})} $$
and, for some constant $C$ depending only on $\theta$ and $s_2-s_1,$ 
$$\norm{u}_{\B^{\theta s_{1}+(1-\theta)s_{2}}_{p,1}(\R^{d})}\leq
C\norm{u}^{\theta}_{\B^{s_{1}}_{p,\infty}(\R^{d})}\norm{u}^{1-\theta}_{\B^{s_{2}}_{p,\infty}(\R^{d})}. $$
Furthermore,  we have  for all $s\in(0,1)$ and $(p,q)\in[1,\infty]^2:$
$$\B^{s}_{p,q}(\R^d)=\bigl(L_{p}(\R^d);\W^{1}_{p}(\R^d)\bigr)_{s,q}.$$
\end{proposition}

%\begin{proposition}
%Let $-\infty <s_0,s_1<\infty,$ $s_0\neq s_1,$ $1<p<\infty,$ $1\leq q_0,q_1,q\leq \infty,$ and $0<\theta<1.$ If %$s=(1-\theta)s_0+\theta s_1,$ then 
%$$(\B^{s_0}_{p,q_0}(\R^d),\B^{s_1}_{p,q_1}(\R^d))_{\theta,q}=\B^{s}_{p,q}(\R^{d})
%\quad \text{and} \quad (B^{s_0}_{p,q_0}(\R^d),B^{s_1}_{p,q_1}(\R^d))_{\theta,q}=B^{s}_{p,q}(\R^d)
%$$where $s=(1-\theta)s_0+\theta s_1.$\end{proposition}
%The following proposition  will be frequently used (see \cite[Cor. 5.5]{BCD}).
%\begin{proposition}
%Let $(s_{1},s_{2})\in(-d/2,d/2)^{2},$ a constant $C$ exists such that if %$s_{1}+s_{2}$ is positive, then we have 
%$$\norm{uv}_{\B^{s_{1}+s_{2}-d/2}_{2,1}(\R^{d})}\leq %C\norm{u}_{\H^{s_{1}}(\R^{d})}\norm{v}_{\H^{s_{2}}(\R^{d})}.$$
%\end{proposition}

The following proposition %in the spirit of \cite[Theorem 2.1]{RD},
has been used several  times.  
\begin{proposition}\label{prop:for existence}
Let $1\leq q<\infty,$ $1\leq p<r\leq\infty$ and $\theta\in(0,1)$ such that
\begin{equation}\label{eq:q}
 \frac{1}{r}+\frac{1}{d}-\frac{2\theta}{dq}=\frac{1}{p}\cdotp
\end{equation}
 Then, there exists $C$ so that the following inequality holds true
 $$\norm{\nabla u}_{L_r(\R^d)}\leq C\norm{\nabla^2 u}^{\theta}_{L_p(\R^d)}\norm{u}^{1-\theta}_{\B^{2-2/q}_{p,\infty}(\R^d)}\cdotp$$
 \end{proposition}
\begin{proof}
Proposition \ref{interpolation} tells us in particular that
$$\|u\|_{\dot B^{2-\frac{2\theta}q}_{p,1}}\lesssim 
\|u\|_{\dot B^2_{p,\infty}}^{1-\theta} 
\|u\|_{\dot B^{2-\frac2q}_{p,\infty}}^{\theta}.$$
It is obvious that 
$$\|u\|_{\dot B^2_{p,\infty}}\lesssim \|\nabla^2 u\|_{L_p}$$
and, according to Proposition \ref{p:A1} and to the definition 
of $\theta,$ we have
$$\dot B^{2-\frac{2\theta}q}_{p,1}(\R^d)\hookrightarrow
\dot B^{2+\frac dr-\frac dp-\frac{2\theta}q}_{r,1}(\R^d)
=\dot B^1_{r,1}(\R^d).$$
As $\dot B^1_{r,1}(\R^d)\hookrightarrow \dot W^1_r(\R^d),$
we get the desired inequality. 
%For all $N\in\Z,$ one may write:
%$$\norm{\nabla u}_{L_q}\leq \sum_{j\leq N}\norm{\dot{\Delta}_j \nabla %u}_{L_q}+\sum_{j>N}\norm{\dot{\Delta}_j \nabla u}_{L_q}.$$
%On the one hand, from Bernstein inequality and \eqref{eq:q}, one gets
%\begin{equation*}\begin{aligned}
%\sum_{j\leq N}\norm{\dot{\Delta}_j \nabla u}_{L_q}&\lesssim \sum_{j\leq N} %(2^j)^{d(\frac{1}{p}-\frac{1}{q})}\norm{\dot{\Delta}_j \nabla u}_{L_p}\\
%&\lesssim  \sum_{j\leq N}  %2^{\frac{j}{r}}(2^{j(1-\frac{2}{r})}\norm{\dot{\Delta}_j \nabla u}_{L_p})\\
%&\lesssim 2^{\frac{N}{r}}\norm{\nabla u}_{\B^{1-\frac{2}{r}}_{p,\infty}}.
%\end{aligned}\end{equation*}
%On the other hand, we have
%$$\sum_{j>N}\norm{\dot{\Delta}_j \nabla u}_{L_q} \lesssim \sum_{j>N} %2^{-\frac{j}{r}}\norm{\dot{\Delta}_j \nabla^{2} u}_{L_p}\lesssim %2^{-\frac{N}{r}}\norm{\nabla^{2} u}_{\B^{0}_{p,\infty}}. $$
%Hence, choosing $N$ such that $$
%2^{2N}\norm{\nabla u}_{\B^{1-\frac{2}{r}}_{p,\infty}}\approx %\norm{\nabla^{2} u}_{\B^{0}_{p,\infty}}^r$$
%gives the result. 
 \end{proof}
 
The following result  that is an easy 
 adaptation of \cite[Prop. 2.1]{DM2} played a key role in  Sections \ref{section2} and \ref{section3}. 
\begin{proposition}\label{propregularity}
Let $1<p,q< \infty$ and $1\leq r\leq \infty.$ Then, for any $u_{0}\in \B^{2-2/q}_{p,r}(\R^{d})$ with $\div u_0=0,$ and any $f\in L_{q,r}(0,T;L_{p}(\R^{d})),$ the  Stokes system
\eqref{eq:stokes}
has a unique solution $(u,\nabla P)$ with $\nabla P\in L_{q,r}(0,T;L_p(\R^d))$
and\footnote{Only weak continuity holds if $r=\infty.$}  
$u$ in the space  $$\W^{2,1}_{p,(q,r)}((0,T)\times \R^{d}):=\bigl\{u\in \mathcal{C}([0,T];\B^{2-2/q}_{p,r}( \R^{d})):u_{t}, \nabla^{2}u\in L_{q,r}(0,T;L_{p}( \R^{d})) \bigr\}\cdotp$$
Furthermore, there exists a constant
$C$ \emph{independent of $T$} such that
\begin{multline}\label{eq:maxreg1}
\mu^{1-1/q}\norm{u}_{L_{\infty}(0,T;\B^{2-2/q}_{p,r}(\R^{d}))}+\norm{u_{t}, \mu\nabla^{2}u,\nabla P}_{L_{q,r}(0,T;L_{p}(\R^{d}))}
\\
\leq C\bigl(\mu^{1-1/q}\norm{u_{0}}_{\B^{2-2/q}_{p,r}(\R^{d})}+\norm{f}_{L_{q,r}(0,T;L_{p}(\R^{d}))}\bigr)\cdotp\end{multline}
Let $\wt s>q$ be such that 
$$\frac1q-\frac1{\wt s}\leq \frac12\andf \frac d{2p}+\frac1q-\frac1{\wt s}>\frac12,$$
 and define $\wt m\geq p$ by the relation
$$\frac{d}{2\wt m}+\frac{1}{\wt s}=\frac{d}{2p}+\frac{1}{q}-\frac{1}{2}\cdotp$$
Then, the following inequality holds true:  
\begin{multline}\label{eq:maxreg2}\mu^{1+\frac{1}{\wt s}-\frac{1}{q}}\norm{\nabla u}_{L_{\wt s,r}(0,T;L_{\wt m}(\R^{d}))}\\\leq C(\mu^{1-1/q}\norm{u}_{L_{\infty}(0,T;\B^{2-2/q}_{p,r}(\R^{d}))}+\norm{u_{t}, \mu\nabla^{2}u}_{L_{q,r}(0,T;L_{p}(\R^{d}))}).
\end{multline}
Finally,  if $2/q+d/p>2,$ then for all $s\in(q,\infty)$ and $m\in(p,\infty)$ such that 
$$\frac{d}{2m}+\frac{1}{s}=\frac{d}{2p}+\frac{1}{q}-1,$$ it holds that 
\begin{multline}\label{eq:maxreg3}\mu^{1+\frac{1}{s}-\frac{1}{q}}\norm{u}_{L_{s,r}(0,T;L_{m}(\R^{d}))}\\\leq C\bigl(\mu^{1-1/q}\norm{u}_{L_{\infty}(0,T;\B^{2-2/q}_{p,r}(\R^{d}))}+\norm{u_{t}, \mu\nabla^{2}u}_{L_{q,r}(0,T;L_{p}(\R^{d}))}\bigr)\cdotp
\end{multline}
\end{proposition}
\begin{proof} 

Let $\p$ and $\q$ be the Helmholtz projectors defined in \eqref{eq:PQ}.
As $u=\p u,$ we have 
$$u_t-\mu\Delta u=\p f,\qquad u|_{t=0}= u_0.$$
Hence applying  \cite[Prop 2.1]{DM2} and using that  $\p$ is continuous on $L_{q,r}(0,T;L_{p}( \R^{d}))$
gives \eqref{eq:maxreg1} and \eqref{eq:maxreg3} for $u.$
Since  $\nabla P=\q f,$ and $\q$ is also  continuous on $L_{q,r}(0,T;L_{p}( \R^{d})),$
 $\nabla P$ satisfies \eqref{eq:maxreg1} too. 
 \smallbreak
 In order to prove \eqref{eq:maxreg2}, take $q_0$ and $q_1$ such that $1<q_{0}<q<q_{1}<\infty$ and  $2/q=1/q_{0}+1/q_{1}.$
From the mixed derivative theorem we have for all  $\gamma \in (0,1)$ and $i=0,1,$
$$
\W^{2,1}_{p,q_{i}}((0,T)\times \R^{d}):=
\W^{2,1}_{p,(q_{i},q_{i})}((0,T)\times \R^{d})\hookrightarrow \W^{\gamma}_{q_{i}}(0,T;\W^{2-2\gamma}_{p}(\R^{d})).$$
Let   $\gamma:=1/q-1/\wt s$  (so that $d/\wt m=d/p+2\gamma-1$). 
 As $\gamma\in (0,\frac{1}{2}]$ and $1-2\gamma<d/p,$  one can use the Sobolev embedding
 \begin{equation}\label{eq:embed1}\W^{\gamma}_{q_{i}}(0,T;\W^{2-2\gamma}_{p}(\R^{d}))\hookrightarrow L_{\wt s_{i}}(0,T;\W^{1}_{\wt m}(\R^{d}))\with \frac{1}{\wt s_{i}}=\frac{1}{q_{i}}-\gamma.\end{equation}
 In the proof of  \cite[Prop. 2.1]{DM2}, it is pointed out that 
 $$\W^{2,1}_{p,(q,r)}((0,T)\times \R^{d})=\bigl(\W^{2,1}_{p,q_0}((0,T)\times \R^{d});\W^{2,1}_{p,q_1}((0,T)\times \R^{d})\bigr)_{\frac12,r}.$$ 
 Consequently, the embeddings \eqref{eq:embed1} with $i=0$ and $i=1$ imply that
  \begin{equation}\label{eq:embed2}
 \W^{2,1}_{p,(q,r)}((0,T)\times \R^{d})\hookrightarrow \bigl( L_{\wt s_{0}}(0,T;\W^{1}_{\wt m}(\R^{d})); L_{\wt s_{1}}(0,T;\W^{1}_{\wt m}(\R^{d}))\bigr)_{\frac 12,r}.\end{equation}
  Note that our definition of $\gamma,$ $\wt s_0,$ $\wt s_1,$ $q_0$ and $q_1$ ensures that 
  $$\frac{1}{2}\left(\frac{1}{\wt s_{0}}+\frac{1}{\wt s_{1}}\right)=\frac{1}{2}\left(\frac{1}{q_{0}}+\frac{1}{q_{1}}\right)-\gamma=\frac{1}{\wt s}\cdotp$$
Hence the real interpolation space in the right of \eqref{eq:embed2} is  nothing but
$L_{q,r}(0,T;\W^1_{\wt m}(\R^{d})),$ which completes the proof. 
  \end{proof}

The usual product is continuous in many Besov spaces (see e.g. \cite{AH, DR2003, RS1996}).
We here present a result that played  a key role  in the proof of uniqueness in dimension two.
In order to prove it, we need to introduce following so-called Bony's decomposition
(see \cite{Bony}):
%\begin{enumerate}
%\item In homogeneous case, $$uv=\dot T_u v+\dot T_v u+\dot R(u,v)$$
%with 
%$$\dot T_u v\triangleq\sum_{j\in \Z} \dot S_{j-1}u\dot \Delta_j v  \andf  \dot R(u,v)\triangleq\sum_{j\in \Z}\sum_{\abs{k-j}\leq 1} \dot \Delta_j u\dot\Delta_k v.$$\item In nonhomogeneous case, 
$$uv= T_u v+ T_v u+ R(u,v)$$
with 
$$ T_u v\triangleq\sum_{j\geq1}  S_{j-1}u \Delta_j v  \andf   R(u,v)\triangleq\sum_{j\geq-1}\sum_{\abs{k-j}\leq 1}  \Delta_j u \Delta_k v.$$
%\end{enumerate}
Above, we used the notation $\Delta_j:=\dot\Delta_j$ for $j\geq0,$ $\Delta_{-1}:=\dot S_0,$
$\Delta_j=0$ if $j\leq-2$ and $S_j:=\sum_{j'\leq j-1} \Delta_j.$
\smallbreak
The above operators $T$  and $R$ are called paraproduct and remainder, respectively.
Their general properties of continuity may be found in  \cite{BCD,BL, HT}.
The last inequality  is new to the best of our knowledge.  
\begin{proposition}\label{prop:h-1}
Let $2\leq p\leq \infty$ and $1\leq r_1,r_2\leq \infty$
satisfy $\frac{1}{r_1}+\frac{1}{r_2}=1.$
Then, the following inequality holds true:
%\begin{equation}\label{reminder2}
%\|\dot R(u,v)\|_{\B^{-\frac{d}{p}}_{p,\infty}(\R^d)}\lesssim \|u\|_{\B^{\frac %dp}_{p,r_1}(\R^d)}\|v\|_{\B^{-\frac dp}_{p,r_2}(\R^d)},
%\end{equation}
%Similarly, for the nonhomogeneous remainder,  we have
\begin{equation}\label{reminder1}
\| R(u,v)\|_{B^{-\frac{d}{p}}_{p,\infty}(\R^d)}\lesssim \|u\|_{B^{\frac dp}_{p,r_1}(\R^d)}\|v\|_{B^{-\frac dp}_{p,r_2}(\R^d)}.
\end{equation}
%\item Let $2\leq p\leq \infty$ and $1\leq r_1,r_2\leq \infty$
%satisfy $\frac{1}{r_1}+\frac{1}{r_2}=1,$ in addition $0<\ep<1.$ Then we have 
%\begin{equation}\label{reminder3}
%\|\dot R(u,v)\|_{\B^{-\frac{d}{p}+\ep}_{p,1}(\R^d)}\lesssim %\|u\|_{\B^{\ep}_{p,r_1}(\R^d)}\|v\|_{\B^{0}_{p,r_2}(\R^d)}\cdotp
%\end{equation}
%\item For all $s\in\R,$ if  $1\leq p_1,p_2,r_1,r_2\leq \infty$
%satisfy $\frac{1}{p_1}+\frac{1}{p_2}=1$ and $\frac{1}{r_1}+\frac{1}{r_2}=1$ then, we have
%\begin{equation}\label{eq:reml1}
%\|\dot R(u,v)\|_{L_1(\R^d)}\lesssim \|u\|_{\B^{s}_{p_1,r_1}(\R^d)}\|v\|_{\B^{-s}_{p_2,r_2}(\R^d)}.
%\end{equation}
In $\R^2$, it holds that
\begin{equation}\label{eq:product1}
    \|uv\|_{H^{-1}(\R^2)}\lesssim \log^{\frac 12}\Bigl(1+\frac{\|v\|_{L_2(\R^2)}}{\|v\|_{H^{-1}(\R^2)}}\Bigr)\|u\|_{H^1(\R^2)\cap L_\infty(\R^2)}\|v\|_{H^{-1}(\R^2)}.
\end{equation}
%\item Moreover, if $v$ have compactly support then, we have
%\begin{equation}\label{eq:product2}
 %   \|uv\|_{\H^{-1}(\R^2)}\lesssim \log^{\frac 12}\biggl(1+\frac{\|v\|_{L_2}}{\|v\|_{\dot H^{-1}}}\biggr)\|u\|_{\H^{\frac 12}(\R^2)\cap \H^1(\R^2)\cap L_\infty(\R^2)}\|v\|_{\H^{-1}(\R^2)},
%\end{equation} and
%$$\|uv\|_{\H^{-1}(\R^2)}\lesssim \log^{\frac 12} \biggl(e+\frac{\|u\|_{L_2}}{\|u\|_{\H^1}}\biggr)\log^{\frac %12}\biggl(e+\frac{\|v\|_{L_2}}{\|v\|_{\dot H^{-1}}}\biggr)\|u\|_{\H^1(\R^2)\cap L_\infty(\R^2)}\|v\|_{\H^{-1}(\R^2)}.$$
\end{proposition}
\begin{proof}
To prove the first statement, we use that, by definition of the homogeneous remainder operator
$$R(u,v)=\sum_{j\geq-1}\wt \Delta_{j} u\Delta_j v\with
\wt \Delta_{j}\triangleq \Delta_{j-1}+\Delta_{j}
+\Delta_{j+1}.$$
Hence, owing to the support properties of the dyadic partition of unity, there exists an integer $N_0$ such that
\begin{equation}\label{eq:decompose}
 \Delta_k  R(u,v)=\sum_{j\geq k-N_0} \Delta_k
(\Delta_j u\wt\Delta_j v)
=\sum_{\nu\leq N_0} \dot \Delta_k(\wt\Delta_{k-\nu} u\Delta_{k-\nu} v) \cdotp
                        \end{equation}
As $2\leq p\leq \infty,$ thanks to Bernstein's inequality, we have
$$\|\Delta_k  R(u,v)\|_{L_p(\R^d)}\leq 2^{k\frac{d}{p}}\|\Delta_k R(u,v)\|_{L_{p/2}(\R^d)}\cdotp$$
Therefore, using convolution inequalities  and \eqref{eq:decompose}, we discover that
$$\begin{aligned}
2^{-k\frac dp}\|\Delta_k R(u,v)\|_{L_p(\R^d)}
         &\leq C\sum_{\nu \leq N_0}\|\wt\Delta_{k-\nu} u\Delta_{k-\nu} v\|_{L_{p/2}(\R^d)}\\
         &\leq C\sum_{\nu\leq N_0} 2^{\frac{(k-\nu)d}{p}}\|\wt\Delta_{k-\nu} u\|_{L_p(\R^d)} 2^{-\frac{(k-\nu)d}{p}}\|\Delta_{k-\nu} v \|_{L_p(\R^d)},
\end{aligned}$$
which gives \eqref{reminder1}. 
%and $$\begin{aligned}
%2^{-k(\frac dp-\ep)}\|\dot \Delta_k \dot R(u,v)\|_{L_p(\R^d)}
%         &\leq C 2^{k\ep}\sum_{\nu \leq N_0}\|\wt\Delta_{k-\nu} u\dot \Delta_{k-\nu} v\|_{L_{p/2}(\R^d)}\\
%         &\leq C\sum_{\nu\leq N_0} 2^{\nu\ep}2^{(k-\nu)\ep}\|\wt\Delta_{k-\nu} u\|_{L_p(\R^d)} \|\dot %\Delta_{k-\nu} v \|_{L_p(\R^d)}\\
%\end{aligned}$$
%Hence, using  H\"older's inequality for series completes the proof of \eqref{reminder2} and \eqref{reminder3}.
%For the nonhomogeneous case, the proof is analogous is thus omitted.
%(we just need to take care of terms of the type $\Delta_j u\Delta_j v$ whose Fourier transform is only supported in balls $2^j B$). T
%For proving \eqref{reminder1}, we just have to write that
%$$\| R(u,v)\|_{L_1}\leq \sum_{j\in\Z}\|\wt\Delta_j u\,\dot\Delta_jv\|_{L_1}
%\leq \sum_{j\in\Z}\bigl(2^{js}\|\wt\Delta_j u\|_{L_{p_1}}\bigr)\bigl(2^{-js}\|\dot\Delta_jv\|_{L_{p_2}}\bigr),$$
%then use  H\"older's inequality for series.
\smallbreak
In order to prove \eqref{eq:product1}, we start from 
the following properties of continuity of the paraproduct operator (see the details in \cite[Chapter 2]{BCD}):
%and  embedding $B^{-1}_{2,2}(\R^2)\hookrightarrow B^{-2}_{\infty,\infty}(\R^2),$ it follows that
\begin{align}\| T_{u}v\|_{H^{-1}(\R^2)}&\lesssim \|u\|_{L_\infty(\R^2)}\|v\|_{H^{-1}(\R^2)},\label{eq:tuv1}\\
\label{eq:tuv2}\| T_{v}u\|_{H^{-1}(\R^2)}&\lesssim %\|v\|_{B^{-2}_{\infty,\infty}(\R^2)}\|u\|_{B^{1}_{2,2}(\R^2)}\nonumber\\
%&\lesssim 
\| v\|_{H^{-1}(\R^2)}\|u\|_{H^1(\R^2)}\cdotp
\end{align}
Next, we decompose  $R(u,v)$  into low and high frequencies, using 
\eqref{reminder1}, to get
$$\begin{aligned}
\|R(u,v)\|^2_{H^{-1}(\R^2)}&=\sum_{j\geq -1} 2^{-2j}\|\Delta_j R(u,v)\|^2_{L_2(\R^2)}\\
  &=\sum_{-1\leq j\leq N} 2^{-2j}\|\Delta_j R(u,v)\|^2_{L_2(\R^2)}+\sum_{j>N} 2^{-2j}\| \Delta_j R(u,v)\|^2_{L_2(\R^2)}\\
  &\lesssim N \|R(u,v)\|^2_{B^{-1}_{2,\infty}(\R^2)}+2^{-2N}\|R(u,v)\|^2_{B^0_{2,\infty}(\R^2)}\\
  &\lesssim N \|u\|^2_{H^1(\R^2)}\|v\|^2_{H^{-1}(\R^2)}+2^{-2N}\|u\|^2_{H^1(\R^2)}\|v\|^2_{B^{-1}_{\infty,2}(\R^2)}\cdotp
\end{aligned}$$ 
Then, choose $N$ be the closest integer larger than $\log_2 \Bigl(1+\frac{\|v\|_{L_2}}{\|v\|_{H^{-1}}}\Bigr)$, by virtue of  embedding $B^0_{2,2}(\R^2)\hookrightarrow B^{-1}_{\infty,2}(\R^2)$ we infer that
$$\|R(u,v)\|_{H^{-1}(\R^2)}\lesssim \log^{\frac 12}\Bigl(1+\frac{\|v\|_{L_2(\R^2)}}{\|v\|_{H^{-1}(\R^2)}}\Bigr)\|u\|_{H^1(\R^2)}\|v\|_{H^{-1}(\R^2)}\cdotp $$
Together with \eqref{eq:tuv1} and \eqref{eq:tuv2}, this completes the proof  of \eqref{eq:product1}. 
\end{proof}

 \begin{small}	 

\end{small}

\bigbreak\bigbreak
\noindent\textsc{Univ Paris Est Creteil, Univ Gustave Eiffel, CNRS, LAMA UMR8050, F-94010 Creteil, France}
\par\nopagebreak
E-mail addresses: raphael.danchin@u-pec.fr,
shan.wang@u-pec.fr


\begin{thebibliography}{9}
 \bibitem{AH} H. Abidi:  Équation de Navier-Stokes avec densité et viscosité variables dans l'espace critique. \emph{Rev. Mat. Iberoam}, 
 {\bf 23} (2007), no. 2, 537--586. 
 \bibitem{AG} H. Abidi and G. Gui: Global Well-posedness for the 2-D inhomogeneous incompressible Navier-Stokes System with large initial data in critical Spaces, \emph{Archiv. Rat. Mech. Anal.}, {\bf 242}, (2021), 1533--1570.
 
 \bibitem{AP} H. Abidi and M. Paicu: Existence globale pour un fluide inhomogène, \emph{Ann. Inst. Fourier}, {\bf  57} (2007), no. 3, 883--917. 

  \bibitem{BCD} H.  Bahouri,  J.-Y.  Chemin and   R.  Danchin:  \emph{Fourier  Analysis  and  Nonlinear  Partial  Differential  Equations,}  Grundlehren  der  Mathematischen  Wissenschaften,  vol. {\bf 343},  Springer-Verlag,  Berlin,  Heidelberg,  2011.
   \bibitem{BL} J. Bergh and J. L\"{o}fstrom: \emph{Interpolation Spaces, an Introduction,} Springer-Verlag, Berlin (1976).
    \bibitem{Bony} J.-M. Bony: 
 Calcul symbolique et propagation des singularit\'es pour
les \'equations aux d\'eriv\'ees partielles non lin\'eaires,
\emph{Ann.  Scientifiques de l'\'Ecole Normale Sup.} {\bf 14}(4), 209--246 (1981).  
 \bibitem{DZW} D. Chen, Z. Zhang and W. Zhao: Fujita-Kato theorem for the 3-D inhomogeneous Navier-Stokes equations, 
\emph{J. Differ. Equ.} {\bf 261} (2016) 738--761.
  \bibitem{CK} Y. Cho and H. Kim: Unique solvability for the density-dependent Navier-Stokes equations, {\em Nonlinear Anal.}, {\bf 59}(4), (2004), 465--489. 
    \bibitem{DR2003} R. Danchin: Density-dependent incompressible viscous fluids in critical spaces. 
    \emph{Proc. Roy. Soc. Edinburgh Sect. A},  {\bf 133} (2003), no. 6, 1311--1334.
    \bibitem{DR2004}  R. Danchin: Local and global well-posedness results for flows of inhomogeneous viscous fluids. \emph{Adv. Differential Equations}, {\bf 9} (2004), no. 3-4, 353--386.
    \bibitem{D-cours} R.Danchin: Fourier analysis methods for Partial Differential Equations  (2005), downloadable on
    \texttt{https://perso.math.u-pem.fr/danchin.raphael/cours/courschine.pdf}
    \bibitem{RD2016} R. Danchin: Fourier Analysis Methods for the Compressible Navier-Stokes Equations. In: Giga Y, Novotny A. (eds) Handbook of Mathematical Analysis in Mechanics of Viscous Fluids. Springer, Cham (2016).
      \bibitem{RD} R. Danchin: The inhomogeneous incompressible Navier-Stokes equations with discontinuous density: three diferent approaches, Analysis of PDE-Morningside Lecture notes.
      \bibitem{DFM} R. Danchin, F. Fanelli and M. Paicu: A well-posedness result for viscous compressible fluids with
only bounded density, \emph{Analysis and PDEs}, {\bf 13} (2020), 275--316.
   
  % \bibitem{DM4} R. Danchin, P.B. Mucha, Incompressible flows with piecewise constant density. Arch. Ration. Mech. Anal. {\bf 207} (2013), no. 3, 991--1023. 
  \bibitem{DM1} R. Danchin and P.B. Mucha:  The incompressible Navier-Stokes equations in vacuum, {\em Communications on Pure and Applied Mathematics}, {\bf 52} (2019), 1351--1385.
   
   \bibitem{DM-ripped}  R. Danchin and P.B. Mucha: Compressible Navier-Stokes equations with ripped density,
   {\em Communications on Pure and Applied Mathematics}, to appear. 
   
   \bibitem{DM2} R. Danchin, P.B. Mucha and P. Tolksdorf:   Lorentz spaces in action on pressureless systems arising from models of collective behavior. \emph{J. Evol. Equ.}, {\bf 21} (2021), 3103--3127.
   
\bibitem{DL1989} R.J. DiPerna and P.-L. Lions: Ordinary differential equations, transport theory and Sobolev spaces,\emph{ Invent. Math.}, {\bf 98} (1989), 511--547.
   
   \bibitem{FK} H. Fujita and T. Kato: On the Navier-Stokes initial value problem I, 
   {\it Archive for Rational Mechanics
 and Analysis}, {\bf 16},   (1964), 269--315.
   
   \bibitem{LG} L. Grafakos: \emph{Classical and Modern Fourier Analysis,} Prentice Hall, 2006.
   \bibitem{HPZ2013} J. Huang, M. Paicu and P. Zhang, Global well-posedness to 
   incompressible inhomogeneous fluid system
with bounded density and non-Lipschitz velocity,   \emph{Arch. Ration. Mech. Anal.}, {\bf  209} (2013) 631–682.
%  \bibitem{OV}  O.  Lady\v{z}enskaja and  V Solonnikov, The unique solvability of an initial-boundary value problem for viscous incompressible inhomogeneous fluids, in: \emph{Boundary Value Problems of Mathematical
%Physics, and Related Questions of the Theory of Functions,} vol. 8, Zap. Nau\v{c}n. Sem. Leningrad.
%Otdel. Mat. Inst. Steklov. (LOMI) {\bf 52} (1975) 52–109, 218–219 (Russian).
\bibitem{AVK} A.V. Kazhikov: Solvability of the initial-boundary value problem for the equations 
of the motion of an inhomogeneous viscous incompressible fluid, 
\emph{(Russian) Dokl. Akad. Nauk SSSR } {\bf 216} (1974), 1008--1010.

% \bibitem{OV1}   O.A. Ladyženskaja and  V.A. Solonnikov:  The unique solvability of an initial-boundary value %problem for viscous incompressible inhomogeneous fluids. (Russian) Boundary value problems of mathematical %physics, and related questions of the theory of functions, \emph{8. Zap. Naučn. Sem. Leningrad. Otdel. Mat. %Inst. Steklov. (LOMI)} {\bf 52} (1975), 52–109, 218--219.
 
 \bibitem{OV2} O.A. Ladyženskaja and V.A. Solonnikov: Unique solvability of an initial and boundary value problem for viscous incompressible inhomogeneous fluids. \emph{J. Sov. Math.},  {\bf 9} (1978), no. 5, 697--749. 
  \bibitem{PL} P.-L. Lions, \emph{Mathematical Topics in Fluid Mechanics. Vol. 1. Incompressible Models}, Oxford Lecture Series in Mathematics and its Applications,  vol. 3, Oxford University Press, 1996. 
 %\bibitem{PBM} P.B. Mucha: Transport equation: extension of classical results for $\div b$  in $BMO.$  \emph{J. %Differential Equations,} {\bf 249} (2010), no.8, 1871–1883.
 %\bibitem {MR} P.B. Mucha and W. Rusin: Zygmund spaces, inviscid limit and uniqueness of Euler flows, \emph{Comm. Math. Phys.} {\bf 280} (3) (2008) 831–841.
   \bibitem{PZZ} M.  Paicu,  P.  Zhang and  Z.  Zhang:  Global  unique  solvability  of  inhomogeneous  Navier-Stokes  equations with  bounded  density,  \emph{Commun.  Partial  Differ.  Equ.}, {\bf 38}  (2013)  1208--1234.
   \bibitem{RS1996} T. Runst and W. Sickel: Sobolev spaces of fractional order, Nemytskij operators, and non- linear partial di erential equations. \emph{De Gruyter Series in Nonlinear Analysis and Applications.}, vol.  3 (Berlin: Walter de Gruyter, 1996).
   \bibitem{JS} J. Simon: Nonhomogeneous viscous incompressible fluids: existence of velocity, density, and pressure, 
   \emph{SIAM J. Math. Anal.},  {\bf 21} (1990) 1093–1117.
   \bibitem{HT}  H. Triebel: \emph{Interpolation theory, function spaces, differential operators.} North-Holland Mathematical Library, 18. North-Holland Publishing Co., Amsterdam-New York, 1978.
  \bibitem{Xu} H. Xu: Maximal $L_1$  regularity for solutions to inhomogeneous incompressible Navier-Stokes equations, arXiv:2103.11513.
   \bibitem{Zhang19} P. Zhang: Global Fujita-Kato solution of 3-D inhomogeneous incompressible Navier-Stokes system, \emph{Adv. Math.}, {\bf  363} (2020), 107007, 43 pp.
    
\end{thebibliography}
\end{document}